\documentclass[12pt]{amsart}
\usepackage{amssymb,amsthm,amsmath}
\usepackage[pdftex,hypertexnames=false,linkcolor=blue,citecolor=red]{hyperref}
\usepackage[url=false,giveninits=true,isbn=false,doi=false,
maxbibnames=99,sortcites,style=numeric,
sorting=nyt,maxnames=4,maxalphanames=4]{biblatex}

\usepackage{enumerate}
\usepackage{fancybox}
\usepackage{listings}
\usepackage{multirow}
\usepackage{graphicx}
\usepackage{graphics}
\usepackage{color}
\usepackage{fullpage,multirow}
\usepackage{longtable}
\usepackage{tikz}
\usepackage{mathtools}

\newcommand{\beqn}{\begin{eqnarray*}}
\newcommand{\eeqn}{\end{eqnarray*}}

\usepackage[capitalise]{cleveref}
\newtheorem{theorem}{Theorem}[section]
\newtheorem{proposition}[theorem]{Proposition}
\newtheorem{lemma}[theorem]{Lemma}
\newtheorem{hypothesis}[theorem]{Hypothesis}
\newtheorem{corollary}[theorem]{Corollary}
\theoremstyle{remark}
\newtheorem{remark}[theorem]{Remark}
\theoremstyle{notation}
\newtheorem{notation}[theorem]{Notation}
\theoremstyle{definition}
\newtheorem{example}[theorem]{Example}
\newtheorem{definition}[theorem]{Definition}

\newtheorem{theoremA}{Theorem}

\newtheorem{theoremB}{Theorem}

\crefname{hypothesis}{Hypothesis}{Hypotheses}
\crefname{hypothesis}{Hypothesis}{Hypotheses}
\crefname{theoremA}{Theorem}{Theorems}
\crefname{theoremB}{Theorem}{Theorems}
\crefname{corollary}{Corollary}{Corollaries}
\crefname{conjecture}{Conjecture}{Conjectures}
\crefname{proposition}{Proposition}{Propositions}
\crefname{qu}{Question}{Questions}
\crefname{lemma}{Lemma}{Lemmas}
\crefname{definition}{Definition}{Definitions}
\crefname{theorem}{Theorem}{Theorems}
\crefname{example}{Example}{Examples}
\crefname{Notation}{Notation}{}

\AddToHook{env/lemma/begin}{\crefalias{theorem}{lemma}}
\AddToHook{env/proposition/begin}{\crefalias{theorem}{proposition}}
\AddToHook{env/definition/begin}{\crefalias{theorem}{definition}}
\AddToHook{env/corollary/begin}{\crefalias{theorem}{corollary}}
\AddToHook{env/hypothesis/begin}{\crefalias{theorem}{hypothesis}}
\AddToHook{env/notation/begin}{\crefalias{theorem}{notation}}
\AddToHook{env/remark/begin}{\crefalias{theorem}{remark}}

\newcommand{\SU}{\mathrm{SU}}

\newcommand{\M}{\mathrm{M}}
\newcommand{\G}{\mathrm{G}}
\newcommand{\D}{\mathrm{D}}

\newcommand{\Aut}{\mathrm{Aut}}
\newcommand{\Hom}{\mathrm{Hom}}

\newcommand{\Syl}{\mathrm{Syl}}\newcommand{\syl}{\mathrm{Syl}}

\newcommand{\GammaL}{\Gamma\mathrm{L}}
\newcommand{\GF}{\mathrm{GF}}
\newcommand{\GL}{\mathrm{GL}}
\newcommand{\Sp}{\mathrm{Sp}}

\newcommand{\PSU}{\mathrm{PSU}}
\newcommand{\SL}{\mathrm{SL}}
\newcommand{\GU}{\mathrm{GU}}
\newcommand{\0}{\emptyset}

\newcommand{\PSL}{\mathrm{PSL}}
\newcommand{\PSp}{\mathrm{PSp}}

\newcommand{\SO}{\mathrm{SO}}

\newcommand{\Sym}{\mathrm{Sym}}
\newcommand{\Alt}{\mathrm{Alt}}

\newcommand{\SDih}{\mathrm{SDih}}

\newcommand{\Stab}{\mathrm{Stab}}

\renewbibmacro{in:}{}
\DeclareFieldFormat[article]{title}{#1}
 
\DeclarePairedDelimiter{\gen}{\langle}{\rangle}

\title[Quasisimple groups with a proper subgroup with the same vector orbits]
{Quasisimple groups with a proper subgroup having the same vector orbits in characteristic $2$}

\author{Chris Parker}
\address{Chris Parker\\
School of Mathematics\\
University of Birmingham\\
Edgbaston\\
Birmingham B15 2TT\\
United Kingdom}
\email{c.w.parker@bham.ac.uk}

\author{B. G. Rodrigues}
\address{B. G. Rodrigues\\Department of Mathematics and Applied Mathematics\\ University of Pretoria\\ Private Bag X20\\ Hatfield, Pretoria 0028, South Africa}
\email{bernardo.rodrigues@up.ac.za}

\begin{document}
\maketitle

\begin{abstract}
Let  $p$ be a prime, $G$ be a finite group, $H$ a proper subgroup of $G$ and $V$ a finite dimensional $\GF(p)G$-module. The triple $(G,H,V)$ is immutable if and only if $G$ and $H$ have the same orbits on the vectors of $V$. We determine the immutable triples for $G$ a quasisimple group, $H$ a subgroup of $G$ and $V$ a $\GF(2)G$-module.
\end{abstract}

\section{Introduction}
Group theorists often say, ``you know a group by its actions.''
In this paper, we examine this principle in the setting of quasisimple groups acting on vector spaces over $\GF(2)$ where  typically exceptions occur.
Specifically, we ask whether the orbit structure of a quasisimple group $G$ on a $\GF(2)G$-module $V$ provides a certificate of largeness, distinguishing $G$ from its proper subgroups.

\begin{definition}\label{def:imm}
Let $p$ be a prime,  $ G $ be a group, $ H <G $, and $ V $ be a $ \GF(p)G $-module.
We say that the triple $ (G,H,V) $ is \emph{immutable} if and only if $ G $ and $ H $ have exactly the same orbits on $ V $.
\end{definition}

Equivalently (viewing $V$ as a right $ \GF(p)G $-module), $ (G,H,V) $ is immutable if and only if $ vG = vH $ for all $ v \in V $.  There are two extreme cases of immutability. The first possibility is the transitive case: if $ H $ acts transitively on the non-zero vectors $V ^\# $ of $V$, then $ (G,H,V) $ is immutable. The second possibility  is the trivial case in which $ V $ is centralized by $ G $ in which case $G$ cannot be distinguished from any of  its subgroups.

The  subgroups of $\GL(V)$ which act transitively on $X^\#$ have been completely determined by Huppert and Hering \cite{Huppert, Hering} (see also \cite[Section~7]{HBIII} and the enumeration in \cite[Appendix~1]{LiebeckAffine}). Hence, whenever we find that $ H $ acts transitively on $ V ^\#$, we may regard $(G,H,V)$ as determined by this classification.

Our main results show that, apart from the transitive and trivial cases, immutable triples are rare when $G$ is quasisimple and $p=2$. In particular, the orbit structure of a $\GF(2)G$-module almost always distinguishes $G$ from its proper subgroups.  In \cite{Liebeck}, Liebeck explores a property related to immutability  when $G$ is a classical group acting on its natural module and $H$ is a subgroup of $G$. These examples appear itemised as parts (i) , (ii) and (iii) of \cref{Main theorem}. Our work extends \cite[Theorem]{Liebeck} as it does not specify the module $V$ involved other than requiring that it is irreducible.  \cref{thm:Theorem B} removes the irreducibility requirement and determines all immutable triples with $H$ maximal and $V$ a faithful $\GF(2)G$-module.

In the statement of \cref{Main theorem}, the \emph{natural module} for a classical group defined in characteristic $ 2 $ refers to that group’s usual module, regarded as a $ \GF(2)G $-module.
Similarly, we use the standard terminology for \emph{spin} and \emph{half-spin} modules for orthogonal and symplectic groups in characteristic $ 2 $.
For $G= \Omega_8^+(2^a) $, there are $ \GF(2)G $-modules of dimension $ 8a $ that are quasiequivalent under the triality automorphism. Hence any irreducible module of dimension $ 8a $ may be viewed as the natural module, though when two distinct such modules occur together, we must distinguish between them.

\begin{theoremA}\label{Main theorem}
Suppose that $ G $ is a quasisimple group, $ H $ a maximal subgroup of $ G $, and $ V $ a non-trivial irreducible $ \GF(2)G $-module.
Then $ (G,H,V) $ is immutable if and only if either $ H $ is transitive on the non-zero vectors of $ V$, or one of the following holds:
\begin{enumerate}
  \item $ G \cong \Omega_{2n}^-(2^a) $ with $ n \ge 3 $ odd,
  $ H \cong \GU_{n}(2^a) $,
  $ V $ is the natural $ \GF(2)\Omega_{2n}^-(2^a) $-module, and $ V|_H $ is the natural module.

  \item $ G \cong \Omega_{2n}^+(2^a) $ with $ n \ge 4 $ even,
  $ H \cong \GU_{n}(2^a).2 $,
  $ V $ is the natural $ \GF(2)\Omega_{2n}^+(2^a) $-module, and $ V|_H $ is the natural module
  (two conjugacy classes of such $ H $).

  \item $ G \cong \Omega_8^+(2^a) $,
  $ H \cong \Sp_6(2^a) $,
  $ V $ is the natural $ \Omega_8^+(2^a) $-module, and $ V|_H $ is the spin module
  (two conjugacy classes of such $ H $).

  \item $ G \cong \Sp_{6}(2^a) $,
  $ H \cong \SO_6^-(2^a) $,
  $ V $ is the spin module for $\GF(2) \Sp_{6}(2^a) $, and $ V|_H $ is the spin module.

  \item $ G \cong \Omega_8^+(2) $,
  $ H \cong \Alt(9) $,
  $ V $ is the natural $ \GF(2)\Omega_8^+(2) $-module, and $ V|_H $ is the spin module
  (two conjugacy classes of such $ H $).

  \item $ G \cong \Alt(9) $,
  $ H \cong \mathrm{P}\Gamma\mathrm{L}_2(8) $,
  $ V $ is the fully deleted permutation module for $ \GF(2)G $
  (two conjugacy classes of such $ H $).
\end{enumerate}
\end{theoremA}

\begin{remark}
In \cref{Main theorem} (iv) with $G \cong \Sp_6(2^a)$, we have   $H \cong  \SU_4(2^a){:}2$, and so $V|_H$ can be regarded as the natural module for $H$. Furthermore, regarding $G$ as a subgroup of $\Omega_8^+(q)$ via the embedding in \cref{Main theorem} (iii),   the normaliser of $H$ is containing in the maximal subgroup $\GU_4(q){:2}$.
\end{remark}

In   \cref{Stabilizers of orbit reps} below,  for the triples $(G,H,V)$ listed in \cref{Main theorem} (i) to (vi), we give the $G$-orbit lengths  and stabilizers of orbit representatives. Here $P_1$ and $P_3$ are parabolic subgroups of $G$ and are described just before \cref{prop:char2 list}.

\begin{table}[h!]
\begin{tabular}{cccc}
\hline
Case & Orbit length&Stabiliser&Multiplicity\\
   \hline
 (i) & $(2^{an}+1)(2^{a(n-1)}-1)$&$O^{2'}(P_1)$&1\\& $ 2^{a(n-1)}(2^{an}+1)  $&$\Sp_{2n-2}(2^a)$&$2^a-1$\\
\hline
(ii)   & $(2^{an}-1)(2^{a(n-1)}+1)$&$O^{2'}(P_1)$&$1$\\& $2^{a(n-1)}(2^{an}-1)$&$\Sp_{2n-2}(2^a)$&$2^a-1$\\
\hline
(iii)&$(2^{4a}-1)(2^{3a}+1)$&$O^{2'}(P_1)$&$1$\\& $2^{3a}(2^{4a}-1)$&$\Sp_{6}(2^a)$&$2^a-1$\\
\hline
(iv)&$(2^{4a}-1)(2^{3a}+1)$&$O^{2'}(P_3)$&1\\
&$2^{3a}(2^{4a}-1)$&$\mathrm{G}_2(2^a)$&$2^a-1$\\
\hline
(v)&$135$&$2^6{:}\Omega_6^+(2)$&1\\
&$120$&$\Sp_6(2)$&$1$\\
\hline
(vi)& 9&$\Alt(8)$&$1$\\
&36&$(\Sym(2) \times \Sym(7))\cap \Alt(9)$&1\\
&84&$(\Sym(3)\times \Sym(6))\cap \Alt(9)$&1\\
&126&$(\Sym(4)\times \Sym(5))\cap \Alt(9)$&1\\
\hline
\end{tabular}\vspace{3mm}
\caption{Orbits, stabilisers and multiplicities for quasisimple groups having a maximal subgroup $H$ and an irreducible module $V$ such that $(G,H,V)$ is immutable.}\label{Stabilizers of orbit reps}
\end{table}

An important consequence of \cref{Main theorem} is that the module $V$ is known. This allows us to determine the immutable triples $(G,K,V)$ where $K$ is an arbitrary subgroup of $G$.

 \begin{corollary}\label{Cor:Main Result} Suppose that $G$ is quasisimple, $K <G$, $V$ is a non-trivial irreducible $\GF(2)G$-module and $(G,K,V)$ is immutable.  Then, either, $K$ acts transitively on the non-zero vectors of $V$, or one of the following holds:
 \begin{enumerate}
 \item $G \cong \Omega_{2n}^-(2^a)$ with $n \ge 3$ odd,  $V$ is the natural $\GF(2)G$-module, $K \cong \SU_{n}(2^a)$ and $(G,N_G(K),V)$ is as in \cref{Main theorem} (i).

  \item $ G \cong \Omega_{2n}^+(2^a) $ with $ n \ge 6 $ even, $V$ is the natural $\GF(2)G$-module,
   $K \cong \SU_{n}(2^a)$ and $(G,N_G(K),V)$ is as in \cref{Main theorem} (ii).
 \item $G \cong \Omega_{2n}^\epsilon(2^a)$ with $n \ge 5$ and $\epsilon 1 =(-1)^n$,  $V$ is the natural $\GF(2)G$-module, $K \cong \SU_{n}(2^a)$ and $(G,N_G(K),V)$ is as in \cref{Main theorem} (i) or (ii).
 \item $G \cong \Sp_6(2^a)$, $V$ is the spin $\GF(2)G$-module, $K \cong \SU_4(2^a)$, and  $(G,N_G(K),V)$ is as in \cref{Main theorem} (iv).
 \item $G= \Omega_8^+(2^a)$,  $a\ge 2$, $V$ is the natural $\GF(2)G$-module and either
 \begin{enumerate} \item $K \cong \Sp_6(2^a)$ is maximal in $G$  (two conjugacy classes);
 \item $K \le L$ with $K \cong \SU_4(2^a)$  and  $L \cong \Sp_6(2^a)$ as in (a), and $N_G(K)$ is maximal in $G$ with $N_G(K) \cong \GU_4(2^a){:}2$ and $|N_G(K):N_L(K)|=2^a+1$ (two conjugacy classes).
 \end{enumerate}
 \item $G \cong \Omega_8^+(2)$,  $V$ is the natural $\GF(2)G$-module and either
 \begin{enumerate}
 \item $K \cong \Sp_6(2)$ is maximal in $G$  (two conjugacy  classes);
 \item $K \cong \SU_4(2)$, $K \le L \cong\Sp_6(2)$, $N_G(K) \cong \GU_4(2){:}2$ is maximal in $G$ with $|N_G(K): N_L(K)|=3$  (two conjugacy classes); or
 \item $K \cong \Alt(9)$ is maximal in $G$  (two conjugacy classes).
\end{enumerate}
\item $G \cong \Alt(9)$, $V$ is the fully deleted permutation $\GF(2)G$-module,  and $K \cong \PSL_2(8)$ or $\mathrm P \Gamma \mathrm L_2(8)$ (two conjugacy classes).
\end{enumerate}
  \end{corollary}

\begin{corollary}\label{cor1} Suppose that $G$ is quasisimple, $K <G$, $V$ is an irreducible $\GF(2)G$-module and $(G,K,V)$ is immutable. Let $\mathbb K=\mathrm{End}_G(V)$. If $G$ has at least $3$ orbits on the $\mathbb KG$ $1$-spaces of $V$, then $G \cong \Alt(9)$,  $\mathbb K=\GF(2)$, $V$ is the fully deleted $\mathbb KG$-module, $K \cong \mathrm{P}\Gamma \mathrm L_2(8)$, or $\PSL_2(8)$, $\mathbb K=\GF(2)$ and $G$ has $4$ orbits on the $\mathbb KG$ $1$-spaces of $V$.
\end{corollary}

\cref{Cor:Main Result,cor1} are proved at the end of \cref{class char 2}.

\begin{remark}
Unlike in the non-transitive cases of \cref{Cor:Main Result}, in the transitive case it is possible to have chains of quasisimple subgroups
\[
 H_1 < \dots < H_k < G
\]
of arbitrary length such that each $(G,H_i,V)$ is immutable.
For example,
\[
\G_2(2^4) < \Sp_6(2^4) < \Sp_{12}(2^2) < \Sp_{24}(2) < \SL_{24}(2)
\]
has the property that $(\SL_{24}(2), \G_2(2^4), V)$ is immutable, where $V$ is the natural module of dimension 24 for $\SL_{24}(2)$.
\end{remark}

Suppose that $V$ is a $\GF(p)G$-module and $0 \neq U \le V$.
We say that $U$ is a \emph{centralized direct summand} of $V$ if $U \le C_V(G)$ and there exists a submodule $W \le V$ such that $V = U \oplus W$.
By \cref{lem:trivial factors}, if $U$ is a centralized direct summand and $V = U \oplus W$, then $(G,H,V)$ is immutable if and only if $(G,H,W)$ is immutable.
Thus, we may disregard centralized direct summands when studying immutability of general modules.

Before stating our second main theorem, we briefly describe a certain module for $\Sp_{2n}(2^a)$.
Let $G = \Sp_{2n}(2^a) \cong \Omega_{2n+1}(2^a)$ with $n \ge 1$ and $(n,a) \ne (1,1)$.
Via this isomorphism, $G$ naturally embeds into $\Omega_{2n+2}^+(2^a)$ . Let $W$ be the natural $\Omega_{2n+2}^+(2^a)$-module and define $\mathcal W(\Sp_{2n}(2^a))$ to be $W|_G$ regarded as a $\GF(2)G$-module (see \cref{lem:symplectic fact} for more details). Recall that for an $\mathbb FG$-module $V$ with submodules $U \ge W$, we refer to $U/W$ as a \emph{section} of $V$.

\begin{theoremB}\label{thm:Theorem B}
Suppose that $G$ is a quasisimple group, $H$ is a maximal subgroup of $G$, and $V$ is a $\GF(2)G$-module.
Assume that $V$ has no centralized direct summand and that $(G,H,V)$ is immutable.
Then either $V$ is irreducible and $(G,H,V)$ is as in \cref{Main theorem}, or one of the following holds:
\begin{enumerate}
  \item $G = \Sp_{2n}(2^a)$, $n \ge 3$,  $V$ is a faithful section of $\mathcal W(\Sp_{2n}(2^a))$ and either
   \begin{enumerate}\item$H \cong \Sp_{2n/b}(2^{ab}){:}b$ for some odd prime $b$ dividing $n$, or \item $2n = 6$ and $H \cong \mathrm{G}_2(2^a)$.
    \end{enumerate}
  \item $G = \Sp_{2n}(2^a)$, $n \ge 2$ even, $(n,2^a) \neq (2,2)$, $H = \Sp_n(2^{2a}){:}2$, and for $\mathcal W = \mathcal W(\Sp_{2n}(2^a))$ and each $B \le C_{\mathcal W}(G)$, there exists a uniquely determined submodule
  $[ \mathcal W, G ] \le B^\dagger \le \mathcal W$ with $\dim B^\dagger / B = a(2n+1)$ such that the faithful section $V = A/B$ satisfies $(G,H,V)$ is immutable if and only if $[V,G] \le A \le B^\dagger$.
  \item $G = \SL_4(2) \cong \Alt(8)$, $H = \Alt(7)$, and $V$ is a direct sum of at most three irreducible $4$-dimensional modules for $G$.
  \item $G = \Alt(6) = \Sp_4(2)'$, $H = \SL_2(4) \cong \Alt(5)$, and $V$ is the $5$-dimensional $\Omega_5(2)$-module or its dual restricted to $G$.
\end{enumerate}
\end{theoremB}

For completeness, and for use in the proof of \cref{thm:Theorem B} in \cref{prop:trans cases}, we list the immutable triples $(G,H,V)$ with $G$ quasisimple, and $H$ a maximal subgroup of $G$ such that $H$ acts transitively on the non-zero vectors of $V$.

 Our work finds its motivation in a construction of linear codes developed by Knapp and the second author    in \cite{KnappRodrigues}. Here one begins with a group $H$ acting faithfully on a $\GF(p)H$-module and constructs an \emph{evolution code} $C_{V,\Omega}$ from an invariant subset $\Omega$ of non-zero vectors of $V$.   The evolution code $C_{V,\Omega}$ has generator matrix of size $\dim V \times |\Omega|$, whose rows span a subspace of $\GF(p)^{|\Omega|}$ isomorphic to $V^*$ (the dual of $V$). A legitimate question to ask is under what circumstances is $H$ close to the full automorphism group of $C_{V,\Omega}$?
This project was initiated to understand when this fails.  When $(G,H,V)$ is immutable, then following the construction with the group $H$ and module $V|_H$, one sees that all invariant subsets   $\Omega$ of $V|_H$ yield codes which have $G$ in their automorphism group and so in these instances the evolution code $\mathcal C=C_{V|_H,\Omega}$ does not have  $H$ as a large normal subgroup of $\Aut(\mathcal C)$.

\begin{example} Let $(G,H,V)$ be the immutable triple in part (vi) of \cref{Main theorem}. Then by Table~\ref{Stabilizers of orbit reps}, $H\cong \mathrm P\Gamma\mathrm L_2(8)$ has orbits of length $9$, $36$, $84$ and $126$ on the non-zero vectors of $V$. For such an orbit $\Omega$,  the evolution  codes for $H$ have length $|\Omega|$ and dimension $\dim V=8$.  Each of these codes is invariant under $\Sym(9)$.
\end{example}

We have organized our paper as follows. In \cref{sec:Immutability}, we explore consequences of a triple $(G,H,V)$ being immutable.
An immediate consequence is that $G$ admits a factorization (\cref{lem:basic1}), making the fundamental contribution of Liebeck, Praeger and Saxl \cite{LiebeckPraegerSaxl} a key ingredient in our proofs, with \cite{GGS} providing additional examples not covered there.

\cref{sec:Factorisations} presents factorization results not contained in \cite{LiebeckPraegerSaxl,GGS}.
In \cref{sec:alternating groups}, we treat groups with $F^*(G)$ an alternating group, obtaining results for symmetric groups in addition to our headline results for the alternating group.  The immutable triple from \cref{Main theorem}(vi) appears in \cref{lem:full deleted}.
\cref{sec:sporadics} considers the candidates for \cref{Main theorem} when $G/Z(G)$ is a sporadic simple group.
Applying \cite{LiebeckPraegerSaxl} yields a short list (\cref{lem:sporadic1}) of groups and subgroups to consider, which we then handle using elementary {\sc Magma} \cite{Magma} calculations, including orbit counting and some cohomology computations.
\cref{sec:exceptional} treats exceptional groups, here the non-existence of immutable triples follows quickly from \cite{LiebeckPraegerSaxl}.

Another immediate consequence of immutability is that $G$ cannot have a regular orbit on $V ^\#$ (\cref{lem:reg orb compendium}(i)).
In \cref{sec:odd} where we consider classical groups in odd characteristic, we apply \cite{Lee1} to obtain a 24-point list of possibilities for $G$, 19 of which require further consideration after accounting for results in previous sections.
The remainder are resolved using {\sc Magma} calculations.

The core of the proof of \cref{Main theorem} appears in \cref{class char 2}, where we consider classical groups defined in characteristic $2$.
Here, for an irreducible, non-trivial $V$ with $(G,H,V)$ immutable, we exploit the fact that some   parabolic $P$ of $G$ leaves invariant a $1$-space of $V$, so $O^{p'}(P)$ centralizes a vector.
This leads to a factorization $G = PH$, and using \cite{LiebeckPraegerSaxl}, we list all potential factorizations (\cref{prop:char2 list}).
We first examine the exceptional “odd ball” cases, revealing the non-family example in \cref{Main theorem}(v).
For the generic cases, \cref{lem:parabolic factorisation} shows that some vector is centralized by $O^{p'}(P)$ for a maximal parabolic $P$, and in \cref{the new list} we determine the candidates for $V$ using \cref{prop:char2 list} and the Steinberg tensor product theorem.
The remainder of the complicated \cref{class char 2} completes the analysis and at the end of the section we prove \cref{Main theorem} and its corollaries.

For the general case where $V$ may be reducible, two elementary lemmas (\cref{lem: reduct to irred,lem:trivial factors}) become crucial: the first shows that every composition factor $W$ of $V$ yields an immutable triple $(G,H,W)$, and the second allows us to disregard centralized direct summands for if $T$ is a trivial $G$-module then $(G,H,V)$ is immutable if and only if $(G,H,V\oplus T)$ is immutable.
To prove \cref{thm:Theorem B}, we first assume that $Z(G)$ has odd order.
In \cref{sec:ThmB1}, we consider the cases where $G$ is not a symplectic group.
Aside from the special cases $G \cong \Alt(6)$ or $\Alt(8) \cong \SL_4(2)$, our goal is to show that if $V$ has no centralized direct summand and $(G,H,V)$ is immutable, then $V$ is irreducible.
The cases $G \cong \SL_4(2)$ and $G \cong \Alt(6)$ yield the examples listed in \cref{thm:Theorem B} (iii) and (iv).

In \cref{sec:ThmBSymp}, we treat the cases where $G \cong \Sp_{2n}(2^a)$.
Here the module $\mathcal W(\Sp_{2n}(2^a))$ shows its head, and the examples in \cref{thm:Theorem B}(i) and (ii) are realized.
Finally, in \cref{sec:schur}, we show that there are no examples with $V$ faithful when $Z(G)$ has even order, using the cohomology functionality built in to {\sc Magma}. Finally, at the end of
\cref{sec:schur} we draw the threads of the proof of  \cref{thm:Theorem B} together.

 Our notation is either explained or is self-explanatory and follows \cite{AschbacherFG}.   We mention specifically that for a vector space $V$, we use $V^\#$ to denote the non-zero elements of $V$.
Finally, we mention that most of the ``small" cases are handled using {\sc Magma}. The code to perform these calculations is contained in ancillary  file.


\section{Immutability}\label{sec:Immutability}

In this section we catalogue various consequence of  $(G,H,V)$ being immutable. Let $p$ be the characteristic of the field over which $V$ is defined. Initially, we present results for an arbitrary prime $p$. We start with an elementary result which says that immutability leads to factorisations of $G$.

\begin{lemma}\label{lem:basic1}
Suppose that $V$ is a $\GF(p)G$-module and $H < G$. Then, for $v \in V$,  $vG=vH$ if and only if $G= C_G(v)H$.
\end{lemma}

\begin{proof}
This is just the Frattini Argument: Let $g \in G$, then $vg \in vG=vH$ and so there exists $h \in H$ such that $vh= vg$.  Hence $gh^{-1}\in C_G(v)$ and this yields $g= gh^{-1}h \in C_G(v) H$. Conversely, if  $G= C_G(v)H$, then surely $vH= v{C_G(v)H}=vG$.
\end{proof}

For a $\GF(p)G$-module $V$ and $v \in V$, we say that an orbit $vG$ is \emph{regular} provided $C_G(v)=1$.

\begin{lemma}\label{lem:reg orb compendium}
Suppose that $(G,H,V)$ is immutable. Then the following hold:
\begin{enumerate}
\item $G$ does not have a regular orbit on $V$.
\item There exists $v \in V^\#$ with $|G:C_G(v)|$ prime to $p$ and $G=C_G(v)H$.
\item  For $G > L > H$, $(G,H,V)$ is immutable if and only if $(L,H,V|_L)$ and $(G,L,V)$ are immutable.
\end{enumerate}\end{lemma}

\begin{proof} (i) Suppose that $v \in V$ is a representative of a regular orbit.  Then, by \cref{lem:basic1}, $G=C_G(v)H=H$, which is a contradiction.

(ii) As $V$ is a $\GF(p)G$-module, there exists  $v \in V^\#$ which is centralized by a Sylow $p$-subgroup of $G$. Hence $|G:C_G(v)| $ is prime to $p$.

(iii) Let $v \in V$. Then $vH\subseteq vL\subseteq vG$. If $(G,H,V)$  is immutable, then  $vG=vH$ and so   $vH= vL=vG$ for all $v \in V$.  Hence $(L,H,V|_L)$ and $(G,L,V)$ are immutable. If, on the other hand, $(L,H,V|_L)$ and $(G,L,V)$ are immutable, then, for all $v in V$, $vH= vL$ and $vL=vG$. Hence $vH=vG$ and $(G,H,V)$  is immutable.
\end{proof}

\begin{lemma}\label{lem:numberorbits} Suppose that $(G,H,V)$ is immutable.  Then $G$ has at least $|V|/|H|$ orbits on $V$.
\end{lemma}

\begin{proof}
Let $v_1, \dots, v_n$ be representatives of the $G$-orbits on $V$. Then \begin{eqnarray*} |V|&=&\sum_{i=1}^n|v_iG|
 = \sum_{i=1}^n|v_iH|\\
&=&\sum_{i=1}^n|H|/|C_H(v_i)| \le n|H|.\end{eqnarray*} The result follows.
\end{proof}

\begin{lemma}\label{same point orbit}
Suppose that $(G,H,V)$ is immutable and $v \in V^\#$. Let $\mathbb K= \mathrm{End}_G(V)$ and regard $V$ as a $\mathbb KG$-module.  Then $\gen{v}_{\mathbb K}G=\gen{v}_{\mathbb K}H$. In particular, $G$ and $H$ have the same orbits on the $1$-dimensional $\mathbb K$-subspaces of $V$.
\end{lemma}
\begin{proof} Since $G$ and $\mathbb K$ commute, $G$ acts on the set of $1$-dimensional $\mathbb K$-subspaces of $V$. Furthermore, $\gen{v}_{\mathbb K}g=\gen{vg}_{\mathbb K}$ for all $v \in V^\#$ and $g \in G$. Let $v \in V^\#$ and $\gen{w}_{\mathbb K}\in \gen{v}_{\mathbb K}G$. Then there exists $g \in G$ such that $\gen{vg}_{\mathbb K}=\gen{w}_{\mathbb K}$. Since $(G,H,V)$ is immutable, there exists $h \in H$ such that $vh=vg$. Thus $\gen{w}_{\mathbb K}=\gen{vg}_{\mathbb K}=\gen{vh}_{\mathbb K}=\gen{v}_{\mathbb K}h$ and so $\gen{w}_{\mathbb K} \in \gen{v}_{\mathbb K}H$. It follows that $\gen{v}_{\mathbb K}H= \gen{v}_{\mathbb K}G$, as claimed.
\end{proof}

\begin{lemma}\label{lem:irred reduct}
Suppose that $(G,H,V)$ is immutable. If $V$ is  irreducible, then $V|_H$ is irreducible. In particular, $O_p(H)=1$ or $O_p(H)\le C_G(V)$.
\end{lemma}

\begin{proof}
Suppose that $W>0$ is a submodule of $V|_H$ and let $w \in W^\#$.  Then $  w  G=   w  H \subseteq W$ and so $w g \in W$ for all $g \in G$.  Hence $W$ is a submodule of $V$.  Since $V$ is irreducible, $W=V$. Hence $V|_H$ is irreducible. Suppose that $O_p(H) \not \le C_G(V)$. Since $C_V(O_p(H)) \ne 0$ is $H$-invariant and $V$ is faithful, we conclude that $O_p(H)=1$.
\end{proof}

For $r$ a prime, the \emph{$r$-rank} of a group $X$, $m_r(X)$, is the minimal number of generators of a maximal order elementary abelian $r$-subgroup of $X$.

\begin{lemma}\label{lem:prank} Suppose that $(G,H,V)$ is immutable and  $r\ne p$ is a prime.  If  $V$ is irreducible and $m_r(G) \ge 2$, then there exists a factorisation $G= HL$ of $G$ with $m_r(L) \ge m_r(G)-1$.
\end{lemma}

\begin{proof}
Suppose that $A \le G$ is elementary abelian of $r$-rank $k$. Then there exists a maximal subgroup $B$ of $A$ such that $C_V(B) \ne 0$. Let $v \in C_V(B)^\#$. Then $G= C_G(v)H$ and $B \le C_G(v)$. Taking $L=C_G(v)$ we have $B\le L$ and $G=HL$ from \cref{lem:basic1}.  This proves the claim.
\end{proof}

\begin{lemma}\label{lem:Dual same number orbits}  The triple  $(G,H,V)$ is immutable if and only if  $(G,H,V^*)$ is immutable.
\end{lemma}\begin{proof}
By Burnside's orbit counting lemma, $G$ has $\frac{1}{|G|}\sum_{g\in G} C_V(g)= \frac{1}{|G|}\sum_{g\in G} C_{V^*}(g)$ orbits on $V$ and $V^*$. Since a similar statement holds for $H$, we obtain have result.
\end{proof}

\begin{lemma}\label{unitary sbgrp} Suppose that $G \cong \SL_n(q)$ and $\sigma_0 \in \mathrm{Gal}(\GF(q)/\GF(p))$ has order $2$. Let $X=\{g\in G\mid (( g^{\sigma_0})^{-1})^T= g\}$ be the unitary subgroup of $G$ and $V$ be the natural $\GF(q)G$-module.  Then $V^*\cong V^{\sigma_0}$ as $\GF(q)X$-modules.
\end{lemma}

\begin{proof}
We may regard $V^*$ as the vector space $V$ with the action of an element $g \in G$ obtained by applying the inverse transpose of $g$.  Since for $g \in X$, $(g^{-1})^T=g^{\sigma_0}$, the claim follows.
\end{proof}

Suppose that $G$ is an   quasisimple group and $x \in G \setminus Z(G)$.  Then $\alpha (x)$ is the minimum number of conjugates of $x$ required to generate $G$ and $$\alpha(G)= \max\{\alpha(x)\mid x \in G\}.$$

\begin{lemma}\label{lem:dimbound}
  Let  $G$ be  a      quasisimple  group  and  $V$ a  faithful  irreducible $\GF(p)G$-module.
 If $G$  has no regular orbits on V, then $|V| \le |G|^{\alpha(G)}$.
 \end{lemma}

\begin{proof}
This is \cite[Lemma 3.5]{FOS}.
\end{proof}

\begin{lemma}\label{SU2nhelp}
Suppose that $V$ is a $\GF(p)G$-module, $(G,H,V)$ is immutable and $r$ is a prime which divides $|G|$ but does not divide $|H|$. Let $S\in \syl_r(G)$ and assume that $|S|=r$.  Then $\dim C_V(S) \ge \log_p(|V|/|H|)$. In particular, if $\log_p(|V|/|H|)> \frac 12 \dim V$ and $C_V(G)=0$, then $G$ is not generated by two conjugates of $S$.
\end{lemma}

 \begin{proof}
 Since $ r$ does not divide $|H|$ and $(G,H,V)$ is immutable, no orbit of $G$ on $V$ has length divisible by $r$.  Hence $S$ fixes a representative of every $G$-orbit  on $V$.  The number of orbits of $G$ on $V$ is bounded below by $|V|/|H|$ by \cref{lem:numberorbits}. Hence $|C_V(S)| \ge |V|/|H|$, and this proves the first part of the claim.

 Suppose that $\log_p(|V|/|H|)> \frac 12 \dim V$. Then $\dim C_V(S)> \frac 12 \dim V$, and so $C_V(S) \cap C_V(S^g) \ne 0$ for all $g \in G$.  Since $C_V(G)=0$, $\langle S,S^g\rangle <G$ for all $g \in G$.
 \end{proof}

Recall that if $V$ is a $\GF(p)G$-module, $0\le U<W \le V$, then the $\GF(p)G$-module  $W/U$ is called a \emph{section} of $V$. The following two  lemmas are the key observations which allows us to prove \cref{thm:Theorem B} once the proof of \cref{Main theorem} is complete.

\begin{lemma}\label{lem: reduct to irred}   Suppose that $(G,H,V)$ is immutable and $W$ is a submodule of $V$. Then
 \begin{enumerate}
 \item  $(G ,H  ,W)$ is immutable; and
\item $(G,H  ,V/W)$ is immutable.
\end{enumerate}
In particular, if $U/W$ is a section of $V$, then $(G,H,U/W)$ is immutable.
\end{lemma}

\begin{proof}
For $w \in W$, we have $w G= w H$ and so (i) holds.

Let $v+ W \in V/W$.  Then $$(v+W) G= v G+W=v H+W=(v+W) H.$$ This demonstrates (ii). The final statement immediately follows.
\end{proof}

The following minor observation is used several times in the proof of \cref{thm:Theorem B}.

\begin{lemma}\label{lem:vec stabiliser}
Suppose that $G$ is a group and $V$ is a $\GF(p)G$-module.
Let $W$ be a proper  submodule of $V$ and $v+W$ be an element of $V/W$. Let $P$ be the stabilizer in $G$ of $v+W$ and assume that $v$ can be chosen in $C_V(P)$. Then, for $w \in W$, $C_G(v+w)=C_P(w)$.
\end{lemma}

\begin{proof} We  know  $C_G(v+w) \le C_G(v+W)= P$ and $C_P(v+w)= C_P(w)$.
\end{proof}

As mentioned in the introduction, our next result shows that we may ignore centralized direct summands when considering immutability.

\begin{lemma}\label{lem:trivial factors} Suppose that $G$ is a finite group, $V$ is a $\GF(p)G$-module, $H\le G$ and that $T$ is the trivial $\GF(p)G$-module. Then $(G,H,V)$ is immutable if and only if $(G,H,V\oplus T)$ is immutable. Furthermore, the number of $G$-orbits on $V\oplus T$ is $p$ times the number of $G$-orbits on $V$. \end{lemma}

\begin{proof} Set $W= V \oplus T$ and let $w \in W$. Then $w=v+t$ where $v \in V$ and $t\in T$. In particular, $C_G(w)= C_G(v)$. Therefore, $G= C_G(w)H$ if and only if $G=C_G(v)H$ and this is the claim. Moreover, if  $\{v_1, \dots, v_n\}$ are $G$-orbit representatives for the action of $G$ on $V$, then $\{t+v_i\mid t \in T, 1\le i\le n\}$ are representatives for the orbits of $G$ on $W$ and this proves the statement about the number of orbits.
\end{proof}

 \begin{lemma}\label{lem:min exactly 2}
 Assume that $G= O^p(G)$, $V$ is a $\GF(p)G$-module which   has no centralized direct summand.  If $V$ is not irreducible, then there exists a section $W$ of $V$ which   has no centralised direct summand and has exactly two composition factors.
 \end{lemma}

 \begin{proof}Suppose the claim is false and choose $V$ to be a counter example of minimal dimension.  If all composition factors of $V$ are trivial modules, then as $G=O^p(G)$, $V$ is a direct sum of trivial modules, a contradiction.

 Let $U$ be   submodule of $V$ minimal such that $U$ has a non-trivial composition factor. Then $U$ has no centralized direct summand by the minimal choice of $U$.  Hence either $U=V$ or $U$ is irreducible for otherwise we could select $W$ to be a section of $U$.

 In the first case, $V$ has a maximal submodule $M$ consisting just of trivial modules. Let $K$ be a maximal  submodule of $M$, then the minimal choice of $U$ implies that $V/K$ has no centralized direct summand and we take $W= V/K$, which is a contradiction.

 Hence $U$ is irreducible.  In particular, $U<V$.
Let $Y$ be a submodule of $V$ with  $Y/U$ irreducible. If $Y$ has no centralized direct summand, then we take $W=Y$ and obtain a contradiction. Therefore $Y =T \oplus U$ with $T$ a trivial $G$-module. By hypothesis $Y<V$. Now set $\overline V = V/T$. Then $\overline Y \ne \overline V$. Since $\dim \overline V < \dim V$, we are done unless $\overline V$ has a centralised direct summand. So $\overline V = C\oplus X$ where $C$ denotes the maximal centralized direct summand of $\overline V$. If $X$ is not irreducible, then we can choose $W$ to be a section of $X$ by induction, a contradiction. Hence $X$ is irreducible and $X \cong U$. Let $C^\dagger$ be the preimage of $C$. All the composition factors of $C^\dagger$ are trivial $G$-modules and so as $G=O^p(G)$, $C^\dagger$ is centralized by $G$ and $V= C^\dagger \oplus U$, contradicting $V$ having no centralized direct summands. This completes the proof of the lemma.
 \end{proof}

Our next result will be used when we apply the Steinberg tensor product theorem for Lie type groups defined in characteristic $2$.

\begin{lemma}\label{lem:no tensors}
Suppose that $n \ge 2$,  $\mathbb F $ is a field and $V$ is an $n$-dimensional $\mathbb FG$-module. Let $\Sigma= \mathrm{Gal}(\mathbb F/\mathbb F_0)$ where $\mathbb F_0$ is a subfield of $\mathbb F$.
 Assume that $\{v_1, v_2 \} \subset V$ is linearly independent,   $\ell \ge 2$ and  $$W= V^{\sigma_1}\otimes \dots  \otimes V^{\sigma_\ell}$$ where $\sigma_1,\dots, \sigma_\ell\in \Sigma$. Let $w_1= v_1 \otimes v_1\otimes \dots\otimes v_2$ and $w_2= v_1 \otimes v_1\otimes \dots\otimes v_2$ be elements of $W$.
Then $C_G(w_1)$ stabilizes $\langle v_1\rangle$ in $V$ and    $C_G(w_2)$ stabilizes $\langle v_1\rangle$ and $\langle v_2\rangle$ in $V$.
\end{lemma}

\begin{proof} For this calculation, we may and do identify $G$ with a subgroup of $\GL(V)$.
Extend $\{v_1, v_2 \}$ to a basis $\mathcal B=\{v_1, \dots, v_n\}$ of $V$ and write  $$w_2=\underbrace{v_1\otimes v_1\otimes \dots \otimes v_1\otimes v_2}_{\ell \text{ factors}}$$
Then, for $g\in G\le\GL(V)$, we express $g= (a_{ij})\in \GL_n(\mathbb F)$ with respect to the basis $\mathcal B$. Thus $v_1g= \sum_{j=1}^n a_{1j}v_j$ and $v_2g= \sum_{j=1}^n a_{2j}v_j$.
Therefore $$
w_2g= (\sum_{j=1}^n a_{1j}^{\sigma_1}v_j\otimes \sum_{j=1}^n a_{1j}^{\sigma_2}v_j\otimes \dots \otimes \sum_{j=1}^n a_{1j}^{\sigma_{\ell-1}}v_j\otimes \sum_{j=1}^n a_{2j}^{\sigma_\ell} v_j)$$
and so,   noting that the coefficient of $v_{i_1}\otimes v_{i_2}\otimes \dots \otimes v_{i_l}$  is $$a_{1i_1} ^{\sigma_1}a_{1i_2}^{\sigma_2}\dots a_{1i_{\ell-1}}^{\sigma_{\ell-1}} a_{2i_\ell}^{\sigma_\ell}$$ we obtain
 $wg=w$ if and only if $v_1g^{\sigma _i}=a_{11}^{\sigma _i }v_1$ for $1\le i \le \ell-1$ and  $v_2g^{\sigma_\ell}=
 a_{22}^{\sigma_\ell} v_2$ with
 $$a_{11} ^{\sigma_1}a_{11}^{\sigma_2}\dots a_{11}^{\sigma_{\ell-1}} a_{22}^{\sigma_\ell}=1$$ and all the other possible combinations are zero.
 Therefore,      the first two rows of $g$ are
$$\begin{pmatrix}a_{11}&0&0&\dots&0\\
0&a_{22}&0&\dots&0\end{pmatrix}.$$ In particular $C_G(w_2) \le \Stab_G(\langle v_1\rangle) \cap \Stab_G(\langle v_2\rangle)$ as claimed.

A similar but slightly easier proof shows that  $C_G(w_1)$ stabilizes $\langle v_1\rangle$.
\end{proof}

When examining modules defined over fields bigger than $\GF(p)$, the following lemma guides our methodology.

\begin{lemma}\label{lem:good vectors}
Suppose that $\mathbb F$  is a finite field of characteristic $p$, $G$ is a group, $H \le G$  and that $V$ is an irreducible $\mathbb FG$-module. Assume $L \le G$ fixes a non-zero vector in $V$ and that for every subgroup $M\le G$ with $L \le M$,  $G\ne  MH$. Let $U \le V$ be an irreducible $\GF(p)G$-submodule of $V$. Then $(G,H,U)$ is not immutable.
\end{lemma}

\begin{proof} We know that on restriction to $\GF(p)G$, $V$ is a direct sum of isomorphic irreducible $\GF(p)G$-modules $V= U_1\oplus\dots \oplus U_k$ for some $k \ge 1$. In particular, we have that the intersection of all the maximal $\GF(p)G$-submodules of $V$ is zero.

Let $v \in V$ be a non-zero vector fixed by $L$. Then there exists a maximal $\GF(p)G$-submodule $W$ of $V$ such that $v \not \in W$. Furthermore, $V/W\cong U$ as $\GF(p)G$-modules. Since $v$ is fixed by $L$, $v+W$ is fixed by $L$ in $V/W$. Assume that $(G,H,U)$ is immutable.
   Then   $G= C_G(v+W)H$ by \cref{lem:basic1}. Since $L \le C_G(v+W)$ this contradicts our assumption about factorisations with over-groups of $L$.
\end{proof}

\cref{class char 2} requires a deeper knowledge of the irreducible modules for groups of Lie type.
The important features that we deploy  are the following.  Let $p$ be a prime, $G$ be a finite quasisimple group of Lie type defined in characteristic $p$. We fix $S \in \syl_p(G)$ and $B=N_G(S)=TS$ a fixed Borel subgroup with Cartan subgroup $T$. Let $\overline{\mathbb L}$ be the algebraic closure of $\GF(p)$.  Then the irreducible  $\overline{\mathbb L} G$-modules are described by the theory of weights. We follow the notation developed in \cite[Sections 1.14 and 2.8]{GLS3}. In particular, we remark that the high weight vectors can be chosen in $C_V(S)$ and that $N_G(C_V(S)) \ge B$ is   a parabolic subgroup. When understanding the irreducible  $\overline{\mathbb L} G$-modules, the Steinberg tensor product theorem is of premium importance. For this, we require the group $\Sigma$ of the field automorphisms of $G$. Thus $\Sigma=\langle \sigma\rangle$ where $\sigma$ is the restriction to $G$ of the    standard Frobenius automorphism of $\overline{\mathbb L}$ defined by $x \mapsto x^p$. Suppose that $\Pi=\{\alpha_1, \dots, \alpha_n\}$ is a fundamental system of roots for $G$ and let $\{\omega_1, \dots, \omega_n\}$ be the corresponding set of fundamental dominant weights.  Thus $\langle \alpha_i,\omega_j\rangle =\delta_{ij}$.   The irreducible $\overline{\mathbb L} G$-representations are in one-to-one correspondence with the dominant weights in the $q$-restricted range. This means that the weight $\lambda$ satisfies $\lambda=\sum_{i=1}^n c_i \omega_i$ with $0\le c_i < q$.  A weight is basic, if in fact $0\le c_i <p$.

\begin{lemma}\label{high weights} Let $p$ be a prime and $q=p^e$. Suppose that $G$ is an untwisted universal quasisimple group of Lie type defined over $\GF(q)$.
Assume that  $S \in \Syl_p(G)$, $B=N_G(S)$, $\Pi$ is a set of fundamental roots for $G$ and  $V=V(\lambda)$ is an irreducible $\overline{\mathbb L}G$-module with $q$-restricted high weight $\lambda$.  Set $P= N_G(C_V(S))$. Then there exists $J \subseteq \Pi$ such that $P=BN_JB$ is a parabolic subgroup of $G$ and  setting $K=\Pi \setminus J$ we have $$\lambda= \sum_{k\in K}   (\sum_{\ell=0}^{e-1}d_{k,\ell}p^\ell)\omega_k=
 \sum_{\ell=0}^{e-1}p^\ell(\sum_{k\in K} d_{k,\ell}\omega_k)$$ where $0<\sum_{\ell=0}^{e-1}d_{k,\ell}p^\ell< q$,  $0\le d_{k,\ell}<p$  and $$V=    \otimes_{\ell=0}^{e-1} W_\ell^{\sigma^{ \ell}}$$ where $W_\ell= V\left(\sum_{k\in K} d_{k,\ell}\omega_k\right)$.
\end{lemma}

 \begin{proof} We are given that $V$ corresponds to the $q$-restricted weight $\lambda$ and we can write $V=V(\lambda)$. Since $B$ leaves $C_V(S)$ invariant, $P\ge B$ and so $P$ is a parabolic subgroup of $G$ by \cite[Theorem 8.3.2 ]{Carter}.    In particular, following the notation in \cite[Proposition 8.2.2]{Carter}  we have $P=BN_JB$ for some $J \subseteq \Pi$. Therefore, $N_J/T$ is the parabolic subgroup of $N/T$ which stabilizes the weight $\lambda$. Hence $\lambda\in \langle \omega_k\mid k \in K\rangle$. If $ \lambda\in \langle \omega_k\mid k \in K_0\rangle$ for some $K_0 \subset K$, then the corresponding reflections of $N/T$ in $K\setminus K_0$ would fix $\lambda$ and stabilize $C_V(S)$.  As they are not in $P$, this is impossible.  Hence $\lambda = \sum_{k\in K} e_{k} \omega_k$ with $0< e_{k}< q$ for $k \in K$.    We now rewrite $\lambda$ as a $p$-adic expansion of basic weights thus we take the $p$-adic expansion of each $e_{k}= \sum_{\ell=0}^{e-1}d_{k,\ell}p^\ell $ with $0\le d_{k,\ell}<p$. This yields
 $$\lambda= \sum_{k\in K}   (\sum_{\ell=0}^{e-1}d_{k,\ell}p^\ell)\omega_k=
 \sum_{\ell=0}^{e-1}p^\ell(\sum_{k\in K} d_{k,\ell}\omega_k).$$
Now observe that each $\sum_{k\in K} d_{k,\ell}\omega_k$ is a basic weight as $0\le d_{k,\ell}<p$. Applying the Steinberg tensor product theorem (see \cite[II.3.17 Corollary]{Jantzen}) yields $$V=V(\lambda) \cong \bigotimes_{\ell=0}^{e-1} V\left(\sum_{k\in K} d_{k,\ell}\omega_k\right)^{\sigma^{ \ell}}.$$
This concludes the proof.
\end{proof}

We now specialize to the special case where $P$ in \cref{high weights} is a maximal parabolic subgroup and $p=2$.

\begin{lemma}\label{high weight app} Suppose that $G$  is an untwisted universal quasisimple group of Lie type defined over $\GF(2^e)$. Let $V=V(\lambda)$ be an irreducible $\overline{\mathbb L}G$-module with $2^e$-restricted high weight $\lambda$. Assume that  $S \in \Syl_2(G)$  and $P= N_G(C_V(S))$ is a maximal parabolic subgroup of $G$. Then there exists a unique $1 \le j\le n$ and $\Theta \subseteq \Sigma$ such that setting $W= V(\omega_j)$ we have
$$V= \bigotimes_{\theta \in \Theta} W^\theta.$$
\end{lemma}
\begin{proof} We may assume that $P=BN_JB$ omits the $j$th fundamental reflection.  Thus $K=\Pi\setminus J=\{j\}$ in \cref{high weights} and
 $$ V=    \otimes_{\ell=0}^{a-1} W_\ell^{\sigma^{ \ell}}$$ where $W_\ell= V\left(\sum_{k\in K} d_{k,\ell}\omega_k\right)$. Since $K=\{j\}$ and $0\le d_{k,\ell}<p=2$, $$\sum_{k\in K} d_{k,\ell}\omega_k= d_{j,\ell}\omega_j=\begin{cases}\omega_j&
 d_{j,\ell}=1 \\0&d_{j,\ell}=0\end{cases}$$ and so $W_\ell=V(\omega_j)$ or is trivial for each $\ell$.
 Define $\Theta=\{\sigma^\ell \mid d_{j,\ell}=1\}$ and $W=V(\omega_j)$.  Then $\Theta \subseteq \Sigma$ and
$ {V}=\bigotimes_{\theta \in \Theta} W^{\theta},$
as claimed.\end{proof}

We shall  need the following two results when we consider unitary groups in \cref{class char 2}.

\begin{lemma}\label{SU2nbound}
Suppose that $m, n \in \mathbb Z$ are positive, $q=2^m$, and $q^{\left({2n}\atop n\right)}  \le 2^{2n} (q^{8n^3}) $. Then $n \le 6$.
\end{lemma}

\begin{proof} Assume that $n \ge 7$. We require $m\left({2n}\atop n\right) \le 2n+8mn^3$. As $n \ge 7$, this is equivalent to $m\le \frac {2n}{  \left({2n}\atop n\right)-8n^3  }$. We know $1\le m$.  So we calculate directly that $1\le \frac n{  \left({2n}\atop n\right)-4n^3  }$ fails for $7\le n \le 11$. Assume that $n \ge 12$.
Now we approximate \begin{eqnarray*}1&\le&\frac {2n}{  \left({2n}\atop n\right)-8n^3  }\le \frac {2n}{  \left({2n}\atop 7\right)-8n^3  }\\&=& \frac{2.7!.n}{2n(2n-1)(2n-2)(2n-3)(2n-4)(2n-5)(2n-6) -8.7!.n^3}\\
&=& \frac{7!}{(2n-1)(2n-2)(2n-3)(2n-4)(2n-5)(2n-6) -4.7!n^2}\\
&\le&\frac{7!}{n^2(n^4-4.7!)}
 < \frac{7!}{576.n^2}<1,\end{eqnarray*} which is a contradiction. Hence $n \le 6$.
\end{proof}

\begin{lemma}\label{lem:unitary dim bound} Suppose that $n \ge 2$, $G= \SU_{2n}(2^a)$ and $V$ is the exterior $n$th power of the natural module for $G$. Then $V $ can be realized over $\GF(2^a)$ and, if $n \ge 7$, then $G$ has a regular orbit on $V$.
\end{lemma}

\begin{proof} Set $q=2^{a}$. By construction $V$ can be realized over $\GF(q^2)$ and can be identified as coming from the fundamental weight $\omega_n$. Let $\sigma \in \Aut(\GF(q^2))$ have fixed field $\GF(q)$ and $\tau$ be the graph automorphism associated with the untwisted Dynkin diagram for $G$. Then $\tau$ permutes the fundamental weights as it does the nodes of the diagram.  In particular, $\tau(\omega_n)= \omega_n$ and so \cite[Proposition 5.4.2]{KleidmanLiebeck} yields the first claim.

Using the notation from \cref{lem:dimbound}, \cite[Theorem 4.2]{GuralnickSaxl} gives $\alpha(G)=2n$ and,    \cite[Corollary 16]{ParkerWilson} gives $|G| \le 2q^{4n^2}$. Thus if $G$ has no  regular orbits on $V$, \cref{lem:dimbound} implies
$$q^{\left({2n}\atop n\right)}  \le 2^{2n} (q^{8n^3})$$ and this only holds when $n \le 6$ by \cref{SU2nbound}.
\end{proof}

We close this section with two \emph{ad hoc} results which find their application in the proof of \cref{thm:Theorem B}.

\begin{lemma}\label{lem:End bound} Assume that $p$ is a prime, $P$ is a group, $Q= O_p(P)=F^*(P)$   and $V$ is a faithful $\GF(p)P$-module. Assume that $Q$ is a minimal normal subgroup of $P$,  and  $W $ is an irreducible submodule of $V$. If $V/W $ is centralized by $P$, then $Q \cong W $ as $\GF(p)P$-modules, and $V/C_V(Q)$ can be identified with a subspace of $ \mathrm{End}_{P/Q}(W)$.
\end{lemma}

\begin{proof} Since $W$ is an irreducible submodule, we   have $W \le C_V(Q)$.
 For $v \in V$, define the map \begin{eqnarray*}\phi_{ v}:Q &\rightarrow& W\\ q &\mapsto& [v,q].\end{eqnarray*}  Since   $P$ centralizes $V/W$, $[v,q]\in W$. Hence $\phi_{ v}$  is well-defined.

Let $q_1,q_2\in Q$ and $1\le e\le p$.  Then, as $q_1$ and $q_2$ commute and the image of $\phi_{  v }$ is contained in $W\le C_V(Q)$,  \begin{eqnarray*}\phi_v(q_1^eq_2)&=&[v,q_1^eq_2]= v(q_1^eq_2-1)= v(q_1^e-1)q_2 +v(q_2-1)\\&=&  v(q_1-1)(q_1^{e-1}+\dots+1)+[v,q_2]= e[v,q_1]+[v,q_2]= e\phi_v(q_1) +\phi_v(q_2) \end{eqnarray*} and so $\phi_{ v}$ is a linear transformation from $Q$ to $W$. Let $g\in P$ and $q \in Q$.  Then $  {v}g^{-1}=   v+w$ for some $w\in W$ and so $$\phi_{  v}(q^g)=[v,q^g] = [vg^{-1},q]g= [v+w,q]g=[v,q]g=\phi_v(q)g.$$ Hence $\phi_{  v}$ is a $\GF(p)P$-module homomorphism.

Since $V$ is faithful, we may choose $x \in V \setminus C_V(Q)$. Hence $\phi_x$ is non-trivial. As $Q$ and $W$ are irreducible,   $\phi_{  x}$ is an isomorphism. Now the map \begin{eqnarray*}
 \phi:   V &\rightarrow& \mathrm{Hom}_{P/Q}(Q,W)\\   v &\mapsto& \phi_{  v}\end{eqnarray*} is a linear transformation with kernel $C_V(Q)$, and so the assertion follows as  $\mathrm{Hom}_P(Q,W)\cong \mathrm{End}_{P/Q}(W)$.
\end{proof}

\begin{lemma}\label{lem:some cohomology}
Let  $p$ be a prime, $G$ a quasisimple group with $Z(G)$ of order $p$, and $V$ be  an absolutely irreducible $\GF(p)G$-module.  Set $U= \mathrm{End}_{\GF(p)}(V)$ and $W= \mathrm{End}_G(V)$.

Suppose that $G$ has a faithful $\GF(p)G$-module $X$ with $C_X(Z(G))= [X,Z(G)] \cong V.$
Then $\dim_{\GF(p)} \mathrm H^1(G,U) <  \dim_{\GF(p)} \mathrm H^1(G,U/W)$. In particular, $$\dim_{\GF(p)} \mathrm H^1(G,U/W)\ne 0.$$
\end{lemma}

\begin{proof} Define $\overline G=G/Z(G)$ and set $D=C_X(Z(G))$. Then, as $D=[X,Z(G)]$, $D$ and $X/D$ are isomorphic as $\GF(p)\overline G$-modules.
 Identify  $G$ as a subgroup of $\GL(X)$ and $\overline G$ as a subgroup of $\GL(X)$ by taking some complement $E$ to $D$ in $X$ and assuming that it is $\overline G$-invariant. Let
 $$Q =Q_D=\{q \in \GL(X)\mid X(q-1)\le D \le C_X(q)\}.$$
 Then, as $D$ is invariant under both $\overline G$ and $G$,  $Q$ is a subgroup of  $X$ which is normalized by $\overline G$ and $G$. Notice that $$Q\cong U=\Hom_{\GF(p)}(X/D,D) \cong \mathrm{End}_{\GF(p) }(V)  $$ and $C_Q(\overline G)=Z(G)\cong   W$ as $V$ is absolutely irreducible.

  Define  $H=Q\overline G$ which is a semidirect product. Then $H= QG$ and $Z(H)  $ has order $p$.

Suppose   $L$ and $M$  are complements to $Q$ in $H$. Then $LZ(H)/Z(H)$ and $MZ(H)/Z(H)$ are complements to $Q/Z(H)$ in $H/Z(H)$. Furthermore, if $LZ(H)/Z(H)$ and $MZ(H)/Z(H)$
are conjugate in $H/Z(H)$, then we may as well suppose that both $L$ and $M$ are in $LZ(H)$. Thus  $L= (LZ(H))'=(MZ(H))'= M$. This shows  that $ \dim_{\GF(p)} \mathrm H^1(G,U) \le  \dim_{\GF(p)} \mathrm H^1(G,U/W)$.

Let $L$ be a complement to $Q$ in $H$. Then, in $H/Z(H)$, both $G/Z(H)$ and $LZ(H)/Z(H)$ are complements to $Q/Z(H)$ in $H/Z(H)$. But surely  $G/Z(H) $ and $LZ(H)$ are not conjugate in $H$ as $G$ is perfect and $LZ(H)$ is not. Hence $$ \dim_{\GF(p)} \mathrm H^1(G,U) <  \dim_{\GF(p)} \mathrm H^1(G,U/W),$$
as claimed.
\end{proof}

\begin{remark}
Suppose that $K$ is a non-abelian simple group, $U$ is a $\GF(p)K$-module and $W$ is a submodule of $U$. The short exact sequence of $\GF(p)K$-modules $$0\rightarrow   W\rightarrow U\rightarrow U/W \rightarrow 0$$ leads to a long exact series in cohomology \begin{eqnarray*}0 &\rightarrow & \mathrm H^0(K,W)\rightarrow \mathrm H^0(K,U)\rightarrow \mathrm H^0(K,U/W)\rightarrow \mathrm H^1(K,W)\\&& \rightarrow \mathrm H^1(K,U)\rightarrow \mathrm H^1(K,U/W)\rightarrow \mathrm H^2(K, W)\rightarrow\dots\end{eqnarray*}
 The condition $$\dim_{\GF(p)} \mathrm H^1(K,U) <  \dim_{\GF(p)} \mathrm H^1(K,U/W)$$  implies that $\mathrm H^2(K, W)\ne 0$. Hence $K$ has a non-split extension with $W$.    In particular, taking $W= C_U(K)$, we see that the condition on $1$-cohomology implies that $K$ has a perfect central extension by a group of order $p$ and the $p$-fold cover can be  found in the semidirect product of $U$ and $K$.
\end{remark}

\section{Factorisations}\label{sec:Factorisations}

The purpose of this short section is to present some additional results related to factorisations which we require.

\begin{lemma}\label{lem: in omega} Suppose that $q=2^a$, $n \ge 3$, $G= \SO_{2n}^+(q)$ and $r$ is a prime divisor of $2n$.  Let $M= \SO_{2n/r}^+(q^r).r$ be an extension field subgroup of $G$. Then $M \le G' \cong \Omega_{2n}^+(q)$  if and only if $r=2$.
\end{lemma}

\begin{proof} This can be deduced from \cite[Proposition 4.3.14]{KleidmanLiebeck}. We outline their argument.  Let  $V$ be the natural $\GF(q)G$-module and $Q$ be the associated quadratic form.  Let $T$ be the trace map from $\GF(q^r)$ to $\GF(q)$ and let $Q_\#$ with $TQ_\#=Q$ be the quadratic from on $V_\#=V$ regarded as a $\GF(q^r)$ vector space and which is preserved by $\SO_{2n/r}^+(q^r)$. Let $\mathcal U$ be the collection of maximal totally singular subspaces of $V$ with respect to $Q$ and $\mathcal U_\#$ be the corresponding set for $Q_\#$. Since singular vectors in $V_\#$ with respect to $Q_\#$ are singular with respect to $Q$, we have $\mathcal U_\# \subseteq \mathcal U$. Aschbacher \cite[23.13]{AschbacherFG} defines an equivalence relation on $\mathcal U$ by saying that $ A\sim B$ if and only if $\dim A/(A\cap B)$ is even. Then he shows in \cite[23.14]{AschbacherFG} that $\Omega_{2n}^+(q)$ is the subgroup of $G$ which preserves this equivalence relation. Let $U \in \mathcal U_\#$ and let $t \in M$ be a reflection on $V_\#$ such that $C_U(t)$ is a $\GF(q^r)$ hyperplane of $U$. Then $|U:U \cap U^t|= q^{r}$. Hence $U$ and $U^t$ are in different equivalence classes if and only if $r$ is odd.  In particular, if $r$ is odd, $M \not \le \Omega_{2n}^+(q)$.   If $r$ is even, then any two members of $\mathcal U_\#$ intersect in a $\GF(q^2)$-subspace and so have intersection of even $\GF(q)$-codimension. It follows that $M \le G'=\Omega_{2n}^+(q)$. This verifies the claim.
\end{proof}

  \begin{lemma}\label{lem:Sp fatorisations1}
  Suppose that $n$ is even, $G=\Sp_{2n}(2^a)$, and $M \cong \Sp_{2n/r}(2^{ar}){:}r$ where $r$ is a prime dividing $2n$. Let $K_0=\Omega_{2n}^+(2^a)$ and $K= N_G(K_0) \cong \SO_{2n}^+(2^a)$. Then $G= MK$ and $G=MK_0$ if and only if $r$ is odd.
  \end{lemma}

\begin{proof} Let $q=2^a$. We know that $G=MK$ from \cite{LiebeckPraegerSaxl}. Since $|G:K|=  |M:M\cap K | =q^n(q^n+1)$,  $M \cap N_G(K)$ has index $q^{n}(q^n+1)$ in $M$ and so we deduce that $M\cap N_G(K) \cong \Omega_{2n/r}^+(q^r).2.r $. Now by \cref{lem: in omega} $$|M\cap K:M \cap K_0|= \begin{cases} 2&r>2 \\1&r=2.\\
\end{cases} $$
Hence $$|G|=\frac{|M||K|}{|M \cap K|}= \begin{cases}\frac{2|M||K_0|}{2|M \cap K_0|}= |MK_0|&r>2\\
\frac{2|M||K_0|}{|M \cap K_0|}= 2|MK_0|&r=2
\end{cases}$$
and this proves the assertion.
\end{proof}

  \begin{lemma}\label{lem:More Sp factorisations}
   Suppose that $G=\Sp_{2n}(2^a)$ and $M \cong \Sp_{2n/r}(2^{ar}){:}r$ where $r$ is a prime dividing $2n$. Let $K_0=\Omega_{2m}^-(2^a)$ and $K= N_G(K_0) \cong \SO_{2n}^-(2^a)$. Then $G= MK=MK_0$.
  \end{lemma}

  \begin{proof}
   We know that $G=MK$ from \cite{LiebeckPraegerSaxl}. As in \cref{lem:Sp fatorisations1}, we  only require $M\cap K \not \le K_0$ from which it follows that $G=MK_0$. In this instance this follows from \cite[Proposition 4.3.16 (II)]{KleidmanLiebeck}.
  \end{proof}

  \begin{lemma}\label{lem:G2q factorisations}
  Suppose that $n$ is even, $G=\Sp_{6}(2^a)$ and $M\cong \mathrm G_2(2^a)$. Let $\epsilon  =\pm$, $K_0=\Omega_{2m}^\epsilon(q)$  and $K= N_G(K_0) \cong \SO_{2n}^\epsilon(q)$. Then $G= MK =MK_0$.
  \end{lemma}

  \begin{proof} We know that $G=MK$ from \cite{LiebeckPraegerSaxl}. So, as in \cref{lem:Sp fatorisations1}, we just need to show that $M\cap K \not \le K_0$. Now, from \cite[Subcase 5.2.3 (b)]{LiebeckPraegerSaxl}  we know that $M \cap K \cong \SL_3^\epsilon(q){:}2$.  Now the outer elements of order two induce the graph automorphism when $\epsilon =+$ and act as either the field automorphism or the graph automorphism when $\epsilon = -$.  In particular, $\SL_3^+(q){:}2$ is not a subgroup of $\SL_4(2^a)\cong \Omega_6^+(2^a)$ and $\SL_3^-(q){:}2$ is not a subgroup of $\SU_4(2^a)\cong \Omega_6^-(2^a)$.  This proves our claim.
  \end{proof}

\begin{lemma}\label{lem:unique para}
Suppose that $G$ is a quasisimple classical group defined in characteristic $2$. Assume that $X$ and $Y$ are maximal parabolic subgroups of $G$ containing the same Borel subgroup. If there exists $H < G$ such that $G= (X\cap Y)H$, then $X=Y$.
\end{lemma}

\begin{proof}
Aiming for a contradiction, assume that $X\ne Y$ and set $P=X \cap Y$.  Then $G=HP=HX=HY$ and $P$ is a parabolic subgroup of $G$.

 Using \cite[Theorem A, Tables 1, 2 3 \& 4]{LiebeckPraegerSaxl} (see also \cref{prop:char2 list}), we see that one of the following holds:
\begin{enumerate}
\item[(a)] $G/Z(G) \cong \PSL_n(q)$ and $P =P_1\cap P_{n-1}$ with $n \ge 3$;
\item[(b)] $G \cong \Omega_{2n}^+(q)$, $n \ge 5$ and $P= P_{n}\cap P_{n-1}$; or
\item[(c)] $G\cong \Omega_8^+(q)$ and $P_1\cap P_3\cap P_4 \le P$.
\end{enumerate}

Assume that (a) holds.
Then $G=H(P_1\cap P_{n-1})$ and so $|G|= |H||P_1\cap P_{n-1}|/|H\cap P_1\cap P_{n-1}|$. Notice that $|P_1:P_1\cap P_{{n-1}}|= (q^{n-1}-1)/(q-1)$. Let $p $ be a Zsigmondy prime dividing $q^{n-1}-1$. Then, using \cite[Theorem A, Tables 1, 2 3 \& 4]{LiebeckPraegerSaxl} to obtain information about $H$, we observe that $p$ does not divide  either $|H|$ or $|P_1\cap P_{n-1}|$ and so we have a contradiction. This leaves us to consider $a(n-1)= 6$ in which case there are no Zsigmondy prime divisors of $2^6-1$.  Hence either $a=3$ and $n= 3$, $a=2$ and $n=4$ or $a=1$ and $n=7$. In the first case, $G=\SL_3(2^3)$  and $H\cong \GL_1(2^9):3$, and we check that $H$ has no factorisations with $P_1\cap P_2$. In the second case,  there are two candidates for $H$ and both have order which is not divisible by $7$. Finally, we need to consider $H=\GL_1(2^7)$ in $\SL_7(2)$ and this case also fails. Thus   case (a) cannot hold.

Now consider cases (b) and (c). We argue similarly. Let $U$ be the natural orthogonal module for $G$. We know that  $G=HP_n= HP_{n-1}$. Assume that $G=H(P_n\cap P_{n-1})$.
We have $|P_n:P_n\cap P_{n-1}|= (q^n-1)/(q-1)$.
As $n\ge 4$, we have $|P_1\cap P_{n-1}|$ is not divisible by the Zsigmondy primes dividing $q^n-1$ or $n=6$ and $q=2$ when there is no Zsigmondy prime. Assume that $G \ne \Omega_{12}^+(2)$ and let  $r$ be  a Zsigmondy prime dividing $q^n-1$. Then $H$ must contain a Sylow $r$-subgroup of $G$.  Notice that the subgroup $\GL_n(q)$ which stabilizes a totally singular $n$-space and its opposite, have an element of order $r$ which has exactly two $n$-dimensional composition factors on $U$.  In particular, no such element is conjugate into $N_1$ or $N_2^-$. This contradicts $H$ containing a Sylow $r$-subgroup of $G$.

 Hence $G \cong \Omega_{12}^+(2)$. In this case  a  {\sc Magma} calculation shows that the candidates for $H$ are not transitive on the cosets of $P_{n} \cap P_{n-1}$ and so this case is also impossible. This completes the elimination  of cases (b) and (c). We conclude that $P$ is maximal in $G$.
\end{proof}

\section{Alternating groups}\label{sec:alternating groups}

Until the end of \cref{sec:odd} we shall assume
 \begin{hypothesis}\label{hyp} The group $G$ has $F^*(G)$    quasisimple, $H$ a maximal subgroup of $G$ not containing $F^*(G)$,  $V$ is a faithful irreducible $\GF(2)G$-module and
$(G,H,V)$ is immutable.
\end{hypothesis}

In this section, we assume that \cref{hyp} holds with $F^*(G)$ a  cover of the alternating group of degree $n$   at least $5$. We identify $G$  with a subgroup of $\GL(V)$.

Suppose that $X \cong \Sym(n)$ or $\Alt(n)$ with $n \ge 5$.  Let $W$ be the $\GF(2)X$-permutation module for $X$ with permuted basis $\Omega=\{e_1, \dots, e_n\}$.  Set $W_0= \langle \sum_{i=1}^n e_i\rangle$, $W_1= \langle e_1+e_j\mid 1\le j \le n\rangle$ and put  $V \cong W_1/(W_0\cap W_1)$. Then $V$ is called the \emph{fully deleted permutation module} for $X$.

\begin{lemma}\label{lem:alt cases}
Suppose that $F^*(G)$ is a covering group of $\Alt(n)$, with $n \ge 5$. Assume that $H \ne F^*(G)$. Then   $V$ is either the fully deleted permutation module for $G$ or we have one of the following cases:
\begin{enumerate}
\item $n=7$, $\dim V= 4$ and $(G,H,V)$ is one of $(\Alt(7), \Sym(5), V)$  or $(\Alt(7),\Alt(6),V)$. Furthermore, $H$ acts transitively on $V^\#$.
 \item $n=8$, $\dim V=4$ and $(G,H,V) $ is one of $(\Alt(8),\Alt(7),V)$, $(\Alt(8),\Sym(6),V)$, $(\Alt(8),(3\times \Alt(5)).2,V)$. Furthermore, $H$ acts transitively on $V^\#$.
 \item $n=8$, $\dim V=8$ and $(G,H,V)= (\Sym(8),\Sym(7),V)$. Furthermore, $G$ has three orbits on $V^\#$ and they have lengths $30$, $105$, $120$.
\end{enumerate}
\end{lemma}

\begin{proof} Assume that $V$ is not the fully deleted permutation module for $G$.
By \cref{lem:reg orb compendium} (i), $G$ does  not have a regular orbit on $V$.  Hence, by \cite[Theorem 1.1]{FOS},  the possibilities for $G$ and $V$ are given in    \cite[Table 1]{FOS}. In particular, $5\le n \le 12$.

For each of the possibilities listed \cite[Table 1]{FOS}, we need to determine the candidates for the maximal subgroup $H$ of $G$ that potentially result in $(G,H,V)$ being immutable.  This is done using the maximal subgroups listed in the \cite{ATLAS} or \cite{BHRoney}.

Suppose that $n=5$. Then $F^*(G) \cong \Alt(5)\cong \SL_2(4)$, and $V$ is the natural $\GF(4)\SL_2(4)$-module regarded as a module over $\GF(2)$. Hence $F^*(G)$ acts transitively on the non-zero vectors of $V$. Since $G$ has no proper subgroups not containing $F^*(G)$ of order divisible by $15$, we have no candidates for $H$ in this case.

Suppose that $n=6$. Then $F^*(G) \cong \Alt(6)$ or $3^.\Alt(6)$.
If $F^*(G) \cong \Alt(6)$, then $\dim V=4$ by \cite[Table 1]{FOS}. Hence $V$ is  the natural module for $G\cong \Sp_4(2)'$. As this module can be constructed as the fully deleted permutation module, we reject this possibility in this instance.
Hence we may suppose  that $F^*(G) \cong  3^.\Alt(6)$ and $\dim V=6$. In this case a {\sc Magma}  calculation reveals no examples.

 Suppose that $n=7$. Then $F^*(G) \cong \Alt(7) $ or $3^.\Alt(7)$. The candidates for $V$ are given in \cite[Table 1]{FOS}.
  If $G \cong \Alt(7)$, then $\dim V= 4$ and $G$ acts transitively on $V^\#$. Thus $15$ divides $|H|$.  This means that $H\cong \Alt(6)$ or $H \cong \Sym(5)$ each of which acts transitively on $V^\#$.
If $G \cong \Sym(7)$, then $\dim V \in \{8,14\}$. Using {\sc Magma} exhibits no examples.

The final possibility is that
  $G \cong 3^.\Alt(7)$ and $V$   the  $12$-dimensional module and this potential example is eliminated using {\sc Magma}.

For $n=8$ and $G \cong \Alt(8)$ we use  \cite[Table 1]{FOS} and {\sc Magma} to obtain (ii).

When $n=9$ and $G\cong \Alt(9)$,   {\sc Magma} eliminates the modules of dimension $8$ which are not the fully deleted permutation module as well as the $20$-dimensional possibilities.
Similar calculations show that  $G \cong \Sym(9)$, $\Alt(10)$ and $\Sym(10)$ are not possibilities.

Finally, assume $n=12$. Then $G \cong \Sym(12)$ and $\dim V= 32$. Since $\Alt(12)$ does not appear on  \cite[Table 1]{FOS}, $\Alt(12)$ has a regular orbit on $V$ (see also \cite[Proposition 4.6]{FOS}). If $(G,H,V)$, is immutable then $H$ must have order divisible by $|\Alt(12)|$ hence $H$ has index at most $2$ in $G$, which is a contradiction.
\end{proof}

Recall that a permutation group on a set $\Omega$ is called \emph{$k$-homogeneous} if it is transitive on $k$-subsets of $\Omega$.
\begin{lemma}\label{lem:Orbs del mod} Suppose that $X \cong \Sym(n)$ or $\Alt(n)$ with $n \ge 5$ and $W$ is the $\GF(2)X$-permutation module  with permuted basis $\Omega=\{e_1, \dots, e_n\}$. Let $V \cong W_1/(W_0\cap W_1)$ be the fully deleted permutation module for $X$. Then
\begin{enumerate}
\item the group $X$ has $n$ orbits on  $W$ with orbit representatives $\sum_{i=1}^ke_i$ for $1 \le k \le n$;
\item the orbits of $X$ on $W_0$ have representatives   $\sum_{i=1}^ke_i$ for $1 \le k \le n$ with $k$ even; and
\item if $n$ is even, then the orbits of $X$ on $V$  have representatives $\sum_{i=0}^{k} e_i+W_0$ with $k$ even and $0\le k \le n/2$.
\end{enumerate}
\end{lemma}

\begin{proof}  The vectors of $W$, correspond to subsets of $\Omega$ given by the support of the vector.  Hence, as $X=\Alt(n)$ is $(n-1)$-homogeneous,   the orbits of $X$ on $W$ are determined by the size of the support set. As the subspace $W_1$ consists of those elements of even support, this proves (i) and (ii).

For (iii), suppose that $n$ is even and calculate in $V=W_1/W_0$. For $v,w \in W$ with $v \ne w$, we have $v+W_0= w+ W_0$ if and only if their supports are complementary.  Hence $X$ has orbit representatives $\sum_{i=0}^{k} e_i+W_0$ where $0\le k \le n/2$ is even, as claimed. \end{proof}

\begin{lemma}\label{lem:full deleted}
Suppose that $( G,H,V)$ is   immutable  with $F^*(G) \cong \Alt(n)$, $n \ge 9$ and $V$ the fully deleted permutation module. Then
 $n=9$, $G \cong \Alt(9)$,   $H\cong \mathrm P\Gamma\mathrm L_2(8)$ (two conjugacy classes) and $\dim V=8$. Furthermore, if $G\cong \Alt(9)$, then its orbits on $V^\#$ have lengths  $9$, $36$, $84$, and $126$.
\end{lemma}

\begin{proof} We have $G =\Alt(n)$ or $G=\Sym(n)$. We adopt the notation introduced in \cref{lem:Orbs del mod} for the fully deleted permutation module for $G$.
Suppose that $t \le n/2$ is even with $0<n/2-t \le 2$. We claim that $H$ is $t$-homogeneous acting on the set $\Omega=\{1, \dots, n\}$.  So let $\tau_1 $ and $\tau_2$ be arbitrary $t$-sets.

If $n$ is odd, then, by \cref{lem:Orbs del mod}  (ii), $V=W_1$ and $\sum_{i\in \tau_1}e_i$ is in the same $G$-orbit as $\sum_{i\in \tau_2}e_i$. Hence, as $( G,H,V)$ is immutable, they are also in the same $H$-orbit.  Thus there exists $h\in H$ such that $\tau_1h=\tau_2$. Hence $H$ is $t$-homogeneous in this case.

Suppose that $n$ is even. Then $t < n/2$. In this case $\sum_{i\in \tau_1}e_i+W_0$ is in the same $G$-orbit as $\sum_{i\in \tau_2}e_i+W_0$. Then, as $t < n/2$, $\sum_{i\in \tau_1}e_i+W_0\ne \sum_{i\in \tau_2}e_i+W_0$. Hence as   $(G,H,V)$ is immutable, again we have a $h\in H$ such that $\tau_1h=\tau_2$.

Since $n \ge 9$, we conclude in both cases that $H$ is  $t$-homogeneous with $t \ge 4$. Hence \cite[Theorem 1]{Kantorhom} implies  $H$ is either $t$-transitive or $n=9$ and  $H$ is $4$-homogeneous. Furthermore possibilities for $H$ are subgroups of $\Alt(9)$  and our maximality assumption yields $H \cong \mathrm P\Gamma\mathrm L_2(8)$. It is then an easy {\sc Magma} calculation to see that $(\Alt(9), \mathrm P\Gamma\mathrm L_2(8), V)$ is immutable with the described orbit lengths.

So assume that $H$ is $t$-transitive with $t$ as described above.
For $n \in \{9,10,11,12\}$, $t=4$ and so $H$ is $4$-transitive and for $n \ge 13$ we have $t \ge 6$.
 Then, as $H$ does not contain $\Alt(n)$, \cite[Theorem 4.11, pg. 110]{CameronPermutationGroups}  gives $t=4$ and  either $H \cong \mathrm M_{11}$ with $n=11$ or $F^*(H)\cong \mathrm M_{12}$ with $n=12$. A calculation with {\sc Magma} shows that $\mathrm M_{11}$ has $7$ orbits on its $10$-dimensional modules whereas $\Alt(11)$ has $6$ orbits. Similarly, we calculate that $(\Alt(12),\mathrm M_{12},V)$ is also not immutable this time with $\Alt(12)$ having $4$ orbits and $\mathrm M_{12}$ having $5$. This completes the proof of the lemma.
\end{proof}

We next examine the smallest possibilities.

\begin{lemma}\label{small:alt}
Suppose that $V$ is the fully deleted permutation module for $G$ with $F^*(G) \cong \Alt(n)$, $5\le n \le 8$. Suppose that $H$ is a maximal subgroup of $G$, with $H \ne F^*(G)$.

Then $(G,H,V)$ is immutable if and only if $(G,H,V)$ is one of the following:
\begin{enumerate} \item $(\Sym(5),\mathrm{Frob}(20),V)$, with orbits of length $5$ and $10$ on $V^\#$; \item $(\Alt(6),\SL_2(4),V)$ with $G$ transitive on $V^\#$;
 or \item $(\Sym(6),\Gamma\mathrm L_2(4),V)$ with $G$ transitive on $V^\#$.
\end{enumerate}
\end{lemma}

\begin{proof} This follows from a straightforward  {\sc Magma} calculation.\end{proof}

The contribution to the proof of \cref{Main theorem} of this section is the following proposition.

\begin{proposition}\label{prop:altgrps}
Suppose that $ G $ is a cover of an alternating group $\Alt(n) $ with $n \ge 7$, $H$ is a maximal subgroup of $G$ and $V$ is an irreducible $\GF(2)G$-module. If $(G,H,V)$ is   immutable,  then either
\begin{enumerate}\item
$n\in \{6,7,8\}$,  $H$ is transitive on $V^\#$ and $\dim V=4$; or
 \item
  $n=9$, $H \cong \mathrm P\Gamma\mathrm L(2,8)$, and $V$ is the deleted permutation module for $G$.
\end{enumerate}
\end{proposition}

\begin{proof} When $V$ is not the fully deleted permutation module, \cref{lem:alt cases} provides the claim in part (i). When $V$ is  the fully deleted permutation module and $n \ge 9$, the result  in (ii) follows from
\cref{lem:full deleted}.  The cases with $V$    the fully deleted permutation module and $5 \le n \le 8$ are handled in \cref{small:alt} yielding one more example in case (i).
\end{proof}

\section{Sporadic simple groups}\label{sec:sporadics}

In this section,  we assume  \cref{hyp}   with $F^*(G)$ is a cover of a sporadic simple group.  By \cref{lem:reg orb compendium} there exists  $v \in V^\#$ is such that $|G:C_G(v)|$ is odd and that $vG=vH$. Thus we have $G= C_G(v)H  $ by \cref{lem:basic1}. Our first lemma exploits this factorisation to provide a short list for the candidates for $F^*(G)/Z(F^*(G))$.

\begin{lemma}\label{lem:sporadic1} Suppose that $F^*(X)$ is  a sporadic simple group, $A$ and $B$ are maximal subgroups of $X$ with $X= AB$ and   $|X{:}B|$ odd. Then
\begin{enumerate}
\item $X=\M_{11}$, $A=\PSL_2(11)$ and $B= \M_{10}$ of index $11$.
\item $X=\M_{11}$, $A=\PSL_2(11)$ and $B= \M_{9}.2=3^2{:}\SDih(16)$ of index $55$.
\item $X=\M_{12}$, $A= \M_{11}$  $B= 4^2{:}\D_{12}$ of index $495$.
\item $X=\M_{23}$, $A= 23{:}11$ and $B= \M_{22}$ of index $23$.
  \item $X=\M_{23}$, $A= 23{:}11$ and $B= \PSL_3(4).2$ of index $253$.
  \item $X=\M_{23}$, $A= 23{:}11$ and $B= 2^4{:}\Alt(7)$ of index $253$.
\item $X=\M_{24}$, $A= \M_{23}$ and $B= 2^6{:}3^. \Sym(6)$ of index $1771$.
\item $X=\M_{24}$, $A= \M_{23}$ and $B= 2^6{:}(\SL_3(2)\times \SL_2(2))$ of index $3795$.
\item $X=\M_{24}$, $A=\PSL_2(23)$ and $B= 2^4{:}\SL_4(2)$ of index $759$.
\end{enumerate}
\end{lemma}

\begin{proof}
This applies \cite[Theorem C, Table 6]{LiebeckPraegerSaxl}. Note that in this application, we have $X= L$.
\end{proof}
 
\begin{proposition}\label{cor:no sporadics} If \cref{hyp} holds, then $G$ is not a cover of a sporadic simple group.
\end{proposition}

\begin{proof} Suppose that \cref{hyp}  holds and that $G$ is a cover of a sporadic simple group.
Then
  \cref{lem:sporadic1} shows that there are only a small number of pairs $(G,H)$ to consider and in all cases  $G/Z(G)$ is a Mathieu group.  Using \cref{lem:reg orb compendium}(i) and \cite[Theorem 1, Table 1]{FMJO}, yields that in each case $V$ is either the code, or the cocode module or that $G \cong 3^.\M_{22}$ and $\dim V=12$.
  An elementary {\sc Magma} calculation shows that the number of orbits of $G$ and $H$ on $V$ is different. Thus there are no immutable triples $(G,H,V)$ and this contradicts the fact that \cref{hyp} holds.  This completes the proof.
\end{proof}

\section{Exceptional groups}\label{sec:exceptional}

This very short section considers the possibility that $ G $ is a cover of a  exceptional group of Lie type. Assume \cref{hyp}.

\begin{proposition}\label{thm: Exceps} Suppose that  \cref{hyp} holds. Then $ G $   is not a cover of an exceptional simple group.
\end{proposition}

\begin{proof} Suppose that the claim is false. Then, by \cref{lem:basic1},  $G= C_G(v)H$ is a factorisation of $G$ for all $v \in V^\#$. If $G$ is defined in characteristic $2$, then we can choose $v$ to be centralized by Sylow $2$-subgroup of $G$.  Hence there must exist a maximal factorisation with one of the factors a parabolic subgroup. However,  \cite[Theorem B, Table 5]{LiebeckPraegerSaxl} shows that no such factorisation exists. Hence $G$ is defined in odd characteristic. Now \cite[Theorem B, Table 5]{LiebeckPraegerSaxl}
 yields that $ G  \cong \G_2(3^d)$ and $G= HL$ where $O^2(L) $ has to be one of $ \SL_3(3^d)$, $\SU_3(3^d)$, or ${}^2\G_2(3^d)$. \cref{lem:prank} implies that we can choose $L$ so that the $3$-rank of $L$ is $m_3( G )-1$.  By \cite[Theorem 3.3.3]{GLS3}, $m_3(G)=4d$ whereas $m_3(\SL_3(3^d))= m_3(\SU_3(3^d))= m_3({}^2\G_2(3^d))=2d$. This shows that \cref{hyp} cannot hold when $G$ is an exceptional group.
\end{proof}

\section{Classical groups in odd characteristic}\label{sec:odd}

In this section we   assume that \cref{hyp} holds with $ G /Z(G)$   a  simple Lie type group defined in odd characteristic. Since $V$ is irreducible, we may as well assume that $Z(G)$ has odd order.

Recalling that $G$ has no regular orbits on $V$ by \cref{lem:reg orb compendium}, we use \cite[Corollary 1.2]{Lee1} to obtain the following  possibilities for the pair $(G,V)$ (we have kept to the same ordering as in \cite{Lee1} to make checking the list more elementary). We first list the cases that are absolutely irreducible over $\GF(2)$. This is elementary to do from  \cite[Table 1.2]{Lee1} and so we do not require a proof.

\begin{enumerate}
\item $G=\PSL_2(4)=\Alt(5)$ and $\dim V=4$ (*).
\item $G=\PSL_2(8) $ and $\dim V = 8$(***).
\item $G=\PSL_2(11)$ and $\dim V=10$.
\item $G=\PSL_2(17)$ and $\dim V= 8$.
\item $G= \PSL_2(23)$ and $\dim V= 11$.
\item $G=\PSL_2(25)$ and $\dim V= 12$.
\item $G=\PSL_2(31)$ and $\dim V= 15$.
\item $G=\PSL_3(2)$ and $\dim V=8$ (***).
\item $G=\PSL_3(3)$ and $\dim V= 12$.
\item $G=\PSL_4(3)$ and $\dim V = 26$.
\item $G=\PSU_3(3) $ and $\dim V= 14$.
\item $G=\PSU_3(5)$ and $\dim V= 20$.
\item $G=\PSU_4(2)$ and $\dim V= 14$(***).
\item $G=\PSU_4(3)$ and $\dim V= 20$.
\item $G= \PSp_4(7)$ and $\dim V= 24$.
\item $G=\G_2(3) $ and $\dim V=14$ (**).
\item $G=\PSL_3(2)$ and $\dim V=3$ (***).
\item $G=\PSL_2(9)$ and $\dim V= 4$ (*).
\item $G=\PSU_3(3)$ and $\dim V=6$.
\item $G=\PSU_4(2) $ and $\dim V= 6$ (***).
\end{enumerate}

From the list above, we can eliminate those groups with socle isomorphic to an alternating group because of the results in \cref{sec:alternating groups}. These are labeled by a (*). Because of \cref{thm: Exceps}, we may also remove the exceptional group possibility  $\G_2(3)$ indicated by (**).
Finally, we remove groups which are defined in characteristic $2$ indicated by (***), as these groups are the subject of the next section.

This leaves 12 cases. These are easily handled with a {\sc Magma} calculation which compares the number of orbits of $G$ on $V$ with those of any maximal subgroup. The exception to this is  possibility (x), where we just make some random orbits and see that there are no examples of immutable triples.

Now we pick up those $\GF(2)G$-representations which are not absolutely irreducible. Here some small arguments are needed. Hence  we state this as a lemma.
\begin{lemma}\label{lem:not abs red}
Suppose that $ G $ is quasisimple and  $V$ is a  $\GF(2)G$-module which is irreducible but not absolutely irreducible.  Assume that $G$ has no regular orbit on  $V$. Then one of the following holds:

\begin{enumerate}\item   $G=\PSL_2(4)=\PSL_2(5) $ and $\dim V=4$.

\item $G= \PSL_2(8)$ and     $\dim V= 6$.
\item   $G= 3^{.} \PSL_2(9)\cong 3^. \Alt(6)$  and $\dim V=6$.
\item $G=\PSL_2(11)$ and $\dim V=10$.
\item  $G=\PSp_6(3)$ and $\dim V= 26$.
\item  $G=\PSU_4(2)$ and $\dim V=8$.
\item  $G=3_1.\PSU_4(3) $ and $\dim V= 12$.
\end{enumerate}
\end{lemma}

\begin{proof}
Again we invoke \cite[Table 1.2]{Lee1} this time looking for $2$-powers greater than $2$ in column $4$ and the examples where $G$ is quasisimple. In particular, in \cite[Table 1.2]{Lee1},   the parameter $c$ must be allowed to be $1$. We also recall that $V$ regarded as a $\GF(2)G$-module   must be irreducible. We use {\sc Magma} to enumerate the irreducible $\GF(2)G$-modules for small groups with $ G $ quasisimple and include the program in the additional materials. This results in the presented list.
\end{proof}

As usual, {\sc Magma} calculations eliminate all the cases listed in \cref{lem:not abs red}.

\begin{proposition}\label{prop:no Lie Odd} If \cref{hyp} holds, then $G/Z(G)$ a not a simple group of Lie type   defined in odd characteristic unless it is isomorphic to an alternating group or a group of Lie type defined in characteristic $2$.
\end{proposition}

\section{Classical groups in characteristic $2$}\label{class char 2}

This section is devoted to finding the immutable triples $(G,H,V)$ with $G$ a quasisimple   classical group defined in  characteristic $2$, $H$ maximal subgroup of $G$ and $V$ an irreducible $\GF(2)G$-module. We follow \cite[Section 2]{GLS3} for our notation related to groups of Lie type. Throughout we will consider universal groups defined over a finite field of characteristic $2$ and order $q=2^a$.  This means that the non-trivial representations that we consider  may have a non-trivial kernel. We formalize our assumptions  in the following hypothesis.

\begin{hypothesis}\label{hypss7} \cref{hyp} holds and for $q=2^a$,   $G$ is a central quotient of one of
\begin{enumerate}
 \item $\SL_m(q)$, with $m \ge 2$ and $(m,q) \ne (2,2)$;
  \item $\SU_m(q)$ with $m \ge 3$ and $(m,q) \ne (3,2)$;
  \item $\Sp_{m}(q)$, $m \ge 4$ even and $(m,q) \ne (4,2)$; or
  \item $\Omega_{m}^+(q)$  or $\Omega_{m}^-(q)$ with $m \ge 8$ even.
   \end{enumerate}
\end{hypothesis}

We assume that \cref{hypss7}   holds throughout this section. Define $\mathbb L=\mathrm{End}_G(V)$ and let $\overline{\mathbb L}$  be the algebraic closure of $\mathbb L$. For information about the characteristic $2$ representation theory of $G$ we refer to \cite[Section 2.8]{GLS3} and the brief discussion in  \cref{sec:Immutability} of this paper. Every irreducible module for a classical group of Lie type can be written over $\GF(q^2)$ \cite[Proposition 5.4.4]{KleidmanLiebeck}.

We start with the following observation.

\begin{lemma}\label{lem:parabolic factorisation}  We have $G= PH$ for some maximal parabolic subgroup $P$ of $G$.
\end{lemma}

\begin{proof} Let    $S \in \Syl_2(G)$.  Then the vectors $v \in C_V(S)^\#$ are high weight vectors. As $(G,H,V)$ is immutable, $vG= vH$ and $G= C_G(v)H$. Since $S \le C_G(v)$, \cite[Theorem 2.6.7]{GLS3} implies there is a maximal parabolic subgroup $P \ge S$ such that $C_G(v) \le P$. But then  $G=C_G(v)H=PH$ and this proves the claim. \end{proof}

We now examine the smallest possibility. We provide a somewhat elementary argument.

\begin{lemma}\label{lem:no factSL2}
We have  $G \not\cong \SL_2(q)$ with $q\ge 4$. \end{lemma}

\begin{proof}  Since $G$ has elements of order $q+1$, $|V| \ge q^2$ by Zsigmondy's Theorem.
 As $(G,H,V)$ is immutable, we have $G= C_G(w)H$ for all $w \in V^\#$.  In particular, by \cref{lem:parabolic factorisation}, $G=HB$ where $B$ is a Borel subgroup of $G$ and so $H$  has order divisible by $q+1$. The subgroup structure \cite[Satz II.8.27]{Huppert} of $\SL_2(q)$ shows that $H$ is contained in the normalizer of the cyclic subgroup of order $q+1$ and so $H$ is  dihedral of order $2(q+1)$. Hence, for $w\in V^\#$, $C_G(w)$ has index  at most $2(q+1)$ in $G$.    Since $G$ has no subgroup of index $2(q+1)$, we have that $|w G|=q+1$ for all $w \in V^\#$ by a theorem of Galois \cite[Satz II.8.28]{Huppert}. In particular,   $C_G(w)$ is conjugate to $B$ for every $w \in V^\#$. Hence $$|C_V(B)|= \frac{(|V|-1)}{(q+1)} +1$$ and so $$|V|= (|C_V(B)|-1)(q+1)+1.$$ As $|V|\ge q^2$, it follows that $|C_V(B)|\ge q$.
Let $B_1\ne B$ be a conjugate of $B$.
Then $G=\langle B,B_1\rangle $ and so
$C_V(B_1) \cap C_V(B)=0$. Hence $$|V|\ge |C_V(B)|^2$$
and therefore \begin{eqnarray*}
|C_V(B)|^2&\le (|C_V(B)|-1)(q+1)+1\le |C_V(B)|(q+1).
\end{eqnarray*}
Hence $|C_V(B)|\le q+1$ and as this order is a power of $2$,  we obtain $q\le |C_V(B)|\le q$ and so $|C_V(B)|=q$.  Thus, $$|V| = (q-1)(q+1)+1=q^2.$$    We conclude that $V= C_V(B)+C_V(B_1)$.  But then $T=B\cap B_1$ centralizes $V$ and this is impossible as $|T|=q-1>1$ and $V$ is non-trivial. This shows that $(G,H,V)$ is not immutable for any proper subgroup $H$ of $V$.
\end{proof}

We adopt the  notation for parabolic subgroups  in classical groups from \cite[Page 5]{LiebeckPraegerSaxl}. Thus $P_i$ denotes the stabilizer of an $i$-dimensional totally singular subspace of the natural module associated with $G$ except when $G=\Omega_{m}^+(q)$ with $m=2n$ even. In this case, $P_i$, $1\le i\le n-2$, is the stabilizer of a totally singular $i$-space and $P_{n-1}$ and $P_n$ are the stabilizers of totally  singular $n$-spaces from different orbits.

 By $N_i$ we denote the stabilizer in $G$ of a non-degenerate $i$-dimensional subspace of $V$ and in the event that $G$ is an orthogonal group and the subspace has type $\epsilon\in \{\pm\}$, we indicate this with a superscript $N_i^\epsilon$. With this notation, using \cref{lem:parabolic factorisation} with \cite[Theorem A, Tables 1, 2 3\& 4]{LiebeckPraegerSaxl} (and including \cite[Table 2]{GGS}) yields the following result.

\begin{proposition}\label{prop:char2 list}
Assume \cref{hypss7} and that $m \ge 3$ when $G \cong \SL_m(q)$. Then there exists a maximal parabolic subgroup $P$ of $G$ such that $G=PH$. Furthermore, one of the following holds.
\begin{enumerate}
\item  $G=\SL_m(q)$, $m \ge 3$, $H$ normalizes a superfield subgroup $\SL_{m/b}(q^{b})$ where $b$ is a prime dividing $m$ and $P\in \{P_1,P_{m-1}\}$.
\item $G= \SL_m(q)$, $m\ge 4$ even, $H$ normalizes $\Sp_{m}(q)$ and $P\in \{P_1,P_{m-1}\}$.
\item $G= \Sp_{m}(q)$, $m=2n$, $n \ge 2$,  $H$ normalizes the superfield subgroup $\Sp_{2(n/b)}(q^{b})$, $b$ is a prime  dividing $n$ and $P=P_1$.
\item $G= \Sp_{m}(q)$, $m=2n$, $n \ge 2$, $H$ normalizes a subgroup $\Omega_{m}^-(q)$ and $P= P_n$.
\item $G=\SU_{m}(q)$, $m=2n$ $n \ge 2$, $H=N_1$ and  $P= P_{n}$.
\item $G=\Omega_{m}^-(q)$, $m=2n$,  $n\ge 5$ with $n$ odd, $H$ normalizes a subgroup $ \SU_m(q)$ and $P=P_1$.
\item $G=\Omega_{m}^+(q)$, $m=2n$, $n \ge 5$,  $H$ is conjugate to $ N_1$  and $P\in\{P_{n-1},P_n\}$.
\item $G=\Omega_{m}^+(q)$, $m=2n$, $n \ge 5$,  $H=N_{2}^-$ and $P\in\{P_{n-1},P_n\}$.
\item $G=\Omega_{m}^+(q)$, $m=2n$, $n \ge 5$ with $n$ even, $H$ normalizes a subgroup $ \SU_m(q)$  and  $P=P_1$.
\item $G=\Sp_6(q)$, $H$ normalizes $\G_2(q)$ and $P=P_1$.
\item $G=\Omega_8^+(q)$, $H $ normalizes $\Sp_6(q)$ and $P\in \{P_1,P_3,P_4\}$.
    \item $G=\Omega_8^+(q)$, $H=N_2^-$ and $P\in \{P_1,P_3,P_4\}$.
\item $G=\SL_5(2)$, $H\cong  \GL_1(2^5).5=31{:}5$ and $P\in \{P_2,P_3\}$.
\item $G=\SU_4(2)$, $H\cong  3^3.\Sym(4)$ and  $P=P_2=2^{ 4}{:}\Alt(5)$.
\item $G= \SU_9(2)$, $H/Z(H) \cong \mathrm J_3$ and $P=P_1$.
\item $G=\Omega_{10}^-(2)$, $H\cong \Alt(12)$, and $P=P_1$.
  \item   $G=\Omega_8^+(2)$, $H \cong \Alt(9)$ and $P\in \{P_1,P_3,P_4\}$.
\end{enumerate}
\end{proposition}\qed

We first consider all the exceptional examples  in \cref{prop:char2 list}.

\begin{lemma}\label{odd balls}
Suppose that  \cref{prop:char2 list} (xiii)-(xvii) holds. Then  (xvii) holds. Specifically, $G \cong \Omega_8^+ (2)$, $H \cong \Alt(9)$, and $V$ is either the natural module or one of the two half-spin modules. In each case, there are two candidates for $H$ and $G$ and $H$ have three orbits on the vectors of $V$.  Furthermore, $V|_H$ is a  spin module for $H$.
\end{lemma}

\begin{proof}
We begin with case (xiii). By \cref{lem:irred reduct}, $H$ acts irreducibly on $V$. Since the irreducible $\GF(2)H$-modules have dimension at most $5$, this yields $\dim V=5$. But then $P\in \{P_1,P_4\}$ as $P_2$ and $P_3$ operate with no fixed points on $V$, a contradiction.

For case (xiv) we adopt the same approach, but this time {\sc Magma} shows that the largest irreducible $\GF(2)H$-module has dimension $16$.  On the other hand,   the non-trivial irreducible $\GF(2)G$-modules of dimension at most $16$   which can restrict to $H$ irreducibly  have dimensions $6$ and $8$. An orbit calculation with {\sc Magma} then shows that these cases lead to no examples.

In case (xv) we have $G= \SU_9(2)$. This case is already too large for {\sc Magma} to calculate all the irreducible $\GF(2)G$-modules. Let $S \in \syl_2(G)$ and consider $V= V(\lambda)$ an irreducible $\overline{\mathbb L}G$-module. We know $C_V(S)$ is normalized by $P_1$ and $H$ does not factorise with any other conjugacy class of parabolic subgroups. The weights for $V$ are $2$-restricted and so $V$ is the restriction of a $2$-restricted irreducible module of $\SL_9(4)$. Since $P_1$ fixes $C_V(S)$ and $C_V(S)= C_V(S_1)$ where $S_1$ is a Sylow $2$-subgroup of $\SL_9(4)$ containing $S$, we have $P_1 \le Q$ where $Q$ is a parabolic subgroup of $\SL_9(4)$.  Because $P_1$ fixes a $1$-space and an $8$-space when it acts on the natural $\SU_9(2)$-module over $\GF(4)$, we see that $P_1$ is in a parabolic subgroup of $\SL_9(4)$ which fixes both these subspaces.  From this, as $\lambda$ is $2$-restricted, \cref{high weights} implies that  $\lambda\in \{ \omega_1, \omega_2,\omega_1+\omega_2\}$. Hence $V$ is either the natural $\SU_9(2)$-module $U$, its dual, or is the irreducible  constituent $A$ of $U\otimes U$ of dimension $80$.  These modules are all realized over $\GF(4)$.

 Let $n \in U$ be a non-isotropic vector, then $L=C_G(n) \cong \SU_8(2)$ is uniquely contained in its normalizer $N_1$.  Since there are no factorisations $N_1H$ by  \cite[Theorem A, Tables 1, 2 \& 3]{LiebeckPraegerSaxl}, we see that $V= U$ is impossible.  Of course, we now have $(G,H,U^*)$ is not immutable by  \cref{lem:Dual same number orbits}.

 Suppose that $V= A$. Then on restriction to $L=\SU_8(2)$, $U$ becomes $1\oplus W$ with $W$ of dimension $8$. Hence, as an $L$-module, we calculate $$V\otimes V^*= (1\oplus W)\otimes (1\oplus W^*) = 1\oplus W\oplus W^* \oplus   ( W\otimes W^*).$$  As $W\otimes W^*$  has a $1$-dimensional $L$-submodule, we conclude that $L$ has a fixed vector on $A$.
 Now we apply \cref{lem:good vectors} with \cite[Theorem A, Tables 1, 2 \& 3]{LiebeckPraegerSaxl} to obtain a contradiction.

For case (xvi), we have  $V$ is the natural module again by \cref{high weights}. Now we calculate with {\sc Magma} to obtain a contradiction.

Finally, we consider  case (xvii). Since the maximal parabolic subgroup of shape $2^{1+8}{:}(\Sym(3)\times \Sym(3) \times \Sym(3))$ and the parabolic subgroups of shape $2^{6+3}{:}\SL_3(2)$ have index which does not divide $|\Alt(9)|$, the stabiliser of a high-weight vector in $V$ is one of $P_1$, $P_3$ or $P_4$. Hence one final application of \cref{high weights} yields $V$ is either the natural module or a half-spin module. Using the triality automorphism we may assume that $V$ is the natural module. This time, using {\sc Magma}  we obtain the examples as detailed in the lemma.
\end{proof}

  Because of \cref{odd balls} from now on we assume that one of \cref{prop:char2 list} (i)--(xii) holds.
We first provide a uniqueness result aligned with \cref{lem:parabolic factorisation}.

\begin{lemma}\label{unique para}
Suppose that $P=N_G(C_V(S))$. Then $P$ is a maximal parabolic subgroup of  $G$.
\end{lemma}

\begin{proof}
We know that $P$ is a parabolic subgroup of $G$ by \cref{high weights}. Assume, aiming for a contradiction,  that $P$ is not maximal.  Then $P$ is contained in two different maximal parabolic subgroups say $X$ and $Y$ and $G=HP=HX=HY$. Now \cref{lem:unique para} delivers $X=Y$, a contradiction.
\end{proof}

We are now in a position to describe the modules which we must consider.  Recall that $\mathbb L=\mathrm{End}_G(V)$ , $\overline {\mathbb L}$ is the algebraic closure of $\mathbb L$ and we may regard $V$ as an absolutely irreducible $\mathbb L G$-module. We make $V$ into an $\overline{\mathbb L}G$-module by forming $\overline V= V\otimes_{\mathbb L}\overline{\mathbb L}$. We adhere to this notational convention for the remainder of this section.

Together with \cref{prop:char2 list,unique para} and the Steinberg tensor product theorem \cite[Corolary 2.8.6]{GLS3}, \cref{high weights} leads to the characterisation of the modules that we will need to understand. When dealing with high weights and the corresponding modules, we follow the notation established in \cref{sec:Immutability} just before \cref{high weights}.Our general strategy is to write down subgroups which fix non-trivial vectors on $V$.

\begin{proposition}\label{the new list} Assume \cref{hypss7} with $m \ge 3$.    Then
there is a subset $\Theta \subseteq \Sigma$ and a basic irreducible $\overline{\mathbb L}$-module $\overline{W}$ such that
$$\overline{V}= \bigotimes_{\theta \in \Theta} \overline{W}^{\theta}$$
where the possibilities for $G$ and $\overline{W}$ are described as follows.
\begin{enumerate}
\item \cref{prop:char2 list} (i) or (ii) holds, $G/Z(G)=\PSL_m(q)$ with $m \ge 3$  and either\begin{enumerate} \item $\overline{W}=V(\omega_1)$ is the natural $m$-dimensional $\overline{\mathbb L}G$-module; or\item $\overline{W}=V(\omega_{m-1})$ is the dual natural $m$-dimensional $\overline{\mathbb L}G$-module in (i)(a).\end{enumerate}
\item \cref{prop:char2 list} (iii), (iv) or (x) holds, $G = \Sp_{2n}(q)$ with $n \ge 2$,    and  either\begin{enumerate} \item $\overline{W}=V(\omega_1)$ is the natural $2n$-dimensional $\overline{\mathbb L}G$-module; or\item $\overline{W}=V(\omega_{n})$ is the $2^n$-dimensional spin  $\overline{\mathbb L}G$-module.\end{enumerate}
\item \cref{prop:char2 list} (v) holds, $G/Z(G)= \PSU_{2n}(q)$ with $n \ge 2$,  and $\overline{W}= V(\omega_n)$ is the  $n$th  exterior square of the natural $\overline{\mathbb L}\SU_{2n}(q)$-module.
\item   \cref{prop:char2 list} (vi) or (ix) holds, $G = \Omega_{2n}^\epsilon(q)$ with $n \ge 5$, $\epsilon 1= (-1)^n$,   and $\overline{W}= V(\omega_1)$ is the natural $\overline{\mathbb L}G$-module.
\item   \cref{prop:char2 list} (vii) or (viii) holds, $G = \Omega_{2n}^+(q)$ with $n \ge 5$,   and $\overline{W}= V(\omega_{n-1})$ or $\overline{W}=V(\omega_{n}) $ is a half-spin $\overline{\mathbb L}G$-module.
\item  \cref{prop:char2 list} (xi) or (xii) holds, $G = \Omega_{8}^+(q)$, and up to relabeling the fundamental weights by triality, and $\overline{W}=V(\omega_1)$ is a natural $\overline{\mathbb L}G$-module.
\end{enumerate}
\end{proposition}

\begin{proof} For the untwisted groups, this is a direct application of \cref{unique para,high weight app,prop:char2 list}.

Suppose that $G=\SU_{2n}(q)$ and set $R=\SL_{2n}(q^{2})$.
Then $\overline V$ is the restriction of an $\overline{\mathbb L}R$-module,   $V(\lambda)$ with $q$-restricted weight $\lambda$. Since $C_V(S)$ is normalized by $P_n$ in $G$, $P_n$ is uniquely contained in the parabolic subgroup $P_n^*$ of $R$. Now \cref{high weight app} yields $\Theta \subseteq \Sigma$ and $\overline{W}= V(\omega_n)$ such that
$$\overline V= \bigotimes_{\theta \in \Theta} \overline{W}^\theta.$$
The identification of $V(\omega_n)$ with the exterior $n$th power, is well-known.

The situation for $G=\Omega_{2n}^-(q)$ is similar this time taking the embedding of $G$ into $\Omega_{2n}^+(q^{2})$.
\end{proof}

We now use \cref{the new list} as the list of core cases to be considered, examining them each in turn.

\begin{lemma}\label{lem:more factors impossible}
Assume that \cref{the new list} (i) holds. Then    $\overline V= \overline W^{\theta}$ for some $\theta\in \Sigma$.  In particular,  $V$ is a natural $\GF(2)G$-module and $H$ acts transitively on $V^\#$.
\end{lemma}
\begin{proof}  In this case $G/Z(G)=\PSL_m(q)$ with $m \ge 3$  and either\begin{enumerate} \item $\overline{W}=V(\omega_1)$ is the natural $m$-dimensional $\overline{\mathbb L}G$-module; or\item $\overline{W}=V(\omega_{m-1})$ is the dual natural $m$-dimensional $\overline{\mathbb L}G$-module.\end{enumerate}

 We know that $\overline{V}= \bigotimes_{\theta \in \Theta} \overline W^{\theta}$ as an $\overline {\mathbb {L}} G$-module and we may  suppose that $\overline W= V(\omega_1)$ by  \cref{lem:Dual same number orbits}. Hence $\overline W$ is a natural $\overline {\mathbb L} G$-module. Let  $\mathbb K=\GF(q)$. Then
 $\overline{W} = W\otimes_{\mathbb K} \overline{\mathbb L}$ where $W$ is the natural $\mathbb K G$-module of $\mathbb K$-dimension $m$.  We write   $U= \bigotimes_{\theta \in \Theta} W^{\theta}$  and note that $U\otimes _{\mathbb K} \overline{\mathbb L} = \overline V$. Thus $U= V \otimes_{\mathbb L}\mathbb K$,
and, when $U$ is regarded as a $\GF(2)G$-module, we can identify $U$ as a direct sum of some number of modules isomorphic to our original module $V$.

Write $\Theta=\{\sigma_1, \dots, \sigma_\ell\} \subseteq \Sigma$.  Assume that $|\Theta|=\ell\ge 2$.

We begin by considering the cases with $m \ge 4$.  Set $L \cong \SL_2(2)\times \SL_{m-2}(q)$ and regard $L$ as a subgroup of $G$ embedded in $\SL_2(q)\times \SL_{m-2}(q)$ with the first factor centralized by $\sigma$.

Then as a $\mathbb KL$-module   $W= N\oplus M$ where $\dim_{\mathbb K} N=2$ and $\dim_{\mathbb K} M= m-2$.  Further we note that $N^{\sigma_i}=N$ as $\sigma$ centralizes the first factor of $L$. As $m\ge 4$, $L$ is not contained in a parabolic subgroup conjugate to either $P_1$ or $P_{n-1}$. We have \begin{eqnarray*} U|_L&=&
 (N\oplus M) ^{\sigma_1}\otimes
 \dots\otimes (N\oplus M) ^{\sigma_{\ell}}\\&=&
 (N\oplus M ^{\sigma_1})\otimes
 \dots\otimes (N\oplus M ^{\sigma_{\ell}})\\
 &\ge& \underbrace{N\otimes \dots \otimes N}_{\ell \text{ copies}}.
\end{eqnarray*}
  Since $N \otimes N \cong T\oplus N$ where $T$ is indecomposable with   socle the trivial module,  by induction we obtain $\underbrace{N\otimes \dots \otimes N}_{\ell \text{ copies}}$ contains a non-zero vector fixed by $L$ and so, therefore, does $U$. Since $L$ is not   conjugate to a subgroup  of $P_1$ or $P_{n-1}$,
 we can apply \cref{lem:good vectors} and \cref{prop:char2 list} together with \cite[Theorem A, Tables 1, 2 \& 3]{LiebeckPraegerSaxl}  to  provide a contradiction as $a=|\Sigma| \ge \ell$. This shows that $m= 3$.

 So assume $m=3$. Consider the subgroup $J $ of $G$ isomorphic to the Frobenius group of order $21$ as a subgroup of the fixed points of $\sigma$ (which is isomorphic to $\SL_3(2)$).  Then $W$ is   irreducible as an $\mathbb K J$-module. If necessary, extend $\mathbb K$ to a field $\mathbb G$ such that $O_7(J)$ acts on $W$ as a diagonal matrix group and fix a generator $x$ of $O_7(J)$. Thus $x$ has eigenvalues $\theta$, $\theta ^2$ and $\theta^4$ where $\theta$ is a non-trivial $7$th root of unity. Hence, as $J$ is centralized by $\sigma$, we know that $$W^{\sigma_k}\otimes_{\mathbb K}\mathbb G = W \otimes_{\mathbb K}\mathbb G$$ as  $\mathbb G J$-modules.  As a $\mathbb G J$-module, $U$ is an $\ell$-fold tensor product of $ W \otimes_{\mathbb K}\mathbb G$.

On $$(W\otimes_{\mathbb K}\mathbb G) \otimes (W\otimes_{\mathbb K}\mathbb G),$$ $x$ has eigenvalues $\theta$, $\theta^2$ and $\theta^4$ all with multiplicity $1$, $\theta^3$, $\theta^5$ and $\theta^6$ all with multiplicity $2$. Writing $Y= W\otimes_{\mathbb K}\mathbb G$ we further deduce that $$Y\otimes Y\cong Y^*\oplus Y \oplus Y^*$$ as an $\mathbb GJ$-module. Suppose that $\ell=|\Theta| \ge 3$. Then $U$ contains a $\mathbb GJ$-submodule isomorphic to $Y \otimes Y^*$ and hence $J$ has a non-zero fixed point $w $ on $Y$. We now apply \cite[Theorem A, Tables 1, 2 \& 3]{LiebeckPraegerSaxl} and \cref{lem:good vectors,prop:char2 list} to obtain a contradiction.

Suppose then that $\ell=2$.  Then $U= W^{\sigma_1}\otimes W^{\sigma_2}$. In particular, $\dim_{\mathbb K} U=9$.
(In this case,   $J$ has no non-zero fixed points on $U$.)

Recall that $\mathbb L=\mathrm{End}_{G}(V)$ and  $\mathbb K \ge \mathbb L$.
  Let  $\langle \sigma_0\rangle=C_\Sigma(\mathbb L) =\mathrm{Gal}(\mathbb K/\mathbb L )$. Then $\sigma_0$ has order $|\mathbb L :\mathbb K |$. If $\mathbb K = \mathbb L$, then $V=U$ and \cref{lem:no tensors} applies to reveal a vector $w\in V$ such that $C_G(w)$ leaves invariant  two distinct $1$-spaces $\langle w_1\rangle$ and $\langle w_2\rangle$ in $W$ where $w_1$ and $w_2$ are $\mathbb K$-linearly independent. Since $H$ normalizes a subgroup of order $q^{2}+q+1$ by \cref{prop:char2 list} (i), and has order coprime to $2$, and $C_G(w)$ does not contain a Sylow $2$-subgroup of $G$ (as it fixes two $\mathbb K$-linearly independent $1$-spaces), we see that $G\ne HC_G(w)$, a contradiction. Therefore $\mathbb K > \mathbb L$.
By \cite[(26.3)(3)]{AschbacherFG}, we have $U \cong U^{\sigma_0}$.

As $\ell=2$, we   have $\sigma_1=\sigma^{a_1}$ and $\sigma_2=\sigma^{a_2}$ so that $$U= W^{\sigma^{a_1}} \otimes W^{\sigma^{a_2}}.$$
Then, as $U \cong U^{\sigma_0}$, Steinberg's tensor product theorem yields $W^{\sigma^{a_1}} \cong W^{\sigma^{a_2}\sigma_0} $ and
$W^{\sigma^{a_2}} \cong W^{\sigma^{a_1}\sigma_0} $. It follows that $\sigma_0$ has order $2$. Hence $|\mathbb L :\mathbb K |=2$. Let $X$ be the unitary subgroup of $G$ as in \cref{unitary sbgrp}. Then $$(W^{\sigma^{a_1}})^*\cong  W^{\sigma^{a_1}\sigma_0} \cong W^{\sigma^{a_2}}$$ as $ \mathbb K X$-modules. Thus $U \cong W^{\sigma^{a_1}}\otimes(W^{\sigma^{a_1}})^*$. It follows that $X$ has fixed vectors on $U$, and as $X$ is not contained in any parabolic subgroups of $G$, we argue  that $(G,H,V)$ is not immutable using \cref{lem:good vectors} and this is our contradiction.

 Now we are left with $|\Theta|=\ell =1$. Hence  $U=W^{\sigma_1}$,  $V= U$ and, as $G$ is transitive on $V^\#$, so is $H$.  This is our claim.
\end{proof}

\begin{lemma}\label{case (ii)(a)}
Assume that \cref{the new list} (ii)(a) holds. Then  $\overline V= W^{\theta}$ for some $\theta\in \Sigma$.  In particular, $V$ is a natural $\GF(2)G$-module and $H$ acts transitively on $V^\#$.
\end{lemma}

\begin{proof} In this case, we have $ G=\Sp_{2n}(q)$, $n \ge 2$ ,$(n,q) \ne (2,2)$ and  $\overline W= V(\omega_1)$. Set $\mathbb K= \GF(q)$.  We know    $$\overline{V}= \bigotimes_{\theta \in \Theta}\overline  W^{\theta}$$ where $\Theta \subseteq \Sigma$.
Note that $\overline W= W\otimes_{ \mathbb K} \overline{\mathbb L}$ where $W$ is the natural $\mathbb KG$-module.
So   $U=\bigotimes_{\theta \in \Theta} W^{\theta}$ is a $\mathbb KG$-module and $\overline V= U\otimes_{\mathbb K}\overline {\mathbb L}$ as in \cref{lem:more factors impossible}.

 If $|\Theta|=1$, then $H$ acts transitively on the non-zero vectors of $V$ and this is our conclusion. So suppose that $|\Theta|\ge 2$.

 Set $L \cong \Sp_2(2)\times \Sp_{2n-2}(q)$ and regard $L$ as a subgroup of $G$ with the first factor centralized by $\sigma$. Notice that
$L$ is not contained in a  conjugate of $P_1$. However, as in the proof of \cref{lem:more factors impossible}, $L$ fixes a non-zero vector of $w \in U$. Again we obtain a contradiction via \cite[Theorem A, Tables 1, 2 \& 3]{LiebeckPraegerSaxl},  \cref{lem:good vectors,prop:char2 list}.
\end{proof}

\begin{lemma}\label{spin symplectic}
Assume that \cref{the new list} (ii)(b) holds.  Then either $n=2$ and $G$ acts transitively on $V^\#$ or $n = 3$, $\overline V=\overline W^{\sigma_1}$, $G \cong \Sp_6(q)$ and $ H \cong \Omega_6^-(q){:} 2$.
 Furthermore,   $G$ has $q$ orbits on the non-zero vectors of $V$. These orbits consist of \begin{enumerate}\item one orbit  of length $(q^{3}+1)(q^2+1)(q^2-1)$, with  point stabiliser  $O^{2'}(P_3)$ of shape $ q^{3+3}{:}\SL_3(q)$; and\item $q-1$   orbits of length $q^{3}(q^{4}-1)$ each  with point
 stabiliser $\mathrm G_2(q)$.\end{enumerate}
\end{lemma}

\begin{proof} In this case, we have $ G=\Sp_{2n}(q)$, $n \ge 2$ ,$(n,q) \ne (2,2)$ and  $\overline W= V(\omega_n)$. When $n=2$,  $\overline W$ is quasiequivalent to $V(\omega_1)$ and so we may apply \cref{case (ii)(a)} to obtain the claim.  So assume that $n \ge 3$. Set  $\mathbb K=\GF(q)$. Let $S \in \syl_2(P_n)$, $B=N_{P_n}(S)$ be a Borel subgroup in $P_n$ and assume that these subgroups are $\Sigma$-invariant. Furthermore recall that $C_V(S)$ is $P_n$-invariant.

Assume that $\ell=1$. Then $\overline V= \overline W^{\sigma_1}$ and $\overline W$ can be realized over $\mathbb K$ and we write this module as $W$. Then $W$ has dimension $2^n$ over $\mathbb K=\mathbb L$ and we have $V=W$. This module  has $\GF(2)$-dimension $2^na$. We first show that in this case we obtain our claimed immutable example.

We start by considering $\Sp_6(q)$. Then $W$ is   $8$-dimensional  over $\mathbb K$. The high weight space is fixed by the parabolic subgroup $P_3$, $P_3$ has shape $q^{3+3}{:}\SL_3(q).(q-1)$ and acts transitively on the non-zero vectors in $C_W(S)$. (To see this, notice that  $C_W(S) \le C_{W}(O_2(P_1))$. Then by \cite[Theorem 6.1]{PR03} this subspace is a spin-module for $O^{2'}(P_1/O_2(P_1)) \cong \Sp_4(q)$. Thus $P_1 $ acts transitively on the non-zero vectors in $C_{W}(O_2(P_1))$ and, in particular,  all the non-zero vectors in $C_W(S)$ are in the same $G$-orbit.) This accounts for $(q^3+1)(q^2+1)(q^2-1)=q^7+q^4-q^3-1$ vectors in $W$.

Consider the   subgroup $M=\G_2(q)$. This is maximal in $\Sp_6(q)$ by \cite[Table 8.29]{BHRoney} and we can assume that it is normalized by $\Sigma$.

Since (by \cite[Theorem 4.4 and Appendix A.49]{Lubeck} for example) the only non-trivial irreducible $\GF(q)M$-modules of dimension at most $8$ is one of the $6$-dimensional modules and since $W \cong W^*$, $M$ centralizes a $1$-dimensional subspace of $W$. Since $M$ is a maximal subgroup of $G$, we see that $M$ is the stabilizer of the $q-1$ non-zero vectors in this $1$-space. This then contributes  $(q-1)q^3(q^4-1)=q^8-q^7-q^4+q^3$ vectors to our count of vectors. Together we now have  all the vectors, we know that the vector stabilizers for the action of $\Sp_6(q)$ on $W$ are conjugates of $P_3$ and $M$. Finally, using \cite[Theorem A, Tables 1 \& 2  ]{LiebeckPraegerSaxl} shows that taking $H \cong \mathrm \Omega_6^-(q){:}2$, $(G,H,W)$ is immutable. This is the example in our conclusion.

Now suppose that $G = \Sp_{2n}(q)$ with $n \ge 4$. Then $W$ is   a $2^n$-dimensional spin module over $\mathbb K$.  Let $P_{n-3}$ be the $\Sigma$-invariant parabolic subgroup  containing $B$ with $\Sigma$-invariant Levi complement isomorphic to $L_{n-3}=\GL_{n-3}(q) \times \Sp_6(q)$ and let $ U_{n-3}$ be the unipotent radical of $P_{n-3}$.  Then $P_{n-3}$ is a maximal subgroup of $G$.   Set $W_{n-3}= \langle C_W(S)^{P_{n-3}}\rangle$.  Then $W_{n-3}$ is centralized by $U_{n-3}$ and the left-hand factor $J$ of $L_{n-3}$ which is isomorphic to $\GL_{n-3}(q)$. Hence $W_{n-3}$ is a $P_{n-3}/C_{P_{n-3}}(W_{n-3})\cong \Sp_6(q)$ spin module of $\mathbb K$-dimension $8$ (see \cite[Theorem 6.1]{PR03}). As in (ii), select $w$ in $W_{n-3}$ so that $C_{P_{n-3}}(w)/U_{n-3}J \cong \G_2(q)$ is $\Sigma$-invariant and set $K = C_{P_{n-3}}(w)$.

As $(G,H,W)$ is immutable,  \cref{prop:char2 list} (iv) gives $H \cong \Omega_{2n}^-(q){:}2$ and $G= HC_G(w)$. Let $N$ be a maximal subgroup containing $C_G(w)$ and hence $K$.  Then $G= HN$.  The candidates for $N$ are included in \cite[Theorem A, Tables 1 \& 2  ]{LiebeckPraegerSaxl}. This yields $N$ is isomorphic to one of

\begin{enumerate}
\item [(a)]$P_n$,
\item [(b)]$\Sp_{2c}(q^b).b$ where $bc=n$ and $b$ is a prime,
\item [(c)]$\Sp_n(q)\wr 2$ when $n$ is even,
\item [(d)]$\Omega_{2n}^+(2):2$ when $q=2$.
\end{enumerate}

Now notice that the Sylow $2$-subgroup of $K$ has index $q^3$ in a Sylow $2$-subgroup of $P_{n-3}$ and so the $2$-part of $|K|$ is $q^{n^2-3}$. Since the $2$-part of the orders of $\Sp_{2c}(q^b).b$ and  $\Sp_n(q)\wr 2$ are at most $2(q^b)^{c^2}$ and $2q^{n^2/4+ n^2/4}$ respectively, cases (b) and (c) do not give overgroups of $C_G(w)$.

Suppose that $K$ is contained in a parabolic subgroup $P=P_n$ of $G$. Then $K$ normalizes $O_2(P)$ and therefore $KO_2(P)$ is a group and $K \cap O_2(P)\le  U_{n-3}$. Then $U_{n-3}O_2(P)$ is a $2$-group and so $O_2(P)$ normalizes   $U_{n-3}$ by \cite[Lemma 4.2]{Grodal}. In particular, $O_2(P)\le P_{n-3}=N_G(U_{n-3})$.  Since $K$ is maximal in $P_{n-3}$, and $O_2(P)K$ is a group, $O_2(P) \le K$. Since $K$ normalizes $O_2(P)$, $O_2(P)\le U_{n-3}$.  Applying  \cite[Lemma 4.2]{Grodal} again yields $P_{n-3}$ normalizes $O_2(P)$ and therefore the maximality of $P_{n-3}$ yields $P_{n-3}= P$.  Since $P  \ne P_{n-3}$, we have a contradiction. Hence (a) does not hold.

Finally consider the possibility that $q=2$ and that $K$ is contained in $\mathrm O_{2n}^+(2)$. Using $2$-parts of the orders of both groups we obtain   $2^{n^2-n+1} \ge 2^{n^2-3}$ and so $n=4$. In this case a {\sc Magma} calculation yields $G$ has $4$ orbits on $W$ whereas $H$ has $5$. We have demonstrated that if $\ell=1$ and $(G,H,V)$ is immutable, then $G\cong \Sp_6(q)$ and $V$ is $8$-dimensional over $\mathbb K$.

Assume now that $\ell \ge 2$. Hence $V$ is isomorphic to a submodule of $U=\bigotimes_{\theta \in \Theta} W^{\theta}$ where $U= V \otimes_{\mathbb L} \mathbb K$.

Suppose that $n \ge 4$ and consider the vector $w$ with centralizer $K$ as above. Then, as $K$ is $\Sigma$-invariant, $w\otimes w\otimes \dots\otimes w\in U$ is fixed by $K$. \cref{lem:good vectors} now shows that $(G,H,V)$ is not immutable. Hence $n=3$.

  Assume $\ell \ge 3$. If $ U$ is realizable over a proper subfield of $\mathbb K$, then the field of definition of $U$, lies somewhere between $\mathbb K$ and $\GF(q^{1/\ell})=\GF(2^{a/\ell})$.  Hence $$|V| \ge |U|^{\frac 1 \ell}= \left( q^{(2^3)^\ell}\right)^{\frac 1 \ell}= q^{\frac{2^{3\ell}}{\ell}}.$$
Using the notation from \cref{lem:dimbound}, \cite[Theorem 4.2]{GuralnickSaxl} gives $\alpha(G)=7$ and \cite[Corollary 16]{ParkerWilson} yields $|G|\le q^{21}$. Since $G$ has no regular orbits on $V$ by \cref{lem:reg orb compendium}, we now use  \cref{lem:dimbound} and the fact that the function $\frac{2^{3\ell}}{\ell}$ increases with $\ell$ to obtain the impossible
$$170\le \frac{512}{3}\le {\frac{2^{3\ell}}{\ell}} \le 21.\alpha(G)=
3.7^2= 147.$$
This contradiction shows that  $\ell=2$. Now let $L \cong \Sp_6(2)$ be centralized by $\Sigma$. Then $W^{\sigma_1} $ and $W^{\sigma_2}$ are isomorphic and self-dual as $\GF(2)L$-modules.  It follows that $L$ fixes a vector $U|_L= W\otimes W\ge \mathrm{End}_G(U)$. By \cite[Theorem A, Tables 1 \& 2  ]{LiebeckPraegerSaxl} there is no factorisation $G=HM$ with $M \ge L$ (as $G$ is perfect) and so \cref{lem:good vectors} yields the result.  This completes the proof of the lemma.
\end{proof}

\begin{remark}\label{rmk G2} Suppose that $G = \Sp_6(q)$, and $V$ is the spin module for $G$. Regard $V$ as a $\GF(q)$-module. Let $G>M \cong \G_2(q)$. Then $M$ centralizes a $1$-space in $V$. Notice that $M$ does not centralize a $2$-space. Indeed, if $M$ centralizes a subspace $U$ of dimension $2$, then $U \cap Ux=0$ for $x \in G\setminus M$ as $M$ is maximal in $G$. Thus $$|\{Ug\mid g \in G\}|= q^{3}(q^{4}-1)(q^{2}-1)>q^{8}=|V|$$ which is impossible. Hence $\dim C_V(M)=1$. Consequently, $V/C_V(M)$ is indecomposable and, as the module is self-dual, we conclude that, as an $M$-module, $V$ is indecomposable with factors of $\GF(q)$-dimensions $1$, $6$ and $1$.
\end{remark}

\begin{lemma}\label{unitary gone} Assume that
 \cref{the new list} (iii) occurs. Then $n=2$, $G \cong \SU_4(q)\cong \Omega_6^-(q)$,  $\overline V=\overline W^{\sigma}$ for some $\sigma \in \Sigma$, $V$ is the natural $\GF(2)\Omega_6^-(q)$-module and the orbits of $G$ on $V$ are described in  \cref{Stabilizers of orbit reps}.
\end{lemma}

\begin{proof} Set $\mathbb K = \GF(q^2)$.  Here we have $G/Z(G)\cong \PSU_{2n}(q)$ with $n \ge 2$,  and $\overline W= V(\omega_n)$ is the  $n$th  exterior power of the natural $\overline{\mathbb L}\SU_{2n}(q)$-module. By \cref{lem:unitary dim bound}, we may realize $\overline W$ over $\GF(q  )$  and we write $W$ for this module.

  We begin   by assuming that $|\Theta|=1$.
Then $W= V$ regarded as an $\mathbb LG$-module where now $\mathbb L=\GF(q)$.

For $S \in \Syl_2(P_n)$, we know that $C_W(S)$ in normalized by $P_n$. We also know that $P_n$ stabilizes a totally isotropic $n$-space $I$ on the natural module $N$ for $G$. Remember that $N$ is defined over $\mathbb K$.  Furthermore, $L_n$ a Levi complement in $P_n$ acts on $N$ such that $N=I\oplus J$ with $J$ an isotropic $n$-space opposite  $I$.  Now, using
\cite[(B.1) in Section B.2, page 473]{FultonHarris} as a $\mathbb K L_n$-module we have
  $$W\otimes _{\mathbb L}\mathbb K= \Lambda^n(I\oplus J)= \bigoplus_{a=0}^n (\Lambda^a (I) \otimes \Lambda^{n-a}(J)).$$  Taking $a=0$ and $a=n$ reveals that $W$ has a $\mathbb KL_n$-submodule
  $$T=(\Lambda^0 (I) \otimes \Lambda^{n}(J))\oplus (\Lambda^n (I) \otimes \Lambda^{0}(J))$$ which is a direct sum of two $1$-spaces and we deduce that $L=O^{2'}(L_n)\cong \SL_n(q^{2})$ centralizes $T$.  Now $N_G(L_n)$ permutes $\{I,J\}$ and so $|N_G(L_n):L_n|=2$ and $N_G(L_n)$ acts irreducibly on $ N$.
Now choose $K\le N_G(L_n)$ such that $K \ge L$ and $K/L \in \syl_2(N_G(L_n)/L)$. Then $K$ acts irreducibly on $N$ and so $K$ is not conjugate to a subgroup of $P_n$.  Furthermore, in its action on $W$, $K $ normalizes $T$ and as $K/L$ is a $2$-group, $C_T(K) \ne 0$.
Let $M \ge K$ be a maximal subgroup chosen so that $M$ contains the stabiliser of a non-zero  vector of $W\otimes _{\mathbb L}\mathbb K$ fixed by $K$. Since $W\otimes _{\mathbb L}\mathbb K$ regarded as a $\mathbb L$-module is a direct sum of modules isomorphic to $W$, we see that $C_W(K) \ne 0$. Therefore, as $(G,H,V)$ is immutable, $G=MH$ and we know $H=N_1$.  Using \cite[Theorem A, Tables 1, 2 \& 3]{LiebeckPraegerSaxl} yields either
\begin{enumerate}
\item [(a)] $M \cong \PSp_{2n}(q)$; or
\item[(b)] that $a=1$ and $n \ge 3$.
\end{enumerate}
In case (a), we know that $M \ge K \ge L \cong \SL_n(q^2)$.
We know by \cite[Satz 1 \& 3]{HuppertSinger} that $L$ has a cyclic subgroup of order $(q^{2n}-1)/(q^2-1)$ while the every irreducible  cyclic subgroup of $M$ has order dividing $q^n+1$. Hence in case (a) we must have $n=2$. In this case $L \cong \Omega_4^-(q)$ and so is a candidate subgroup of $\Sp_4(q)$ as $q$ is even.  Now we identify $G$ with $\Omega_6^-(q)$ and $W$ with that natural $\mathbb K\Omega_6^-(q)$-module and obtain the example as claimed.

Thus case (b) holds with $G/Z(G) \cong \PSU_{2n}(2)$, $n \ge 3$. Combining \cref{lem:dimbound,lem:unitary dim bound} shows that $n \le 6$.  Thus $G/Z(G)=\PSU_6(2)$, $\PSU_8(2)$, $\PSU_{10}(2)$ or $\PSU_{12}(2)$. In the first case, we can calculate directly using {\sc Magma} that $G$ has $6$ orbits on $W$, whereas $H$ has $9$ orbits. In the remaining cases, we can choose a prime $p$ such $p$ does not divide $|H|$ (the pairs $(2n,p)$ are $(8,17), (10,31), (12,13)$). Let $R \in \syl_p(G)$. Then a random check with {\sc Magma} shows that $G$ is generated by two conjugates of $R$.

Employing \cref{SU2nhelp} implies that $\dim C_W(R) \ge \log_2(|W|/|H|)$. Set $\ell_{2n}= \log_2(|W|/|H|)$. Then $\ell_{8}=19$, $\ell_{10}\ge 169$ and $\ell_{12}\ge 801$.
 Now the second clause of \cref{SU2nhelp}, shows that $2n=8$.  Finally, we check that $\dim C_W(R)=6$ when $2n=8$.  We have demonstrated that $(G,H,W)$ is not immutable.

So suppose that $|\Theta|\ge 2$.  Then $a> 1$ and so $q>2$.
When $n \ge 3$, taking $w\in T^\# \subset W^\#$ fixed by $K$ as in the previous paragraph and for all   maximal subgroups $M$ containing $K$, we have that $G\ne HM$.
Writing  $U= \bigotimes_{\theta \in \Theta} W^{\theta}$, we have  $w\otimes \dots \otimes w\in U$ is also fixed by $K$
as $K$ can be chosen to be stabilised by $\Sigma$. But $G\ne HM$, and we obtain a contradiction from \cref{lem:good vectors}. Thus $n=2$. In this case,  we again regard $G$ as $\Omega_6^-(q)$ and $W$ as that natural $\mathbb K\Omega_6^-(q)$-module. Now let $L \le \Omega_4^+(2)\times \Omega_2^-(q)$  with the first factor centralized by $\Sigma$ be such that $L \cong \SU_2(2)\times \Omega_2^-(q)$ where $\SU_2(2)\cong \SL_2(2)$ acts with two non-trivial $2$-dimensional composition factors on the natural $\mathbb K \Omega_4^+(2)$-module. Then as in \cref{lem:more factors impossible}, we see that $L$ fixes a non-zero vector on  $U=\bigotimes_{\theta \in \Theta} W^{\theta}$. Keeping in mind that $q>2$ and $G$ is simple, we obtain a contradiction using \cite[Theorem A, Tables 1, 2 \& 3]{LiebeckPraegerSaxl},   with \cref{lem:good vectors,prop:char2 list}.
\end{proof}

\begin{lemma}\label{Orthogonals and natural} If \cref{the new list} (iv) holds, then  $n \ge 5$, $\epsilon 1= (-1)^n$, $G\cong \Omega_{2n}^\epsilon(q)$, $H$ normalizes  $ \SU_n(q)$, $ \overline V=\overline W^{\sigma}$ for some $\sigma\in \Sigma$ and $V$ is the natural $\GF(2)G$-module. The orbits are presented in \cref{Stabilizers of orbit reps}.
\end{lemma}

\begin{proof}  \cref{the new list} (iv) states that $G \cong \Omega_{2n}^\epsilon(q)$ with $n \ge 5$, $\epsilon 1= (-1)^n$,   and $\overline{W}= V(\omega_1)$ is the natural $\overline{\mathbb L}G$-module.  Let $\mathbb K= \GF(q)$, $\mathbb L=\mathrm{End}_G(V)$. We have $G \cong \Omega_{2n}^\epsilon(q)$ with $n \ge 5$, $\epsilon 1=(-1)^n$,   and  $\overline{W}=V(\omega_1)$.
We have $$\overline{V}= \bigotimes_{\theta \in \Theta} \overline W^{\theta}$$
and note that $\overline W$ can be realized over $\mathbb K$ as $W$. Thus we set
  $$U=\bigotimes_{\theta \in \Theta} W^{\theta}$$ as a $\mathbb KG$-module.
By \cref{prop:char2 list} (vi) and (ix), we have $G= P_1H$  with $H$ normalizing a subgroup $\SU_n(q)$.

Suppose that $|\Theta|\ge 2$ and consider the subgroup $$L \cong \SU_2(2) \times \Omega_{2n-4}^\epsilon(q) \le \Omega_{4}^+(2)\times \Omega_{2n-4}^\epsilon(q) .$$ Then the first factor $\SU_2(2)\cong \Sym(3)$ and is centralized by the standard Frobenius automorphisms of $G$ and  so, as $|\Theta|\ge 2$, our familiar argument from \cref{lem:more factors impossible} shows that $L$ has a fixed point on $U$. Since $L$ is not contained in $P_1$, we must have some maximal factorisation $HM$ with $M \ge L$. Using \cite[Theorem A, Tables 1 \& 2  ]{LiebeckPraegerSaxl} shows that $M \cong N_1$.  Since $L$ does not fix any $1$-spaces  in $W$,  we have a contradiction.  Hence $|\Theta|=1$.  Now $V=U=W$. Since  $G$ and $H$ have the same number of orbits on $V$, we have that $(G,H,V)$ is immutable in this case and the details of the orbits are given in \cref{Stabilizers of orbit reps}.
\end{proof}

\begin{lemma}\label{casev}    \cref{the new list} (v) cannot hold.
\end{lemma}

\begin{proof} Set  $\mathbb K=\GF(q)$. We have $G \cong \Omega_{2n}^+(q)$, $n \ge 5$ and we may as well assume that $\overline W=V(\omega_{n})$ with high-weight space invariant under $P_n$. Thus $\overline W$ is a half-spin module for $G$ and this module can be defined over $\mathbb K$.

We start by recalling that the orbits of $\Omega_8^+(q)$ on the non-zero vectors of the $8$-dimensional half-spin module have stabilisers $O^{2'}(P_3)$ (or $O^{2'}(P_4)$) and $ \Sp_6(q)\cong \Omega_7(q)$. (These are the same as on the natural module, by application of the triality automorphism.)

We intend to show that $\overline V=\overline W$, leads to $(G,H,V)$ not being immutable. The argument is reminiscent of that in \cref{spin symplectic}. Suppose that $\overline V=\overline W$. Then as $W$ can be regarded as $\mathbb KG$-module and $\mathbb K=\mathbb L$, we can assume that $V= W$ where $V$ is now regarded as a $\mathbb LG$-module.

Let $P_{n-4}$ be  from the conjugacy class of parabolic subgroups which stabilize an $(n-4)$-dimensional totally singular subspace on the natural module. Then $P_{n-4}$ is a maximal subgroup of $G$ and has  Levi complement isomorphic to $L_{n-4}=\GL_{n-4}(q) \times \Omega_8^+(q)$.  Let $U= U_{n-4}$ be the unipotent radical of $P_{n-4}$.   We may assume that $P_{n-4}$ contains the same Borel subgroup as $P_n$ and that $P_n$ fixes the high weight space $W_n$ of  dimension $1$.  Set $W_{n-4}= \langle W_n^{P_{n-4}}\rangle$.  Then $W_{n-4}$ is centralized by $U_{n-4}$ and the left hand factor $J$ of $L_{n-4}$ which is isomorphic to $\GL_{n-4}(q)$. Hence $W_{n-4}$ is a $P_{n-4}/C_{P_{n-4}}(W_{n-4})\cong \Omega_8^+(q)$ half-spin module of $\mathbb K$-dimension $8$ (see \cite[Theorem 6.1]{PR03}). Select $w$ in $W_{n-4}$ so that $C_{P_{n-4}}(w)/U_{n-4}J \cong \Sp_6(q)$ and set $K= C_{G}(w)$.

Using  \cref{prop:char2 list} (vii) and (viii) gives
$$H \cong \begin{cases} N_1\cong \Sp_{2n-2}(q)\\N_2^- \cong  \mathrm O_{2n-2}^-(q)\times \mathrm O_2^-(q) \cap G \end{cases}$$ and as $(G,H,V)$ is immutable, $G= HK$. Let $M$ be a maximal subgroup containing $K$.  Then $G= HM$.  The candidates for $M$ are included in \cite[Theorem A, Tables 1 \& 2  ]{LiebeckPraegerSaxl} noting that $G$ is perfect. Thus, not taking into account the structure of $K$, the initial options for the pair $(H,M)$ are as follows:
\begin{enumerate}
\item[(a)] $(N_1,P_j)$ with $j \in \{n,n-1\}$.
\item[(b)] $(N_1,{} \GU_n(q).2)$ with $n$ even.
\item[(c)] $(N_1,\Sp_2(q) \otimes \Sp_n(q))$ with $n$ even and $q> 2$.
\item[(d)] $(N_1, \GL_n(q).2)$  with $n$ even.
\item[(e)]  $(N_1,\Omega_n^+(4).2^2)$ with $n$ even and $q=2$.
    \item[(f)]  $(N_1,\mathrm{Co}_1)$ with $n=12$   and $q=2$.
 
 \item[(g)]  $(N_2^- ,P_j)$ with $j \in \{n,n-1\}$.
 \item[(h)] $(N_2^-, \GL_n(2).2)$, $n$ even and $q=2$.
      \end{enumerate}

      Now notice that the Sylow $2$-subgroup of $K\cap P_{n-4}$ has index $q^3$ in a Sylow $2$-subgroup of $P_{n-4}$ and so the $2$-part of $|K|$ is at least $q^{n(n-1)-3}$.  Considering the Sylow $2$-subgroups of $M$ in all but cases (a) and (g) yields a contradiction. Thus $K$ is contained in a parabolic subgroup $P$ of $G$.    Thus $K$ normalizes $O_2(P)$ and therefore $(K\cap P_{n-4})O_2(P)$ is a group. In particular, $X=U_{n-4}O_2(P) $ is normalized by $K\cap P_{n-4}$. If $O_2(P) \not \le U_{n-4}$, then $N_X(U_{n-4})>U_{n-4}$. As $P_{n-4}= N_G(U_{n-4})$ and $K\cap P_{n-4}$ is a maximal subgroup of $ P_{n-4}$, we obtain a  contradiction as $K \cap P_{n-4}$ normalizes $N_X(U_{n-4})$. Therefore $O_2(P) \le U_{n-4}$. Applying  \cite[Lemma 4.2]{Grodal}   yields $P_{n-4}$ normalizes $O_2(P)$ and therefore the maximality of $P_{n-4}$ yields $P_{n-4}= P$.  Finally noting that $G= HK= HP_{n-4}$  gives a contradiction. Hence $(G,H,V)$ is not immutable when $V=W$.

 Now assume that $\Theta=\{\sigma_1, \dots, \sigma_\ell\}\subseteq \Sigma$ with $\ell \ge 2$ and that $$U= W^{\sigma_1}\otimes\dots \otimes W^{\sigma_\ell}.$$
Now take $w\in W$ fixed by $K$. Then $w\otimes \dots \otimes w$ is also fixed by $K$ and we obtain a contradiction via \cref{lem:good vectors}.
\end{proof}

\begin{lemma}\label{OrthRed}
Assume that \cref{the new list} (vi) holds. Then $\overline V=\overline W^{\sigma}$ for some $\sigma\in \Sigma$.  In particular, $V$ is the natural   $\GF(2)G$-module.
\end{lemma}

\begin{proof}
We have that $G =\Omega_8^+(q)$ and  $W$ is the natural   $\mathbb KG$-module where $\mathbb K=\GF(q)$.  Assume $|\Theta|=\ell\ge 2$ and set
$U= W^{\sigma_1} \otimes \dots\otimes W^{\sigma_\ell}$. Then $V$ is a summand of $U$ when $U$ is considered as a $\GF(2)$-module and so, as in the last paragraph of \cref{spin symplectic}, $|V|\ge  q^{\frac{8^\ell}{\ell}}$. Using \cref{lem:dimbound} we have $|V|\le |G|^{\alpha(G)}$ and by \cite[Corollary 16]{ParkerWilson} (using  $2\le q$)
  and \cite[Theorem 4.2]{GuralnickSaxl} we have $|G| \le q^{ 29.8}= q^{232}.$
Hence $$8^\ell \le 232\ell$$ which indicates that $\ell \le 3$. Furthermore, if $\ell=3$, then $V$ has field of definition $\mathbb L=\GF(2^{a/3})$.

Suppose that $\ell=3$.  Take $L={}^3\mathrm D_4(q)$. Then $L= C_G(\kappa \tau)$ where $\kappa \in \Sigma$ has order $3$ and $\tau$ is the triality automorphism which commutes with $\Sigma$.
We know that $$U= W^{\sigma_1}\otimes W^{\sigma_2}\otimes W^{\sigma_3}$$ and since $U$ has   field of definition $\mathbb L=\GF(2^{a/3})$, we have $U^\kappa=U$. It follows that $\sigma_2=\sigma_1\kappa$ and $\sigma_3=\sigma_1\kappa^2$ and so $$U= (W\otimes W^\kappa \otimes W^{\kappa^2})^{\sigma_1}.$$ Hence we may just consider $V=W\otimes W^\kappa \otimes W^{\kappa^2}$. By   \cite[(6.1)]{AschbacherMultLin}, $L$ preserves a Hermitian symmetric $3$-sesquilinear form $\mathcal H$ on $W$. Since the vector space of $3$-sesquilinear forms on $V$ is isomorphic to the dual of the vector  space $W\otimes W^\kappa \otimes W^{\kappa^2}$ and $W$ is self-dual, $L$ fixes a vector in $W\otimes W^\kappa \otimes W^{\kappa^2}$ \cite[Proposition 1.3', page 14]{Yokonuma} and so we must have $G=HL$. However, by \cite[Table 8.50]{BHRoney}, $L$ is maximal in $G$ and by \cite[Theorem A, Table 4]{LiebeckPraegerSaxl} we have a contradiction as no factorisations involve $L$.

Thus $\ell=2$ and we argue that $\Omega_{2n}^+(2)$ fixes a non-zero vector on $W^{\sigma_1}\otimes W^{\sigma_2}$ for a contradiction to \cite[Theorem A, Table 4  ]{LiebeckPraegerSaxl}. We have proved $\ell=1$ and the result follows.
\end{proof}

We now address the question of whether the natural $\Omega_8^+(2^a)$-module can be part of an immutable triple. Notice that in the next lemma $L$ is not a maximal subgroup.

\begin{lemma}\label{lem:casevi} Suppose that  \cref{the new list} (vi) holds with $V$ the natural module. Assume that $\tau$ is the triality automorphism of $G$. Then taking $L= (N_2^-)^\infty$. Then $(G,L^\tau,V)$ and $(G,L^{\tau^2},V)$ are immutable.
\end{lemma}

\begin{proof} Since $V$ is a natural module for $\Omega_8^+(q)$, we know that $G$ has $2$ orbits on the $1$-spaces of $W$ and $q+1$ orbits on vectors. The stabilizers are $O^{2'}(P_1)$ and $N_1\cong \Sp_6(q)$.

We have $L\cong \Omega_6^-(q)\cong \SU_4(q)$ and $L^\tau$ and $L^{\tau^2}$ act irreducibly on $V$.  We regard $V$ as the natural $\SU_4(q)$-module defined over $\GF(q^2)$ (so the spin modules for $\Omega_6^-(q)$). Hence $L^\tau$ and $L^{\tau^2}$ also have $q+1$ orbits on vectors, namely the isotropic vectors, and the vectors of a non-singular subspace which take different values in $\GF(q)$ with respect to the Hermitian form. Since the number of orbits is the same, we have $(G,L^\tau,V)$ and $(G,L^{\tau^2},V)$ are immutable.
\end{proof}

\begin{corollary}\label{cor:vi1} Let $G=\Omega_8^+(q)$ and $V$ be the natural module for $G$. Assume that $\tau$ is the triality automorphism of $G$. Set $H= N_1\cong \Sp_6(q)$.  Then $(G,H^\tau,V)$ and $(G,H^{\tau^2},V)$ are immutable.
\end{corollary}

\begin{proof}
Since $L=   (N_{2}^-)^\infty\le N_1=H$, and $(G,L^\tau,V)$ and $(G,L^{\tau^2},V)$ are immutable by \cref{lem:casevi}  so are $(G,H^\tau,V)$ and $(G,H^{\tau^2},V)$ by \cref{lem:reg orb compendium} (iii).
\end{proof}

\begin{corollary}\label{cor:vi2}
Let $G=\Omega_8^+(q)$ and $V$ be the natural module for $G$. Assume that $\tau$ is the triality automorphism of $G$. Set $H= N_2^-\cong \GU_4(q).2$.  Then $(G,H^\tau,V)$ and $(G,H^{\tau^2},V)$ are immutable.
\end{corollary}

\begin{proof} This is almost the same as for \cref{cor:vi1}.
\end{proof}

 Bringing the results of this section together we obtain the following result.

\begin{proposition}\label{prop:classical} Suppose that \cref{hypss7} holds. Then one of the following holds.
\begin{enumerate}
\item $G$ acts transitively on the non-zero vectors of $V$.
\item $G \cong \Sp_{6}(q)$, $H \cong \SO_6^-(q)$ and $  V$ is a spin module for $G$.
\item $G \cong \Omega_{2n}^\epsilon(q)$, $n \ge 3$ with $\epsilon 1= (-1)^n$, $H$ normalizes $\SU_{n}(q)$ and $V$ is a natural module for $G$.
\item $G \cong \Omega_8^+(q)$, $H \cong \Sp_6(q)$ and $V$ is a natural module for $G$.
\item $G \cong \Omega_8^+(2)$, $H \cong \Alt(9)$ and $V$ is a natural module for $G$.
\end{enumerate}
\end{proposition}

\begin{proof} Because of \cref{lem:no factSL2}, we may assume that $G \not \cong \SL_2(q)$ and so  $G$  is defined in dimension at least $3$. Then \cref{prop:char2 list} lists the pairs $(G,H)$ that we need to consider.  \cref{odd balls} shows that if any of the ``exceptional" cases arise then its due to $G\cong \Omega_8^+(2)$ acting on one of its half-spin modules and $H$ being an irreducible $\Alt(9)$-subgroup. Using triality, we may transform the module to the natural module and these examples are  listed in case (v). With these cases considered, \cref{the new list} determines the candidates for the module $V$ for an immutable triple $(G,H,V)$. This leads to six cases which are then considered in the subsequent lemmas.  In particular, \cref{lem:more factors impossible,case (ii)(a)} give rise to instances where $H$ is transitive on the non-zero vectors of $V$ and cover the possibilities given in \cref{the new list} (i) and (ii)(a).  \cref{the new list} (ii)(b) is examined in \cref{spin symplectic} and this results in the case listed in (ii). By \cref{unitary gone},
\cref{the new list} (iii) leads to $n=2$, and $G \cong \SU_4(q)\cong \Omega_6^-(q)$ acting on the natural orthogonal module. This is included as the $n=3$ case of part (iii).
For \cref{the new list} (iv), \cref{Orthogonals and natural} yields the triples in (iii) when $n \ge 5$.  Case (v) of \cref{the new list} gives no examples of immutable triples by \cref{casev}. Finally, considering  \cref{the new list} (vi), we have $G \cong \Omega_8^+(q)$. In this case, we first show that $V$ can be taken to be the natural module in \cref{OrthRed} and then \cref{cor:vi1,cor:vi2} provide the final example in part (iii) with $n=4$ and case (iv).
\end{proof}

We can now collect together the results needed for the proof of \cref{Main theorem,cor1,Cor:Main Result}.
\begin{proof}[Proof of Theorem~\ref{Main theorem}] Suppose that $G$ is a quasisimple group, $H$ is a maximal subgroup of $G$ and $V$ is a non-trivial irreducible $\GF(2)G$-module. Suppose that $H$ is not transitive on $V^\#$. If $G/Z(G)$ is an alternating group, then \cref{prop:altgrps} provides the example in case Theorem~\ref{Main theorem} (vi). For $G/Z(G)$ not a classical group defined in characteristic $2$,
\cref{cor:no sporadics,thm: Exceps,prop:no Lie Odd} show that no examples of immutable triples arise. Finally, when $G/Z(G)$ is a classical group defined in characteristic $2$, \cref{prop:classical} provides all the other examples in \cref{Main theorem}.
\end{proof}

\begin{proof}[Proof of \cref{cor1}]  In cases (i)-(v) of \cref{Main theorem}, $G$ has just $2$ orbits on the $\mathbb KG$ $1$-spaces of $V$. As $G$ has at least $3$-orbits, on the $\mathbb KG$ $1$-spaces of $V$, \cref{Main theorem} (vi) holds and an elementary {\sc Magma} calculation yields the result.
\end{proof}

\begin{proof}[Proof of \cref{Cor:Main Result}] Set $q=2^a$. We apply \cref{Main theorem}. If either case (v) or $G \cong \Omega_8^+(2)$ holds, then we simply check all subgroups of $G$ for immutability and obtain the stated result enumerated in (iii) or (iv).

Suppose that \cref{Main theorem} (i), (ii) or (iii) hold. Then, setting $\mathbb K=\mathrm{End}_G(V)$, we have $\mathbb K= \GF(q)$ and $V$ is the natural $\mathbb KG$-module for $G\cong \Omega^{\epsilon}_{2n}(q)$ with $\epsilon 1= (-1)^n$.
In particular, on the $1$-dimensional $\mathbb K$-spaces in $V$, $G$ has orbits of length
\begin{enumerate}
\item[(a)] $(q^n-1)(q^{(n-1)}+1)/(q-1)$ and $q^{(n-1)}(q^{n}-1)$ when $n$ is even.
\item[(b)]   $(q^{n}+1)(q^{(n-1)}-1)/(q-1)$ and $ q^{(n-1)}(q^{n}+1) $ when $n$ is odd.
\end{enumerate}
Since $(G,K,V)$ is immutable, \cref{same point orbit} implies that $G$ and $K$ have the same orbits on the $1$-dimensional $\mathbb K$-spaces in $V$. As the orbits \cite[Table 1 (2) and (3) with $q$ even]{Liebeck} are the same as (a) and (b) respectively, we may apply \cite[Theorem]{Liebeck} to obtain (i), (ii) and (iii). Notice that \cref{Main theorem} (iv) is subsumed in (ii) as $\Sp_6(q)$ action on a spin module extends to a natural module for $\Omega_8^+(q)$. Applying \cref{Main theorem} one final time, we see that all the candidates for $K$ arising from  \cite[Theorem]{Liebeck}  indeed yield immutable triples.
\end{proof}

\section{The proof of \cref{thm:Theorem B}: part 1}\label{sec:ThmB1}

In this section, we prove \cref{thm:Theorem B} for all the candidates for $G$ which are not symplectic groups and do not have centre of even order. Our first goal is to prove \cref{prop:trans cases}. Thus we intend to list the triples $(G,H,V)$ with $H$ maximal in $G$ and $H$ transitive on the non-zero vectors of $V$. First we require a lemma which may be well-known.

\begin{lemma}\label{Normal basis lemma}
Assume that $a$ and $n$ are natural numbers, $p$ a prime, $q=p^a$, $H=\Gamma\mathrm L_1(q^n)$ regarded as a subgroup of $\GL_n(q)$ and $V$   the natural $n$-dimensional $\GF(q)H$-module.  The subgroup $H_1$ of determinant $1$ elements of $H$ acts transitively on the non-zero elements of $V$ if and only if $q= 2$ or $n$ is even and $q=3$.
\end{lemma}

\begin{proof} Let $\Sigma=\mathrm{Gal}(\GF(q^n)/\GF(q))$ be generated by $\sigma$ which maps $x\in \GF(q^n)$ to $x^q$.
Identify $V$ with $\GF(q^n)$ regarded as a $\GF(q)$ vector space.  Then by the well-known Normal Basis Theorem \cite[Lemma 2]{JacobsonIII}, there exists $\beta \in V$ such that $\{\beta^{\sigma^{i}}\mid 0\le i\le n-1\}$ is an ordered basis for $V$.  With respect to this ordered basis $\sigma$ is represented by the matrix $$S=\left(\begin{smallmatrix}
0&1&0&0&\dots&0\\
\vspace{-1mm}
0&0&1&0&\dots&0\\
\vspace{1mm}
\vdots&\vdots&\vdots&\vdots&\vdots&\vdots\\
0&0&0&0&\dots&1\\
1&0&0&0&\dots&0\end{smallmatrix}\right)$$
which has order $n$.

Recall that the embedding of $\GF(q^n)^\times$ into $\GL_n(q)$ arises from the assignment of $\alpha\in \GF(q^n)$ to the transformation $T_\alpha: x\mapsto x\alpha$.  Take $\alpha$ to be a primitive element of $\GF(q)$. Then $\gen{T_\alpha}$ has order $q^n-1$ in $\GL_n(q)$. Now $\det (T_\alpha)$ is just the norm of $\alpha$ (see \cite[Proposition VI.5.6]{Lang}).  Thus
 $$\det (T_\alpha)= \prod_{i=0}^{n-1} \alpha^{\sigma^i}= \prod_{i=0}^{n-1} \alpha^{q^i}= \alpha ^{(q^n-1)/(q-1)}$$ and so $\det(T_\alpha)$ has order $q-1$ in $\GF(q)$.
In particular, $H_1$ has order $\frac{q^n-1}{q-1}.n$

 Suppose first that $n$ is odd or that $n$ is even and that $p=2$. Then  $\det S=1$.  Hence  $\sigma \in H_1$. Since $\sigma$ fixes the non-zero vector $v=\sum_{i=0}^{n-1}\beta^{\sigma^{i}}$, we see that $|vH_1|= \frac{q^n-1}{q-1}$. Thus $|vH_1|=q^n-1$ if and only if $q=2$.

Now suppose that $n$ is even and $p$ is odd.  Then $\det S=-1$ and $\sigma^2 \in H_1$. Since $\gen{\sigma}$ is the stabilizer of $v$ in $H$, $\gen{\sigma^2}$ is the stabilizer of $v$ in $H_1$. Therefore, as $|H_1|= \frac{q^n-1}{q-1}.n$, we deduce that
  $$vH_1=|H_1:\gen{\sigma^2}|=  2\frac{q^n-1}{q-1}.$$ Thus in this case,  $|vH_1|=q^n-1$  if and only if  $q=3$. This completes the proof.
\end{proof}

\begin{proposition}\label{prop:trans cases} Suppose that $G$ is a quasisimple group and $V$ is a non-trivial $\GF(2)G$-module.  Let $H$ be a maximal subgroup of $G$ and assume that $(G,H,V)$ is immutable with $H$ transitive on $V^\#$. Then one of the following holds.
\begin{enumerate}
\item  $G=\SL_n(2^a)$, $H$ normalizes $K=\SL_{n/b}(2^{ab})$ where $b$ is a prime which divides $n$ with $b<n$;
    \item  $G=\SL_n(2 )$, $H =\Gamma \mathrm L_1(2^n)$ with $n$  an odd prime;
\item   $G= \SL_{2m}(2^a)$, $H$ normalizes $K=\Sp_{2m}(2^a)$;
\item  $G=\SL_{4}(2)$, $H=\Alt(7)$;
\item   $G= \Sp_{2n}(2^a)$, $H$ normalizes $K=\Sp_{2n/b}(2^{ab})$ where $b$ is a prime which divides $n$;
\item  $G =\Sp_6(2^a)$, $H$ normalizes $K=\G_2(2^a)$;
\item  $G=  \Alt(7)$, $\dim V=4$, $H\cong \Sp_4(2)'\cong \Alt(6)$;
    \item  $G=  \Alt(7)$, $\dim V=4$, $H \cong \SL_2(4){:}2 \cong \Sym(5)$; or
\item  $G=  \Alt(6)$, $\dim V=4$, $H \cong \SL_2(4)$ acting naturally on $V$.
\end{enumerate}
Furthermore, in cases (i) to (vi), $V$ is the natural module for $G$.
\end{proposition}

\begin{proof}
This uses results of Huppert and Hering
 from \cite{Huppert, Hering} (see \cite[Section 7]{HBIII} or \cite[Appendix 1]{LiebeckAffine}).  Plainly, since $H$ is transitive on $V^\#$, so is $G$.  Hence $G$ has to be a quasisimple group listed in  \cite[Appendix 1]{LiebeckAffine} and $V$ is required to be a module defined in characteristic $2$.  Hence $G$ is one of $\SL_n(2^a)$, $\Sp_{2n}(2^a)$, or $\mathrm G_2(2^a)'$ for some $a$ and $n \ge 2$ or is an exceptional case $G\cong \Alt(7)$ or $\Alt(6)$ with $\dim V=4$.

 Suppose that $G \cong \SL_n(2^a)$. Then $\dim_{\GF(2)} V= na$. We apply \cite[Appendix 1]{LiebeckAffine} again to determine the candidates for $H$. If $H$ normalizes a linear group, then we must have $H$ normalizes $\SL_{n/b}(2^{ab})$ and the maximality of $H$ forces $b$ to be a prime dividing $n$. Now in the special case that $n=b$, then we notice that $H$ is in case A(1)  of \cite[Appendix 1]{LiebeckAffine} but it is also contained in $\SL_n(2^a)$. Hence \cref{Normal basis lemma} implies $q=2$. These two possibilities give cases (i) and (ii) of the theorem. If $H$ normalizes a symplectic subgroup, then surely $n$ is even and we have case (iii). Since the normalizer of
$\mathrm G_2(2^a)'$ in $\SL_6(2^a)$ is contained in $\Sp_6(2^a)$, there are no possibilities with $H$ normalizing such a subgroup.  If $H = \Alt(7)$, then $n=4$ and this is listed in (iv). Since $\Alt(6)$ is not maximal in $\SL_4(2)$, this completes the analysis of the linear case for $G$.

Suppose that $G\cong \Sp_{2n}(2^a)$. Assume that $H$ normalizes a linear group $\SL_m(2^c)$ with $m \ge 3$.   Then $H$ has an irreducible cyclic subgroup of order $  \alpha=(2^{mc}-1)/(2^c-1)$ whereas the maximal order of an irreducible  cyclic subgroup of $G$ is $\beta=2^{na}+1$ by \cite[Satz 5]{HuppertSinger}. The transitivity of $G$ and $H$ yields $2^{2na}-1= 2^{mc}-1$.
Hence $\alpha = (2^{2na}-1)/(2^c-1)= (2^{na}+1)(2^{na}-1)/(2^{2na/m}-1)$. As we require $\alpha \le \beta$ and as $m\ge 3$, we have a contradiction. Since $\SL_2(2^c)$ is a symplectic group, when this subgroup is maximal we have listed this in case (iv). If $H \le \Gamma\mathrm L_1(2^c)$ then  $c=2n$ and $H \le \Sp_{2n}(2^a) \le \SL_{2n}(2^a)$. Now \cref{Normal basis lemma} implies that $c=1$. But then $H$ contains a cyclic subgroup of order $2^{2n}-1$, again contradicting  \cite[Satz 5]{HuppertSinger}.

Continuing with $G\cong \Sp_{2n}(2^a)$, if $H$ normalizes a subfield  subgroup $\Sp_{n/b}(2^{ab})$, then $H$ is transitive on the non-zero vectors and for $H$ to be maximal we require $b$ to be a prime. These examples appear  in (v).

Finally, for the symplectic groups we have $\G_{2}(2^a)$ acts transitively on the non-zero vectors  of the $6a$-dimensional natural module.  As $\G_2(2^a)$ is maximal in $\Sp_{6}(2^a)$  this is (v).

Now suppose that $G=\G_{2}(2^a)$. Since $G< \Sp_{6}(2^a)$, the possible transitive subgroups have to normalize  one of $\SL_2(2^{3a})=\Sp_{2}(2^{3a})$, $\SL_3(2^{2a})$ but these groups are not subgroups of $G$.

Next, we consider $G=\Alt(7)$. Here $G$ has transitive subgroups $\Alt(6)$ and $\SL_2(4){:}2\cong \Sym(5)$. Finally, we could have $G \cong \Alt(6)$ and then we obtain $H\cong \SL_2(4)$ and that is the lot. These three examples are listed in parts (vii), (viii) and (ix).
 \end{proof}

We now commence to prove \cref{thm:Theorem B} for groups which are not symplectic.  We start with   ``oddities" which arise when $G$ can be identified with an alternating group. Throughout, we use \cref{lem: reduct to irred,lem:trivial factors,lem:min exactly 2} to determine the composition factors that we need to study as well as the minimal configurations to be examined.

\begin{lemma}\label{lem:alt(6) done} Suppose that $G \cong \Sp_4(2)'\cong \Alt(6)$. Assume that $H$ is a maximal subgroup of $G$ and that $V$ is a $\GF(2)G$-module with no centralized direct summands. Then $V$ is the natural $4$-dimensional $\Sp_4(2)'$- module, the natural $\Omega_5(2)$-module or the dual of the natural $\Omega_5(2)$-module. Furthermore, in all cases $H \cong \SL_2(4)\cong \Alt(5)$.
\end{lemma}

\begin{proof} Using \cref{Main theorem} and \cref{prop:trans cases} we see that the faithful composition factors of $V$ are isomorphic to either $U_1$ or $U_2$ where $U_1$ and $U_2$ are the $4$-dimensional (quasiequivalent) irreducible modules for $G$.
 We check using {\sc Magma} that the stated examples are indeed examples. We then construct all the $\GF(2)\Alt(6)$-modules up to dimension 9, remove those with a centralized direct summand and calculate the orbits of maximal subgroups on the remaining modules. This just yields two irreducible modules and four modules of dimension $5$. This is sufficient to  demonstrate our statement.
\end{proof}

\begin{lemma}\label{Alt(7) gone}
Suppose that $G \cong \Alt(7)$. Assume that $H$ is a maximal subgroup of $G$ and that $V$ is a $\GF(2)G$-module with no centralized direct summand. Then $V$  is irreducible of dimension $4$.
\end{lemma}

\begin{proof} This time we have that the faithful composition factors for $G$ are $4$-dimensional. Using {\sc Magma} we check that $V$ is completely reducible and then that a direct sum of two irreducible $4$-dimensional modules gives no immutable examples. This proves the claim.
\end{proof}

\begin{lemma}\label{lem:A8-A7}
Suppose that $G=\SL_4(2)\cong \Omega_6^+(2)\cong \Alt(8)$, $H  $ is a maximal subgroup of $G$ and $V$ is a $\GF(2)G$-module with no centralized direct summand. Denote the natural $4$-dimensional $\GF(2)G$-modules  by  $F$ and $F^*$.    If $(G,H,V)$ is immutable, then $V$ is isomorphic to one of $F$, $ F\oplus F$, $ F\oplus F\oplus F$, $ F\oplus F^*$, $ F\oplus F\oplus F^*$, or the dual of such a module.  Furthermore, if $V$ is irreducible, then $H \cong \GL_2(4).2 \cong (3\times \Alt(5)){:}2$, $\Sp_4(2)\cong \Sym(6)$ or $\Alt(7)$, and otherwise $H\cong\Alt(7)$ and the number of orbits of $G$ on $V$ is $2$ when $\dim V=4$, $5$ when $\dim V= 8$ and $16$ when $\dim V=12$.
\end{lemma}

\begin{proof} We know that $G$ acts  transitively  on the faithful composition factors of $V$ by \cref{Main theorem} and \cref{prop:trans cases}. Moreover, \cref{prop:trans cases} (i), (ii), (iii) and (iv) yield $H \cong \GL_2(4){:}2$, $\Sp_4(2)$ or $\Alt(7)$.
Now assume   that $V$ is not irreducible and that is has no centralized direct summand.
Then the faithful composition factors of $V$ are either the $4$-dimensional module $F$ or its dual $F^*$. We know from {\sc Magma} that $\mathrm H^1(G,F)= \mathrm H^1(G,F^*)=0$ and so $V$ only has faithful composition factors. A {\sc Magma} calculation also shows that $\mathrm{Ext}^1(F,F ) =0$ and $\dim \mathrm{Ext}^1(F,F^*)=1$. The non-split extension of $F$ by $F^*$ produces no instances of an immutable triple. So we only need to consider direct sums of copies of $F$ and of $F^*$. Now we simply check that all possibilities for direct sums of two or three modules give immutable examples (the orbits are of lengths $5$ for a sum of two modules and $16$ for three). Furthermore, in all cases, we have $H \cong \Alt(7)$.
 Now  modules which are a sum of at least four faithful $4$-dimensional modules have $67$ orbits  for $G$ and $74$ for $H$. This proves the lemma.
\end{proof}

\begin{lemma}\label{lem:A9 done}
Suppose that $G\cong \Alt(9)$, $V$ is a $\GF(2)G$-module with no centralized direct summands. If $(G,H,V)$ is immutable, then $V$ is the irreducible fully deleted permutation module.
\end{lemma}

\begin{proof} We know from \cref{Main theorem,prop:trans cases} that the composition factors of $V$ are trivial modules and   copies of the fully deleted permutation module $W$. We also know $H \cong \mathrm P\GammaL_2(8)$. A {\sc Magma} calculation shows that $\mathrm H^1(G,W)$ is trivial and so  we may assume that $V$ has no trivial composition factors. Thus,  a minimal counter example will have two composition factors. So we may assume that $V$ has dimension 16. We calculate using {\sc Magma} that $\mathrm {Ext}^1(W,W)=0$ and so $V \cong W\oplus W$. Now calculating with {\sc Magma}  we see that $(G,H,V)$ is not immutable.
\end{proof}

We continue with a final exceptional case.
\begin{lemma}\label{lem:Omega82-A9 done}
Suppose that $G\cong \Omega_8^+(2)$, $V$ is a $\GF(2)G$-module with no centralized direct summands. If $(G,H,V)$ is immutable, then (up to triality) $V$ is the natural module for $G$.
\end{lemma}

\begin{proof} We know that every non-trivial composition factor of $V$ is quasiequivalent to the natural $\Omega_8^+(2)$-module.
The proof is a {\sc Magma} calculation. We determine all three $8$-dimensional modules and check that there are no non-split extensions between them. We also check that there are no extensions with trivial module. In the end we consider a direct sum of two $8$-dimensional modules and this delivers no immutable triples.
\end{proof}

In model, the  cases above identify what has to be done in general. However, though $\mathrm H^1(G,W)$ is known for all the candidates for faithful composition factors $W$ that we encounter, the dimensions of the spaces $\mathrm{Ext}^1(W,W)$ are not generally available and this is why we need more extended and case specific arguments. Having said this, we note that in the next case \cite{Bell} describes $\mathrm{Ext}^1(W,W)$ for $W$ the natural $\GF(2)\SL_n(2^a)$-module.

\begin{lemma}\label{lem:not linear}
Assume that $(G,H,V)$ is immutable, $G \cong \SL_{n}(2^a)$, $n \ge 3$. Assume that $V$ has no centralized direct summand. Then either $V$ is irreducible or $G \cong \SL_4(2)$.
\end{lemma}

\begin{proof} Assume that $G \not \cong \SL_4(2)$. Noticing that $\Omega^+_6(2^a)$ is not listed in \cref{Main theorem} and using \cref{prop:trans cases}, we deduce that all the faithful composition factors of $V$ are natural, or dual natural modules for $G$. Furthermore, the candidates for $H$ are provided by \cref{prop:trans cases}.

A {\sc Magma} check   shows that $G \not \cong \SL_3(2)$.

We claim   that $V$ is irreducible. Suppose this is false. Then by \cref{lem: reduct to irred,lem:trivial factors,lem:min exactly 2}, we may assume that $V$ has exactly two  composition factors. By  \cite[Tables B and C]{JonesParshall}, as  $G \not \cong \SL_3(2)$, we deduce that   $V$ has no trivial composition factors. Thus we may assume that $\dim V= 2na$.

Let $W$ be a proper submodule of $V$. Thus $ W$ and $V/W$ are natural or dual natural $\GF(2)G$-modules.

Set $\overline{V}= V/W$ and let $v \in V \setminus W$. Then every vector of $V \setminus W$ is mapped into $\overline v=v+W$ by some element of $G$.  Let $P$ be the subgroup of $G$ which fixes $\overline{v}$ and define $Q=O_2(P)$. Then, as $G \not \cong \SL_3(2)$, $P$ is perfect, $P/Q \cong \SL_{n-1}(2^a)$, $Q$ regarded as a $\GF(2)P/Q$-module is dual to the module $\overline{V}/C_{\overline V}(P)$ and has order $2^{a(n-1)}$. Let $U=\langle v,W\rangle$. Then $\dim U=na+1$. We intend to show that we can choose $v$ to be centralized by $P$.

Suppose that $\overline V \cong W$.  Then $Q$ acts on $U/C_W(P)$ where $\dim C_W(P)=\dim C_{\overline V}(P)=1$. If this action is non-trivial, then  $[U/C_W(P),Q]=W/C_W(P) $ is isomorphic to $Q$ as a $\GF(2)P/Q$-module by \cref{lem:End bound}. However, these modules are dual to each other. Hence $[U,Q]\le C_W(P)$.  As $G \not\cong \SL_4(2)$, \cite[Tables B and C]{JonesParshall} shows that   $U/C_W(P)$ splits as a direct sum of a $1$-space $C/C_W(P)$ and $W/C_W(P)$. But then, as $P$ is perfect, $C$ is centralized by $P$ and consequently we may assume that $v$ is centralized by $P$.

Now consider   $\overline V \cong W^*$. Set $Y=[W,P]$ and note $\dim W=(n-1)a$. Then $Y\cong Q$ as a $\GF(2)P/Q$-module and $Q$ centralizes $Y$.  As $P$ is perfect, $P$ centralizes $U/Y$.  Since $\mathrm {End}_G(V)\cong \GF(2^a)$, and $|U/Y|=2^{a+1}$, we must have $C_U(Q)> Y$ by \cref{lem:End bound}.  By  \cite[Tables B and C]{JonesParshall} again,   $C_U(Q)$ splits as a direct sum with $Y$ and once again we may assume that $v$ is fixed by $P$.

If $W \cong \overline V^*$, then $W$ contains a submodule of dimension $a(n-1)$ on which $P$ induces the natural $ \SL_{n-1}(2^a)$-module whereas when $W \cong \overline V$, there is a quotient of $W$ on which  $P/Q$ acts naturally. In particular, in the first case we can choose $w \in W$ so that $|wP|= 2^{a(n-1)}-1$ and in the second case we can choose $w$ to have $|wP|$   a multiple of $ (2^{a(n-1)}-1) $. By \cref{lem:vec stabiliser}, $C_G(v+w)= C_P(w)$. It follows that $|G:C_P(v+w)|$ is divisible by $(2^{a(n-1)}-1)(2^{an}-1)$. Therefore, immutability implies that  $|H|$ is a multiple of $(2^{a(n-1)}-1)(2^{an}-1)$ and this is impossible as $H$ is known by  \cref{prop:trans cases} (i), (ii) and (iii).
\end{proof}

\begin{lemma}\label{lem:O8 all the same}
 Suppose that $G \cong \Omega_{8}^+(2^a)$.  Assume that $H$ is a maximal subgroup of $G$ and that $V$ is a faithful $\GF(2)G$-module which has no centralized direct summands. If $(G,H,V)$ is immutable, then all the faithful composition factors of $V$ are isomorphic.
\end{lemma}

\begin{proof} Suppose false. By \cref{Main theorem,lem: reduct to irred}, the non-trivial composition factors of $V$ are all $8a$-dimensional and are either natural modules or half-spin modules. By     \cite[Tables B and C]{JonesParshall}, we see that all the trivial modules split from such composition factors and as $V$ has no centralized direct summands, all the composition factors of $V$ are $8a$-dimensional. By \cref{lem:min exactly 2,lem: reduct to irred}, we may assume that $V$ has exactly two faithful composition factors. Hence $\dim V=16a$. Furthermore, as the claim is false, we may assume that the composition factors are not isomorphic. Thus there is $W< V$   such that $V/W$ is a natural module and $W$ is a half-spin module. Let $v \in V\setminus W$ be such that $v+W$ is a singular vector in $V/W$ and let $P$ be the centralizer of $v+W$.  Then $P$ has shape $2^{6a}{:}\Omega_{6}^+(2^a)\cong 2^{6a}{:}\SL_4(2^a)$. Set $U=\langle v,W\rangle$ of dimension $1+8a$. Then, as $U$ is a half-spin module, the composition factors for $P$ on $U$ have dimension $1$, $4a$ and $4a$. Applying \cref{lem:End bound} twice together with the fact from \cite[Tables B and C]{JonesParshall} that $\mathrm H^1(P/O_2(P),N)=0$ for $N$ the natural $\SL_4(2^a)$-module or its dual,  we  see that $C_U(P) \ne 0$ and $U= C_U(P)+W$.  Hence we may select $v$ to be centralized by $P$. Now picking an element $w$ of $C_W(O_2(P))^\#$ we see that $C_P( v+w)= C_P(w)$ by \cref{lem:vec stabiliser}. Hence, identifying $P$ as a subgroup of the parabolic subgroup $P_1$, we have that $C_P(w)$ is conjugate to a subgroup of  $P_1 \cap P_3$ or $P_1\cap P_4$. This   contradicts \cref{lem:unique para}.
\end{proof}

\begin{lemma}\label{lem:orthgone} Suppose that $G \cong \Omega_{2n}^\epsilon(2^a)$ with $\epsilon 1 =(-1)^n $.
 Assume that $H$ is a maximal subgroup of $G$ and that $V$ is a faithful $\GF(2)G$-module with no centralized direct summands. If $(G,H,V)$ is immutable, then $V$ is irreducible.
\end{lemma}

\begin{proof} By \cref{Main theorem,lem:O8 all the same}, we may assume  the faithful composition factors of $V$ are natural modules. Furthermore, by \cite[Tables B and C]{JonesParshall} we know that there are no trivial composition factors (notice that   $G \not \cong \Omega_6^+(2)$ as $\epsilon 1 = (-1)^n$).
Hence to prove the lemma, by \cref{lem: reduct to irred} we may as well assume that $V$ has a submodule $W$ with $W $ and $ V/W$ both natural modules.  Let $v +W\in V/W$ be a non-zero singular vector and let $P\le G$ be its stabilizer.  Then $P $ has shape $q^{2n-2}{:}\Omega_{2n-2}^\epsilon (2^a)$ (see \cite[Proposition 4.1.20]{KleidmanLiebeck}).

Set $Q= O_2(P)$ and note that $P$ is perfect since $G \not \cong \Omega_6^+(2)$.
Now set $U=\langle v,W\rangle$ of dimension $2na+1$.  Then, as $P$ is perfect, $[U,P]=[W,P]$ has dimension $2(n-1)a$ and $W/C_W(P) \cong O_2(P)$ as $\GF(2)P/Q$-modules have dimension $(2n-2)a$. Since $\mathrm{End}_{P/Q}(W) \cong \GF(2^a)$, using \cref{lem:End bound} we deduce that $C_U(Q)/C_W(P)\not=0$ and $U= C_U(Q) +W$. Since, by \cite[Tables B and C]{JonesParshall}, $\mathrm{H^1}(P/O_2(P),[W,P]/C_W(P))=0$ whenever $(2n,\epsilon,2^a) \ne (8,+,2)$, there exists a $P$-submodule $C/C_W(P)$ centralized  by $P$ and such that $C_U(Q)/C_W(P)= C/C_W(P)\oplus W/C_W(P)$. Since $P$ is perfect, $P$ centralizes $C$. Hence we may select $v \in C_U(P)$.

By \cref{lem:vec stabiliser} $C_G(v+w)= C_P(w)$ for all $w \in W$ such that $w+C_W(P)$ is singular.  Then $C_P(w)$ centralizes $\langle C_W(P),w\rangle$ and so is contained in the parabolic subgroup conjugate to $P_2$. But there are no factorisations $P_2H$ and so we have a contradiction unless $G\cong \Omega_{8}^+(2)$. However, the omitted case is handled in \cref{lem:Omega82-A9 done} and this completes the proof.
\end{proof}

\begin{lemma}\label{lem:spin yields irreducible}
Suppose that $G \cong \Sp_6(2^a)$, $V$ is a faithful $\GF(2)G$-module with no  centralized direct summands and $H$ is a maximal subgroup of $G$.
Assume that $V$ has a composition factor which is a spin module. If $(G,H,V)$ is immutable, then $V$ is irreducible.
\end{lemma}

\begin{proof}   Suppose that $(G,H,V)$ is immutable. Assume that $U_1$ and $U_2$ are faithful irreducible sections of $V$ with $U_1$ a spin module. Then \cref{lem: reduct to irred} implies that $(G,H,U_1)$ and $(G,H,U_2)$ are both immutable. Since $U_1$ is a spin module, \cref{Main theorem} (iv) implies that $H \cong \SO_6^-(2^a )$.  Assume that $U_2$ is not a spin module.  Then $U_2$ is the natural $\GF(2)G$-module and $G$ acts transitively on $U_2^\#$ again by \cref{Main theorem}. Since $H$ has two orbits on $U_2^\#$, this is impossible. We conclude that all faithful composition factors in $V$ are spin modules.
 By \cite[Table B]{JonesParshall}, as $V$ has no trivial direct summands, every composition factor of $V$ is a spin module.

 It now suffice to show that $V$ cannot have two composition factors. So suppose that $V$ has dimension $16a$ and that $W$ is a submodule of dimension $8a$.  If $2^a=2$, we perform a {\sc Magma} calculation to verify the claim. Hence we may assume that $2^a\ge 4$.

 Then $\overline V= V/W$ and $W$ are spin modules. Let $v \in V \setminus W$ be such that $ v+W$ is stabilized by $P\cong \G_2(2^a)$ (see \cref{Stabilizers of orbit reps} (iv)). Then $P$ is perfect as $2^a\ge 4$. Set $U= \langle v,W\rangle$ of dimension $8a+1$.  Then the $\GF(2)P$ composition factors of $V$ are trivial, and the natural $\mathrm G_2(q)$-module $S$ of dimension $6a$.  Since $\dim C_W(P)=a$, $W/C_W(P)$ is indecomposable of dimension $7a$ with natural module socle (see \cref{rmk G2}). Appealing to \cite[Table B]{JonesParshall} we have $\dim \mathrm H^1(P,S)=a$ and consequently we may assume that $[v,P] \in C_W(P)$.  As $P$ is perfect, we conclude that $v$ is centralized by $P$.
Let $w \in W$ be chosen so that $w+C_W(P)$ has an orbit of length $2^{6a}-1$ in $W/C_W(P)$. By \cref{lem:vec stabiliser}, $C_G(w+v)=C_P(w)$ which has order dividing $2^{6a}(2^{2a}-1)$. Now $H$ is $\SO_6^-(2^a)$ and $|G:H|$ is divisible by $2^{3a}-1$. Thus there is no factorisation $C_G(v+w)H$, a contradiction.
\end{proof}

\section{The proof of \cref{thm:Theorem B}: symplectic groups}\label{sec:ThmBSymp}

Throughout this section $G$ will be a symplectic group. Let $G=\Sp_{2n}(2^a)$ with $n \ge 2$  and $a\ge 1$ with $G \not\cong \Sp_4(2)$.  We know from \cref{Main theorem,prop:trans cases} that if $V$ is a faithful irreducible $\GF(2)G$-module, and $(G,H,V)$ is immutable then either  $G$ acts transitively on $V^\#$ or $n=3$, $H \cong \SO_6^-(2^a)$ and $V$ is an $8a$-dimensional spin module for $G$. We have dispatched that latter case in \cref{lem:spin yields irreducible}. Thus in this section we handle only non-trivial composition factors which are natural modules.

We use \cref{prop:trans cases} (v) to extract a subtle point regarding the representations of $G=\Sp_4(2^a)$. Let $N$ be the natural module for $G$.
Now $G$ has a graph automorphism $\gamma$ which then gives rise to another $4$-dimensional $\GF(2^a)$-representation $M$ which we regard as a $\GF(2)G$-module.  Then, taking $(G,H,N)$ immutable, we have $(G,H^\gamma,M)$ is also immutable. Now here is the subtle point, $(G,H^\gamma,N)$ is not immutable. Plainly $H \cong H^\gamma$, but in fact $H^\gamma$ acts as  $\SO_4^-(2^a)$ and   has three obits on $N$. Thus we obtain the following lemma, which will be useful in the final argument of this section.

\begin{lemma}\label{lem: sp4 all natural}
Suppose that $G\cong \Sp_4(2^a)$, $H$ is a maximal subgroup of $G$ and $V$ is a faithful $\GF(2)G$-module. If  $(G,H,V)$ is immutable, then all the faithful composition factors are isomorphic.  In particular, we can assume that they are all natural modules.
\end{lemma}

\begin{proof} We know that $G$ acts transitively on each  faithful composition factor and that they have dimension $4a$. Assume that $U_1$ and $U_2$ are faithful composition factors.  Then $(G,H,U_1)$ and $(G,H,U_2)$ are both immutable by \cref{lem: reduct to irred}, now the above discussion shows that $U_1$ and $U_2$ are isomorphic.
\end{proof}

\begin{lemma}\label{lem:symplectic fact}
Suppose that $q=2^a$, $G \cong \Sp_{2n}(q)$ and $N$ is the natural module for $G$ regarded as a $\GF(2)G$-module. Assume that $X$ is a $\GF(2)G$-module with $[X,G]>C_X(G)$, such that $[X,G]/C_X(G)\cong N$.  Then $X$ is isomorphic to $A/B$  where $A$ and $B$ are submodules of the $\GF(2)\Omega_{2n+2}^+(q)$-module $W$ restricted to $\mathrm \Omega_{2n+1}(q)\cong G$.  Furthermore, the lengths of orbits of $G$ on $W$ and the stabilizer of a representative of each orbit are given in the following table.

\begin{table}[h]
\begin{tabular}{|ccc |}
\hline
 {\rm Orbit length}&{\rm Multiplicity}&{\rm Stabilizer}\\
 \hline
$1$&$q$&G\\
 $q^{2n}-1$&$q$&$O^{2'}(P_1)\sim q^{2n-1}{:}\Sp_{2n-2}(q)$\\
 $(q^n-1)q^n$&$q(q-1)/2$&$\Omega_{2n}^+(q)$\\
 $(q^n+1)q^n$&$q(q-1)/2$&$\Omega_{2n}^-(q)$\\
\hline
\end{tabular}\vspace{3mm}\caption{The orbits on $W$}\label{tab:orbs on W}
\end{table}
\end{lemma}

\begin{proof}
We know that $|\mathrm H^1(G,N)|=q$ from \cite[Table B]{JonesParshall}. Hence $|X| \le  q^{2+2n}$. Identify $G$ with  $\Sp_{2n}(q) \cong \Omega_{2n+1}(q)$ and recall that the natural $\GF(q)\Omega_{2n+1}(q)$-module is uniserial with a $1$-dimensional socle and $2n$-dimensional quotient   as a module for $\GF(q)\Omega_{2n+1}(q)$.

Let $H= \Omega_{2n+2}^+(q)$  and $W$ be
 the natural $\GF(q)H$-module.  Let $v \in W$ be a non-singular vector which is stabilized by  $G$. Set $Y_0=\langle v \rangle_{\GF(q)}$ and $Y_1=Y_0^\perp$. Then $Y_1$ is stabilized by $G \cong \Omega_{2n+1}(q)$.  Since $Y_0$ is isotropic, $Y_0<Y_1$ and $Y_1$ is a $\GF(q)$-hyperplane of $W$. Furthermore, $Y_1$ is a  natural $\GF(q)G$-module.  Now $W/Y_0 \cong  Y_1^*$ as a $\GF(q)G$-module and so  $W/Y_0  $ is uniserial with socle of dimension $2n$.
 Thus $W$ realizes the maximum possible size  of a module satisfying the hypothesis of the lemma. This proves that $X$ is a section of $W$.

 We now consider the orbits of $G$ on $W$. We first check by computer that the statements are true for $\Sp_4(2)$ and $\Sp_6(2)$. This means that for non-trivial $v+Y_0\in Y_1/Y_0$, we know that $C_G(v+Y_0)$ is perfect.

 As $G$ fixes every vector in $Y_0$, $G$ has $q$ orbits of length $1$ on $Y_0$. The subgroup $G$ acts transitively on  $Y_1/Y_0$ . Let $u+Y_0$ be in $Y_1/Y_0$.  Then the stabiliser in $L$ of $u+Y_0$ is perfect and so centralize $\langle u,Y_0\rangle$. Thus on the elements of $Y_1\setminus Y_0$, $G$ has  $q$ orbits each of length $q^{2n}-1$.

  Let $Z_+$ and be an non-degenerate $2$-space of $+$-type and $Z_-$  be a  non-degenerate $2$-space of $-$-type both containing $Y_0$. Write $V= Z_+\perp  Z_+^\perp= Z_-\perp  Z_-^\perp$. Let $\epsilon =\pm$ and define $H_\epsilon$ be the subgroup  of $H$ which preserves the decomposition $V= Z_\epsilon\perp  Z_\epsilon^\perp$.  As $G$ is perfect, every element of $G$ has determinant $1$.  Set $G_\epsilon =G\cap H_\epsilon$   Then $G_+\cong \Omega_{2n}^\epsilon(q){:}2\cong \SO_{2n}^\epsilon(q)$. Note that    the involutions in the outer half of $G_\epsilon$ have determinant $1$ and so are not reflections on $V$. In particular, on $Z_\epsilon$   they act with $C_{Z_\epsilon}(G_\epsilon)= Y_0$ and with orbits of length $2$ on the vectors of $Z_\epsilon\setminus Y_0$.  Thus for $w \in Z_\epsilon\setminus Y_0$ we have $C_G(w) \cong \Omega_{2m}^\epsilon(q)$.  In this way $Z_\epsilon$ accounts for $\frac{1}{2}(q^2-q)$ orbits each of length $(q^{n}-\epsilon 1)q^n$.  Since $$\frac{(q^2-q)}{2}\left((q^{n}- 1)q^n+ (q^{n}+ 1)q^n\right)=
  \frac{(q^2-q)}{2}q^n(q^n-1+q^n+1)=
  q^{2n+2}-q^{2n+1}.$$
  We have shown that the orbits of $G$ on $W$ are as described in \cref{tab:orbs on W}. That is $q$ of length $1$, $q$ of length $q^{2n-1}$, $\frac{1}2(q^2-q)$ each of length $(q^n- 1)q^n$ and $(q^n+ 1)q^n$.
  \end{proof}

\begin{notation}\label{not:symp mod} Suppose that $q=2^a$, $G=\Sp_{2n}(q)$ and $W$ be the natural $\GF(q)\Omega_{2n+2}^+(q)$-module. Then $\mathcal W$ is the $\GF(2)G$-module $W|_G$ as in \cref{lem:symplectic fact}. \end{notation}

  \begin{lemma}\label{lem:sp2nq immutable struct}
  Suppose that $G\cong \Sp_{2n}(2^a)$, and let $\mathcal W$ be the $\GF(2)G$-module described in \cref{not:symp mod}. Assume that $H$ is a maximal subgroup of $G$. The triple $(G,H,\mathcal W)$ is immutable if and only if  one of the following holds:
  \begin{enumerate}
 \item  there exists an odd prime $b$ dividing $n$ and $H$ is conjugate to $\Sp_{2n/b}(2^{ab}){:b}$;
 \item $n=3$ and $H \cong \mathrm{G}_2(2^a)$.
 \end{enumerate}
 Furthermore, if $n$ is even, and  $H$ is conjugate to  $\Sp_{n}(2^{2a}){:}2$, then  for each $B \le C_{\mathcal W}(G)$, there exists  a uniquely determined $[{\mathcal W},G] \le B^\dagger \le {\mathcal W}$ with $\dim B^\dagger /B=a(2n+1)$ such that the faithful section $V= A/B$ satisfies $(G,H,V)$ is immutable if and only if $[{\mathcal W},G]\le A\le B^\dagger$.
  \end{lemma}

  \begin{proof}  First of all we consider $G$ acting on ${\mathcal W}$. Then we note that by \cref{lem:symplectic fact,lem:Sp fatorisations1,lem:More Sp factorisations,lem:G2q factorisations},  $(G,H,{\mathcal W})$ is immutable if and only if  there exists an odd prime $b$ dividing $n$ such that $H$ is conjugate to $\Sp_{2n/b}(2^{ab}){:}b$; or
  $n=3$ and $H$ is conjugate to  $\mathrm{G}_2(2^a)$. It follows from \cref{lem: reduct to irred}  that every section of ${\mathcal W}$ is immutable in these cases.

So now assume that $(G,H,V)$ is immutable with $n$ even, and $H$ is conjugate to $\Sp_{n}(2^a){:}2$. Then by \cref{lem:symplectic fact,lem:Sp fatorisations1,lem:More Sp factorisations}, $(G,H,{\mathcal W})$ is not immutable and this failure of immutability is due to the fact that the orbits with stabiliser $\Omega_{2n}^+(2^a)$ split into two orbits for $H$.  Thus to achieve immutability, we need to double the size of the stabilizers.

 Assume that $B \le C_{\mathcal W}(G)$ and $V=A/B$ is a faithful section of ${\mathcal W}$ with $[{\mathcal W},G]\le A \le {\mathcal W}$.  The elements of $a \in A\setminus [{\mathcal W},G]$ with stabilizers $\Omega_{2n}^+(2^a)$ are the vectors whose stabiliser needs to grow in $A/B$.  Let $K$ be conjugate to $\Omega_{2n}^+(2^a)$ such that $a$ is centralized by $K$. Write ${\mathcal W}=Z \perp Z^\perp$ where $Z$ is $2a$-dimensional  centralized by $K$ and $Z^\perp=[\mathcal W,K]\le [\mathcal W,G]$. Then $a\in Z$ (as in the proof of \cref{lem:symplectic fact}).  Now $N_G(K) \cong \Omega_{2n}^+(2^a){:}2$ has an element $t$ of order $2$ which acts as a $\GF(2^a)$-transvection on both $Z$ and $Z^\perp$ and so we need to find sections where $t$ acts trivially. Note that $C_Z(t)= C_{\mathcal W}(G)$ and set $Z^\dagger$ to be the preimage of $C_{Z /B}(t)$. Then $Z^\dagger=\{z\in Z \mid[z,t]\le B\}$ and  $\dim Z /Z^\dagger=\dim C_{\mathcal W}(G)/B$.

  Now $a+B$ is centralized by $N_G(K)$ if and only if $[a,t] \in B$ which is if and only if $a \in Z^\dagger$.  Hence $(G,H,V)$ is immutable if and only if $[\mathcal W,G]\le A \le B^\dagger$ and this proves our claim.
  \end{proof}

  \begin{proposition}\label{prop:symplecticGroups} Suppose that $n \ge 2$, $G \cong \Sp_{2n}(2^a)$, $H$ is a maximal subgroup of $G$ and that $(n,2^a) \ne (2,2)$. Assume that $V$ is a $\GF(2)G$-module with no trivial direct summands. Then $(G,H,V)$ is immutable if and only if $V$ is a non-trivial section of $\mathcal W$  described in \cref{lem:symplectic fact}.
  \end{proposition}

  \begin{proof} Write $q=2^a$. By \cref{lem:spin yields irreducible,lem: sp4 all natural}, we may assume that every faithful composition factor of $V$ is a natural module for $G$.
  We first claim that $V$ has at most one faithful composition factor.  So, for a contradiction, assume that $V$ has at least two faithful composition factors.    Choose a section of minimal dimension that contains at least two faithful composition factors.  Then by \cref{lem: reduct to irred} and induction, we may as well suppose that this section is $V$.  Then $V=[V,G]$ and $C_V(G)=0$. Let $0<X<Y<V$ be such that $X$ is faithful and $V/Y$ is  faithful.  Then $X$ and $V/Y$ are natural $G$-modules and $Y/X$ is centralized by $G$.

   If $q=2$ and $n=3$, then we check using {\sc Magma}  that the only candidate for  $V$ is a direct sum of two natural modules.  Then we  just check that a direct sum of two natural modules does not produce immutable triples.  Thus $q> 2$ when $n=3$.
In particular, the stabilizer of a vector in the natural module has shape $q^{1+2(n-2)}{:}\Sp_{2n-2}(q)$ and is perfect.

   Since $V/Y$ and $X$ are  natural modules, $G$ acts transitively on the non-zero vectors of both $V/Y$ and $X$. Assume that $(G,H,V)$ is immutable. Thus $H\cong \Sp_{2n/b}(q^{ab}){:}b$ for some prime $b$ dividing $n$ or $2n=6$ and $H \cong \mathrm G_2(q)$. Let $\tau$ be an element of order $p$ where $p$ is a Zsigmondy prime divisor of $q^{2(n-1)}-1$. Then $p$ divides $q^{(n-1)}+1$. Notice that $p \ge 2a(n-1)+1>n$ and   so in particular, $p > b$ for all prime divisors $b$ of $n$.
   We know $X= C_X(\tau)\perp [X,\tau]$ where $C_X(\tau)$ has dimension $2a$ is non-degenerate. Suppose that $n \ge 3$. Then the Sylow $p$-subgroups of $G$ are cyclic.
    Suppose that $p$ divides $|H|$. Then  $\tau$ is conjugate to an element of $H$ and so to an element of $H'$. If $H'\cong \Sp_{2n/b}(q^b)$, then as $|C_X(\tau)|= 2^{2a}$ and $|C_X(\tau)|$ is a $\GF(q^b)$ subspace, we deduce that $b=2$. Then $C_X(\tau)$ is a $1$-space and so is isotropic with respect to the $\GF(q^b)$ form preserved by  $H'$. However, this means that $C_X(\tau)$ is isotropic as a $\GF(q)$-space; but it is not as it is non-degenerate. Hence $p$ does not divide $|H |$ in this case. If $H\cong \mathrm G_2(q)$, then $|H|=q^6(q^6-1)(q^2-1)$ and so $|H|$  and $p$ are coprime. Since $(G,H,V)$ is immutable, every vector of $V$ is centralized by some conjugate of $\tau$. We intend to prove something similar when $n=2$.  In this case we know that $H\cong \SL_2(q^2){:}2$ and $p$ divides $q+1$. Hence $p$ certainly divides $|H|$. However the elements of $H$ of order $p$ act fixed-point-freely on $X$ and so are not conjugate to $\tau$. It follows that $H$ does not contain a Sylow $p$-subgroup of $G$ and so if $v\in V$, then $C_G(v)$ must have order divisible by $p$. Consider $v \in V \setminus Y$ and let $P$ be the centralizer of $v+Y$. Then some element of $G$ of order $p$ centralizes $v$ and so centralizes $v+Y$. But then this element is conjugate into $P$ and we conclude that $v$ is centralized by a conjugate of $\tau$.  Thus, every element of $V\setminus Y$ is centralized by a conjugate of $\tau$.

 Let $v \in V \setminus Y$ and $P$ be the centralizer of $v+Y\in V/Y$ in $G$.  Set $J= \langle v,Y\rangle$ and $K=[X,P]$.  Then $\dim X/K = a$ and $\dim J/K=a+\dim Y/X+1$. Since $P$ centralizes $J/Y$, $Y/X$ and $X/K$, we have $[J,P,P,P]<K$  and so, as $P$ is perfect,  we deduce that $J/K$ is centralized by $P$.

 Now acting on  $J$, $P$ leaves invariant a chain of submodules $0<L<K<J$ with $\dim L=a$, and $K=L^\perp$ of dimension $(2n+1)a$.  Furthermore, $L$ and $J/K$ are centralized by $P$ and $Y/X$ is the natural $(2n-2)a$-dimensional module for $P/O_2(P)$. Set $Q= O_2(P)$.  Then $K\ge [J,Q]\ge [Y,Q]=K$, $[K,Q]=X$ and $C_Y(Q)=L$. Let $R= Z(P)$. Then $R$ has order $q$ with $[X,R]=L$ and $C_X(R)=K$.

 Let $r \in R$ be an involution. Set $\widetilde J=J/L$. Then $C_{\widetilde J}(r) \ge \widetilde X$ and $[\widetilde J,r] \le [\widetilde J,P] \le \widetilde K$. Since $\widetilde J/C_{\widetilde J}(r)\cong [\widetilde J, r]$ as $P$-modules, and $\widetilde J/C_{\widetilde J}(r)$ is centralized by $P$, we deduce from the fact that $\widetilde K$ is irreducible, that $[\widetilde J, r]=0$. Hence $[J,R]=[Y,R]=L$. Notice that $|J/K|\ge 2^{a+1}> \mathrm{End}_{P/O_2(P)}(\widetilde Y)$. Applying \cref{lem:End bound} to $P/R$ yields
 $C_{\widetilde J}(Q)$ has co-dimension $a$ in $\widetilde J$.  As $C_{\widetilde X}(Q)= \widetilde K$ has codimension $a$ in $\widetilde X$, we must have $\widetilde J= C_{\widetilde J}(Q)+\widetilde X$. Let $C$ be the preimage of $C_{\widetilde J}(Q)$ in $J$.  Then $$\dim C= \dim \widetilde K + \dim \widetilde J/\widetilde X +\dim L = (2n-2)a+\dim J/X=(2n-2)a+\dim Y/X +a+1.$$

  Let $T\le P$ have order $(q-1)^{n-1}$. Then $TQ/Q$ is contained in a subgroup of $P/Q$ isomorphic to a direct product of $n-1$ copies of  $\SL_2(q)$. In particular $T$ acts fixed-point-freely on $K/L$.  Since $C$ is $P$-invariant, $C$ is in particular $T$-invariant.  Set $D= C_C(T)$.  Then $$\dim D=\dim C-\dim \widetilde K= \dim Y/X+a +1$$ and $D>L$. Let $d \in D \setminus Y$. Then $d$ is centralized by an element $s$ of order $p$ (as every vector of $V\setminus Y$ is centralized by a conjugate of $\tau$). Since $s$ fixes $d+W=v+W$, $s\in P$.  Set $P^*= \langle s,T\rangle$.  Then, by \cite[Main Theorem]{GuralnickPentillaPraegerSaxl}, $P^* Q=P$. Since $P^*$ centralizes $L$ and $D/L$, $P^*$ centralizes $D$. As $[D,Q]\le L$, $Q$ also normalizes $D$ and so, as $P$ is perfect, $P$ centralizes $ D$.

It follows that there is an element $x \in V$ with $x+Y = v+Y$ such that $x$ is centralized  by $P$. Now pick an element $b \in K \setminus L$.  Consider $x+b$.  We know that $x+b$ is centralized by some element $\tau^*$ of order $p$ which is conjugate to $\tau$ (as every of $V\setminus Y$ is centralized by a conjugate of $\tau$).  Thus $x+b+Y= v+Y$ is centralized by $\tau^*$.  It follows as above  that $\tau^* \in P$.  Therefore $$ x+b = (x+b)\tau^*=x\tau^*+b\tau^*=x+b\tau^*$$ and so $b=b\tau^*$.  But $\tau^*$ acts fixed-point-freely $K/L$ and thus we have a contradiction. Hence $(G,H,V)$ is not immutable as asserted.
We have shown that $V$ has a unique non-trivial composition factor. Suppose that $[V,G]$ does not contain $C_V(G)$. Then for $v \in C_V(G)\setminus [V,G]$, we have $\gen{v} $ is a centralized direct summand of $V$. Indeed, let $U \ge C_{[V,G]}(G)$ be a complement to $\gen{v}$ in $C_V(G)$, then $U+[V,G]$ is a $G$-invariant complement to $\gen{v}$. Since there are no such centralized summands, we have $[V,G]\ge C_V(G)$.
Hence $[V,G]/C_V(G)$ is isomorphic to the natural module for $\Sp_{2n}(q)$. Now \cref{lem:symplectic fact} shows that $V$ is a section of $\mathcal W$ and   the possibilities for $(G,H,V)$ are then described in \cref{lem:sp2nq immutable struct}. This completes the proof.
  \end{proof}

\section{The proof of \cref{thm:Theorem B}: Schur covers}\label{sec:schur}

Finally, we consider those groups $G$ from the list of immutable triples in \cref{Main theorem,prop:trans cases} for which the quasisimple group $G$ has Schur multiplier of even order. In this situation, it is conceivable that there exists a faithful $\GF(2)G$-module $V$ on which a central involution $t \in Z(G)$ acts non-trivially, and for which $(G,H,V)$ is immutable for some maximal subgroup $H \le G$. Comparing the  groups listed in \cref{Main theorem,prop:trans cases}  with the groups with even order Schur multiplier, we are left to consider the groups $2^.\Alt(n)$ with $6\le n\le 9$, $2^.\SL_3(2)$, $2^.\Sp_6(2)$, $2^.\Omega_6^-(2)$, $2^.\Omega_8^+ (2)$ and $(2^2)^.\Omega_8^+ (2)$.

\begin{lemma}\label{lem:CVt} Suppose that $G$ is quasisimple, $Z(G)$ has order $2$ and $(G,H,V)$ is immutable. Then $Z(G)$ centralizes $V$.
\end{lemma}

\begin{proof}  Let $t\in Z(G)$ be the involution. Assume that the result is false and that $(G,H,V)$ is an immutable triple with $t$ acting non-trivially on $V$ and $\dim V$  minimum with this property. In particular,  $V$ has no centralized direct summand. Since $t$ does not centralize $V$, $$V>C_V(t)\ge [V,t]>0$$ and \cref{lem: reduct to irred} implies that both $(G,H,V/[V,t])$ and $(G, H,C_V(t))$ are immutable.

If  $  V/C_V(t)$ is not irreducible, then choose a submodule $V>W>C_V(t)$. By \cref{lem: reduct to irred}, $(G, H,W)$ is immutable. Since $t$ acts non-trivially on $W$, this contradicts the minimal choice of $V$. Hence $V/C_V(t)$ is an irreducible $\GF(2)G$-module. As     $[V,t]\cong V/C_V(t)$, $[V,t]$ is also irreducible.  Assume that $U \le C_V(t)$ is a $G$-submodule such that $U \cap [V,t]=0$. Then $t$ acts non-trivially on $V/U$ and so the minimality of $V$ yields $U=0$. Hence $[V,t]$ is the socle of $C_V(t)$.

We intend to determine the candidates for $G$ and $V$ as listed in \cref{tab:double covers}.
  Assume that $[V,t] \ne C_V(t)$. Then, as $C_V(t)$ has socle $[V,t]$, the results of \cref{sec:ThmBSymp,sec:ThmB1} together yield the first two lines and final two lines of \cref{tab:double covers}. So assume that $C_V(t)=[V,t]$. Then the  potential candidates for $V$ are obtained from \cref{Main theorem,prop:trans cases} after remembering that $V/C_V(t)\cong [V,t]$ as $\GF(2)G$-module, we obtain the remaining lines of \cref{tab:double covers}.

\begin{table}[h]
\begin{tabular}{cccc}\hline
Group&$\dim V/C_V(t)=\dim [V,t]$&$\dim C_V(t)$&$\dim V$\\\hline
$2^.\Alt(6)$&$1$&$5$&$6$\\
$2^.\Sp_6(2)$&$1$&$7$&$8$\\
$2^.\Alt(6)$&$4$&$4$&$8$\\
$2^.\Alt(7)$&$4 $&$4$&$8$\\
$2^.\Alt(8)$&$4 $&$4$&$8$\\
$2^.\Alt(9)$&$8 $&$8$&$16$\\
$2^.\SL_3(2)$&$3$&$3$&$6$\\
$2^.\Sp_6(2)$&$6$&$6$&$12$\\
$2^.\Sp_6(2)$&8&8&16\\
 $2^.\Omega_6^-(2)$&$6$&$6$&$12$\\
$2^.\Omega_8^+ (2)$&$8$&$8$&$16$\\
$2^.\Alt(6)$&$4$&$5$&$9$\\
$2^.\Sp_6(2)$&$6$&$7$&$13$\\
\hline
\end{tabular}\vspace{3mm}\caption{Double covers and candidates for $V$.}\label{tab:double covers}
\end{table}

Suppose first that $C_V(t)\ne [V,t]$. Then $G$ is either $2^.\Alt(6)$ or $2^.\Sp_6(2)$. We employ the {\sc Magma} command {\sc LowDimensionalModules} and see that there are no faithful modules of dimension less then 14 for the $2^.\Sp_6(2)$ and that  $2^.\Alt(6)$ has   two  faithful $9$-dimensional modules as detailed in the second last row of \cref{tab:double covers}. A quick {\sc Magma} check shows that these examples do not result in immutable triples.

Thus we may assume that $C_V(t)=[V,t]$. Here we check the cohomology condition from \cref{lem:some cohomology}. Thus we set $X= V/C_V(t)$ and calculate using {\sc Magma} that $\dim \mathrm H^1(G,(X\otimes X^*)/C_{X\otimes X^*}(G))=0$,  and conclude that there are in fact no faithful modules $V$. This completes the proof.
\end{proof}

We remark that in \cref{lem:CVt}, \cref{lem:some cohomology} is only required in the case that $G/Z(G)\cong \Omega^+_8(2)$ as in this case the {\sc Magma} command {\sc LowDimensionalModules} fails to complete in a timely way.

\begin{lemma}\label{ZOdd}
Suppose that $(G,H,V)$ is immutable with $V$ a faithful $\GF(2)G$-module. Then $G \not \cong (2^2)^.\Omega_{8}^+(2)$.
\end{lemma}

\begin{proof} Choose a counter example with $\dim V$ minimal.
Let $t \in Z(G)^\#$. Then $V/[V,t]$ is a module for $G/\gen{t}$.  Since $(G ,H ,V/[V,t])$ is immutable by \cref{lem: reduct to irred}, \cref{lem:CVt} yields $[V,t]=[V,Z(G)]$. Similarly $C_V(t)=C_V(Z(G))$. As in \cref{lem:CVt}, we may assume that $[V,t]$ is irreducible and $C_V(t)$ has socle $[V,t]$.  Since $C_V(t)$ has no centralized summand, we deduce that $C_V(t)$ is irreducible.  Hence $[V,t]=C_V(t)=C_V(Z(G))=[V,Z(G)]$. Since $G$ is perfect, it follows that $\dim [V,t]=8$ by \cref{Main theorem}. In particular, $\dim V=16$. Now commutation with the elements of $Z(G)$ determines a set of four linear maps in $\mathrm{Hom}_G(V/C_V(t),[V,t])\cong \mathrm{End}_G(V/C_V(t)) \cong \GF(2)$, a contradiction.
\end{proof}

We finally prove \cref{thm:Theorem B}.

\begin{proof}[The proof of \cref{thm:Theorem B}] Assume that $G$ is quasisimple, $H$ is a maximal subgroup of $G$, $V$ is a $\GF(2)G$-module with no centralized direct summand and that $V$ is not irreducible.  The possibilities for $(G/C_G(W),H/C_G(W),W)$ where $W$ is a faithful composition factor for $G$ are given by \cref{Main theorem,prop:trans cases}. \cref{lem:CVt,ZOdd} together with the opening discussion in \cref{sec:schur} show that $Z(G)$ has odd order.

Assume that $G$ is transitive on $W^\#$. If $G$ is as in \cref{prop:trans cases} (i), (ii), (iii) or (iv), then \cref{lem:A8-A7,lem:not linear} yield case (iii) of \cref{thm:Theorem B}. The symplectic possibilities listed in \cref{prop:trans cases} (v) and (vi) and \cref{Main theorem} (iv)  in \cref{prop:trans cases} are the subject of \cref{lem:spin yields irreducible,prop:symplecticGroups} and they give \cref{thm:Theorem B} (i) and (ii). The possibilities listed in \cref{prop:trans cases} (vii), (viii) and (ix) are investigated in \cref{lem:alt(6) done,Alt(7) gone} and yield the restrictions of the $\Omega_5(2)$-modules to $G \cong \Alt(6)$ and its dual as listed in \cref{thm:Theorem B} (iv).

\cref{lem:orthgone}, shows that  $G$ cannot come from   \cref{Main theorem} (i), (ii) or (iii). \cref{Main theorem} (iv) does not extend to an reducible example by \cref{lem:spin yields irreducible}.  Finally, \cref{Main theorem} (v) and (vi) give no reducible examples by \cref{lem:A9 done,lem:Omega82-A9 done}.
\end{proof}
\printbibliography

 \end{document}